\documentclass[reqno,11pt]{amsart}

\usepackage{amssymb,amsfonts,amsmath,amsthm,color,hyperref}
\usepackage[all,arc]{xy}
\usepackage{enumerate}
\usepackage{mathrsfs}
\usepackage[toc,page]{appendix}
\usepackage[left=3cm, right=3cm, bottom=3cm]{geometry}
\usepackage{tabularx}
\usepackage{caption}
\usepackage{color}
\usepackage{tikz-cd}
\usepackage{comment}
\usepackage{enumitem}

\newtheorem{thm}{Theorem}[section]
\newtheorem*{thm*}{Theorem}
\newtheorem{cor}[thm]{Corollary}
\newtheorem{prop}[thm]{Proposition}
\newtheorem*{prop*}{Proposition}
\newtheorem{lem}[thm]{Lemma}

\newtheorem*{ntn*}{Notation}

\theoremstyle{definition}
\newtheorem{defn}[thm]{Definition}

\newtheorem{exmp}[thm]{Example}

\theoremstyle{remark}
\newtheorem{rem}[thm]{Remark}

\newcommand{\parhdeg}{\operatorname{parhdeg}}
\newcommand{\Pic}{\operatorname{Pic}}
\newcommand{\Z}{\mathbb{Z}}

\makeatletter
\let\c@equation\c@thm
\makeatother
\numberwithin{thm}{section}
\numberwithin{equation}{section}

\title{Level structures on parahoric torsors and complete integrability}

\author{Georgios Kydonakis and Lutian Zhao}

\date{14 June 2025
}

\begin{document}

\begin{abstract}
For a smooth complex algebraic curve $X$ and a reduced effective divisor $D$ on $X$, we introduce a notion of $D$-level structure on parahoric $\mathcal{G}_{\boldsymbol \theta}$-torsors over $X$, for any connected complex reductive Lie group $G$. A moduli space of parahoric $\mathcal{G}_{\boldsymbol \theta}$-torsors equipped with a $D$-level structure is constructed and we identify a canonical moment map with respect to the action of a level group on this moduli space. This action extends to a Poisson action on the cotangent, thus inducing a Poisson structure on the moduli space of logahoric $\mathcal{G}_{\boldsymbol \theta}$-Higgs torsors on $X$. A study of the generic fibers of the parahoric Hitchin fibration of this moduli space identifies them as abelian torsors and introduces new algebraically completely integrable Hamiltonian Hitchin systems in this parahoric setting. We show that this framework generalizes, among other, the integrable system of Beauville and recovers the classical Gaudin model in its simplest form, the space of periodic KP elliptic solitons and the elliptic Calogero--Moser system, thus demonstrating that the logahoric Hitchin integrable system unifies many integrable systems with regular singularities under a single geometric framework. 
\end{abstract}
\maketitle
\section{Introduction}
Algebraically completely integrable Hamiltonian systems on Poisson manifolds are characterized by the properties that the general level sets of the momentum map are isomorphic to an affine part of an abelian variety and the flows of the integrable vector fields are linearized by this isomorphism (\cite{DoMa}, \cite{Vanhaecke}). Many classical systems, such as the geodesic flow on an ellipsoid, the Korteweg--de Vries equation and its extensions, the elliptic Calogero--Moser system, various Euler--Arnold systems, or the Neumann system of evolution of a point on the sphere subject to a quadratic potential, can embed as symplectic leaves of certain spaces (\cite{Beauville}, \cite{HM}, \cite{Markman}, \cite{RS}).  

Hitchin introduced in \cite{Hitchin} an algebraically completely integrable Hamiltonian system of Poisson-commuting functions on the cotangent bundle of the moduli space of stable vector bundles on a smooth compact Riemann surface. This cotangent bundle parameterizes on the one hand Higgs bundles (vector bundles over the Riemann surface with a canonically twisted endomorphism) and on the other hand spectral data (generically line bundles on branched covers of the Riemann surface). Hitchin's system is linearized on Jacobians of spectral curves and it is the lowest-rank symplectic leaf of a natural infinite-dimensional Poisson variety. In an extension to Hitchin's result, Markman \cite{Markman} studied families of Jacobians of spectral curves over algebraic curves of any genus obtained by twisting by any sufficiently positive line bundle. Markman proposed a geometric method via a deformation-theoretic construction, in which the Poisson structure on an open subset of the algebraically completely integrable Hamiltonian system is obtained via symplectic reduction from the cotangent of the moduli space $\mathcal{U}_{\Sigma}(r,d,D)$ of stable rank $r$ degree $d$ vector bundles $E$ over a Riemann surface $\Sigma$ equipped with a $D$-level structure, namely, an isomorphism $E\vert _D \to \bigoplus_{i=1}^{r}\mathcal{O}_D$, for a fixed reduced effective divisor $D$ on $\Sigma$. The concept of a \emph{$D$-level structure} was introduced earlier by Seshadri in \cite{Seshadri}, who also constructed a moduli space of stable vector bundles equipped with such structure as a smooth quasi-projective variety. The same result to \cite{Markman} for the moduli space of twisted stable Higgs pairs was obtained independently by Bottacin in \cite{Bottacin}, who produced an explicit antisymmetric contravariant 2-tensor at the stable points. Using explicit computations, he proved that this defines a Poisson structure and checked the linearity of the flows. 

The aim in this article is to introduce algebraically completely integrable Hamiltonian Hitchin systems for general complex reductive structure groups $G$, as a way to provide a mathematical framework for describing the geometry of certain integrable systems in Physics with regular singularities. We do this by generalizing Markman's approach, that is, by studying the symplectic leaves of the $G$-Hitchin system over a complex algebraic curve with a divisor of finitely many distinct points. Various similar treatments have appeared in the literature with limitations on the parabolic structure and the framing imposed locally on the principal bundles (see, for instance, \cite{BKV}, \cite{BisLogPeon}, \cite{LogMar}, \cite{Zelaci}). 

In this article, we use the language of parahoric $\mathcal{G}_{\boldsymbol \theta}$-torsors and introduce a new notion of $D$-level structure that both generalizes Seshadri's concept and interacts coherently with the parahoric Bruhat--Tits data. This provides a more transparent framework for studying the Hitchin system and its symplectic leaves for general reductive groups. We next describe our results in more detail. 

\vspace{2mm}

Let $X$ be a smooth complex projective curve, $D$ a reduced effective divisor on $X$, and $G$ a connected complex reductive Lie group. For each point $x \in D$, fix a collection of weights $\boldsymbol \theta :=\{\theta_x\}_{x \in D}$, that is, a collection of points in the Bruhat--Tits apartment of $G\left(\mathbb{C}((t))\right)$; each of these points determines a parahoric subgroup scheme $\mathcal{G}_{\theta_x}$ on the formal disk $\mathbb{D}_x :=\mathrm{Spec}\widehat{\mathcal{O}}_{X,x}$. We introduce the following:

\begin{defn}[Definition \ref{defn:D_level_structure}] A \emph{$D$-level structure} of parahoric type $(\theta_x)_{x\in D}$ on a parahoric $\mathcal{G}_{\boldsymbol\theta}$-torsor $\mathcal{E}$ over $X$ is a choice of section $\eta_x:\mathbb{D}_x \to G_{\theta_x}/G_{\theta_x}^+$, for each $x\in D$, where $G_{\theta_{x}}^+$ denotes the pro-unipotent radical of $G_{\theta_{x}}$. We will denote a parahoric $\mathcal{G}_{\boldsymbol\theta}$-torsor equipped with a $D$-level structure as a pair $(\mathcal{E}, \eta)$.
\end{defn}

Note that the data of a $D$-level structure  refines the parahoric torsor by imposing a ``flag-type'' or ``framing-type'' condition at each point in $D$. We call a $D$-level structure on a parahoric $\mathcal{G}_{\boldsymbol \theta}$-torsor \emph{stable} (resp. \emph{semistable}) if the underlying parahoric $\mathcal{G}_{\boldsymbol \theta}$-torsor is $R$-stable (resp. $R$-semistable) in the sense introduced in \cite{KSZparh}. A moduli space $\mathcal{U}(X, \mathcal{G}_{\boldsymbol \theta})$ of stable pairs $(\mathcal{E}, \eta)$ of parahoric $\mathcal{G}_{\boldsymbol \theta}$-torsors over $X$ with a $D$-level structure is then constructed as an irreducible normal projective variety:

\begin{thm}[Theorem \ref{thm:moduli-d-level}]
Let \(X\) be a smooth complex projective curve and \(D\subset X\) a reduced effective divisor. Let \(\mathcal{G}_{\boldsymbol \theta}\) be the parahoric Bruhat--Tits group scheme on \(X\) corresponding to a collection \(\boldsymbol \theta :=\{\theta_x\}_{x\in D}\) of rational weights. Then the moduli functor which assigns to any scheme \(S\) the set of \(S\)-equivalent classes of semistable parahoric \(\mathcal{G}_{\boldsymbol \theta}\)-torsors on \(X\) with fixed \(D\)-level structure is corepresented by an irreducible, normal, projective variety.
\end{thm}

We analyze the deformations of the $D$-level structure to show that the tangent space of the moduli space $\mathcal{U}(X, \mathcal{G}_{\boldsymbol \theta})$ corresponds to the space of logarithmic Higgs fields $H^0(X, \mathcal{E}(\mathfrak{g})\otimes K(D))$ (Proposition \ref{prop:tangent_U}). 

In order to identify a canonical moment map on the cotangent $T^*\mathcal{U}(X, \mathcal{G}_{\boldsymbol \theta})$, we consider the action of a certain \emph{level group} $G_D$ on the moduli space $\mathcal{U}(X, \mathcal{G}_{\boldsymbol \theta})$ (Definition \ref{defn:level_group}). This action is free on the regularly stable locus and, in fact, extends to a Poisson action on the cotangent $T^*\mathcal{U}(X, \mathcal{G}_{\boldsymbol \theta})$ (see Sections \ref{sec:Parahoric} and \ref{sec:Poisson_moment} for further notational explanations):

\begin{thm}[Theorem \ref{thm:PoissonGD}]
The group \(G_{D}\) acts \emph{Poisson} on \(T^*\mathcal{U}(X,\mathcal{G}_{\boldsymbol \theta})\).  Moreover, the \emph{canonical} moment map
\[
  \mu \colon T^*\mathcal{U}(X,\mathcal{G}_{\boldsymbol \theta})\;\longrightarrow\;\mathfrak{g}_D^*
\]
is given by dualizing the infinitesimal action and can be explicitly described via coresidues at the divisor \(D\). Its image is the element of $\mathfrak{g}_D^* = \bigoplus_{j=1}^s \hat{\mathfrak{l}}_{\theta_j}^*$ given by the direct sum of the coresidues at each point $x_j \in D$:
\[
\mu([(\mathcal{E},\eta)],\varphi) = \bigoplus_{j=1}^s \operatorname{CoRes}_{x_j}(\varphi).
\]
Explicitly, for any element $Y = (Y_1, \dots, Y_s) \in \mathfrak{g}_D = \bigoplus_{j=1}^s \hat{\mathfrak{l}}_{\theta_j}$, the pairing is given by the sum of Killing form pairings over the divisor $D$:
\[
\langle \mu([(\mathcal{E},\eta)],\varphi), Y \rangle = \sum_{j=1}^s (\operatorname{Res}_{x_j}(\varphi), Y_j).
\]
\end{thm}
Similarly to Markman's original work, we deduce a Poisson structure on the moduli space $\mathcal{M}_H(X, \mathcal{G}_{\boldsymbol \theta})$ of logahoric $\mathcal{G}_{\boldsymbol \theta}$-Higgs torsors via a forgetful morphism $$l: \mathcal{M}_{LH}(X, \mathcal{G}_{\boldsymbol \theta}) \to \mathcal{M}_H(X, \mathcal{G}_{\boldsymbol \theta}),$$ from the coarse moduli space $ \mathcal{M}_{LH}(X, \mathcal{G}_{\boldsymbol \theta})$ of triples $(\mathcal{E}, \varphi, \eta)$ of logahoric $\mathcal{G}_{\boldsymbol \theta}$-Higgs torsors with $D$-level structure to the moduli space of pairs $(\mathcal{E}, \varphi)$. The cotangent \(T^*\mathcal{U}(X, \mathcal{G}_{\boldsymbol \theta})\) of the moduli space \(\mathcal{U}(X, \mathcal{G}_{\boldsymbol \theta})\) is an open subset of \(M_{LH}(X, \mathcal{G}_{\boldsymbol \theta})\) and the forgetful map $l$ is Poisson, thus inheriting with a Poisson structure the space $\mathcal{M}_H(X, \mathcal{G}_{\boldsymbol \theta})$ (Theorem \ref{thm:forgetful_Poisson}).

We then use the cameral covers of Donagi \cite{Donagi} to study the parahoric Hitchin fibration on the space $\mathcal{M}_H(X, \mathcal{G}_{\boldsymbol \theta})$. We prove that its generic fibers are Lagrangian subvarieties with respect to the symplectic leaves of the Poisson moduli space $\mathcal{M}_{H}(X,\mathcal{G}_{\boldsymbol{\theta}})$, thus concluding to our main result:

\begin{thm}[Theorem \ref{thm:main}]\label{thm:main at Intro} 
Let $X$ be a smooth complex algebraic curve and $D$ be a reduced effective divisor on $X$. Let $G$ be a connected complex reductive group. The moduli space $\mathcal{M}_H(X,\mathcal{G}_{\boldsymbol\theta})$ of logahoric $\mathcal{G}_{\boldsymbol \theta}$-Higgs torsors over $X$ is Poisson and is fibered via a map $h_{\boldsymbol \theta}: \mathcal{M}_H(X,\mathcal{G}_{\boldsymbol\theta}) \to \mathcal{A}_{\boldsymbol \theta}$ by abelian torsors. Moreover, $h_{\boldsymbol \theta}: \mathcal{M}_H(X,\mathcal{G}_{\boldsymbol\theta}) \to \mathcal{A}_{\boldsymbol \theta}$ is an algebraically completely integrable Hamiltonian system.
\end{thm}

We will call the algebraically completely integrable Hamiltonian system of Theorem \ref{thm:main at Intro}, the \emph{logahoric Hitchin integrable system}. Considering the image $\mathcal{A}_{\boldsymbol \theta}^+$ under $h_{\boldsymbol \theta}$ of the strongly logahoric $\mathcal{G}_{\boldsymbol \theta}$-Higgs torsors, we show (Theorem \ref{thm:commuting}) that there exists an isomorphism of affine varieties 
\[
  \mathcal{A}_{\boldsymbol\theta}/\mathcal{A}_{\boldsymbol\theta}^+ \cong \mathfrak{g}_D^* // G_D,
\]
for the Hitchin fibration ${\widetilde H_{\boldsymbol\theta}}\colon \mathcal{M}_{LH}(X, \mathcal{G}_{\boldsymbol{\theta}}) \to \mathcal{A}_{\boldsymbol \theta}$ on the space of leveled logahoric $\mathcal{G}_{\boldsymbol \theta}$-Higgs torsors. This allows us to describe the symplectic leaf foliation explicitly:

\begin{thm}[Corollary \ref{cor:parahoric-leaves-refine}]
The foliation of the smooth locus
\(\mathcal{M}_H(X,\mathcal{G}_{\boldsymbol\theta})^{\mathrm{sm}}\) by its symplectic
leaves refines the foliation by fibers of the map
\[
  q\circ\widetilde H_{\boldsymbol\theta}
  \;:\;
  \mathcal{M}_H(X,\mathcal{G}_{\boldsymbol\theta})^{\mathrm{sm}}
  \;\longrightarrow\;
  \mathcal{A}_{\boldsymbol\theta}/\mathcal{A}_{\boldsymbol\theta}^+,
\]
for the quotient map $q: \mathcal{A}_{\boldsymbol\theta} \to  \mathcal{A}_{\boldsymbol\theta}/\mathcal{A}_{\boldsymbol\theta}^+$.
Moreover, each fiber of \(q\circ\widetilde H_{\boldsymbol\theta}\) contains
a unique symplectic leaf of maximal dimension.
\end{thm}

In the last part of the article, we demonstrate examples of algebraically completely integrable Hamiltonian systems that can be analyzed as logahoric Hitchin integrable systems. Namely, apart from the integrable systems studied, for instance, in \cite{Beauville}, \cite{Markman}, \cite{RS} that can be studied using this approach in the absence of a parahoric structure, we show that over the Riemann sphere $X=\mathbb{P}^1$ the logahoric Hitchin integrable system generalizes the integrable system of Beauville \cite{Beauville}. Furthermore, still for $X=\mathbb{P}^1$ one can recover the classical Gaudin model for a specific choice of parahoric data corresponding to the simplest pole structure. For an elliptic base curve $X=\Sigma$, and for group $G=\mathrm{SL}(r, \mathbb{C})$, we recover the KP hierarchy when taking the subvariety of $\mathcal{M}_{LH}(\Sigma, \mathcal{G}_{\boldsymbol{\theta}})$ consisting of triples $(\mathcal{E},\eta,\varphi)$, where the coresidue of the Higgs field is the coadjoint orbit of the element in $\mathfrak{sl}(r)^*$ corresponding (via the Killing form) to the matrix $\mathrm{diag}(-1, -1, \dots, -1, r-1)$. Another example in this elliptic case that can be studied using the logahoric Hitchin integrable system is the elliptic Calogero--Moser system
as demonstrated by Hurtubise and Markman in \cite{HM}. We note, lastly, that the framework built here, enables the study of the geometry of integrable systems for higher genus and/or non-trivial orbits than the ones exhibited here.

\section{Parahoric Torsors}\label{sec:Parahoric}
Let $G$ be a connected complex reductive Lie group, and fix a maximal torus $T$ in $G$ with corresponding Lie algebras $\mathfrak{t}$ and $\mathfrak{g}$. We consider the character group $X(T) := \text{Hom}(T, \mathbb{G}_m)$ and the co-character group $Y(T) := \text{Hom}(\mathbb{G}_m, T)$, which is also the group of one-parameter subgroups of $T$. The canonical pairing $\langle \cdot, \cdot \rangle: Y(T) \times X(T) \rightarrow \mathbb{Z}$ can be extended to $\mathbb{R}$ by tensoring $Y(T)$ and $X(T)$ with $\mathbb{R}$. We refer to co-characters with coefficients in $\mathbb{R}$ and $\mathbb{Q}$ as \emph{real weights} and \emph{rational weights}, respectively. In general, by a \emph{weight} we will always consider a real weight unless stated otherwise.

Denote the root system with respect to the maximal torus $T$ as $\mathcal{R}$, and let $\mathcal{R}_+ \subseteq \mathcal{R}$ be the set of positive roots. For a root $r \in \mathcal{R}$, we have an isomorphism of Lie algebras ${\rm Lie}(\mathbb{G}_a) \rightarrow ({\rm Lie}(G))_r$, which induces a natural homomorphism of groups $u_r: \mathbb{G}_a \rightarrow G$ such that $t u_r(a) t^{-1}= u_r(r(t)a)$, for $t \in T$ and $a \in \mathbb{G}_a$. We denote the image of the homomorphism $u_r$ by $U_r$, which is a closed subgroup of $G$. The reductive group $G$ is generated by its subgroups $T$ and $U_r$, for $r \in \mathcal{R}$. Namely, an element $g$ in $G$ can be written as a product $g=g_t \prod_{r \in \mathcal{R}} g_r$, where $g_t \in T$ and $g_r \in U_r$. Sometimes, we write $g$ as a tuple $(g_t,g_r)_{r \in \mathcal{R}}$ for convenience.

\subsection{Parahoric subgroups}
Given a weight $\theta$, we can consider it as an element in $\mathfrak{t}$, the Lie algebra of $T$, under differentiation. We define the integer $m_r(\theta) := \lceil -r(\theta) \rceil$, where $\lceil \cdot \rceil$ is the ceiling function and $r(\theta) := \langle \theta, r \rangle$. We introduce the following definition:

\begin{defn}\label{defn parahoric_group}
Let $R:=\mathbb{C}[[z]]$ and $K:=\mathbb{C}((z))$. With respect to the above data, we define the \emph{parahoric subgroup $G_{\theta}(K)$} of $G(K)$ as
\begin{align*}
G_{\theta}(K):=\langle T(R), U_r(z^{m_r(\theta)}R), r
\in \mathcal{R} \rangle.
\end{align*}
Denote by $\mathcal{G}_{\theta}$ the corresponding group scheme of $G_{\theta}(K)$, which is called the \emph{parahoric group scheme}.
\end{defn}

The parahoric subgroup of $G(K)$ determined by $\theta$ can be alternatively defined as
\begin{align*}
	G'_{\theta}(K):=\{g(z) \in G(K) \text{ }|\text{ } z^{\theta} g(z) z^{-\theta} \text{ has a limit as $z \rightarrow 0$}\},
\end{align*}
where $z^{\theta} := e^{\theta \ln z}$. This definition is of a more analytic nature; the equivalence of these two definitions can be found in our previous paper \cite[Lemma 2.2]{HKSZ}. We will thus utilize either definition of parahoric subgroups, $G_{\theta}(K)$ or $G'_\theta(K)$, in the study of connected complex reductive groups and their representations.

\subsection{Parahoric Torsors}\label{sec:parah_torsors}

For $G$ as introduced earlier, let $X$ be a smooth projective curve over $\mathbb{C}$, and let 
\[
D=\{x_1,\dots ,x_s\}\subset X
\] 
be a reduced effective divisor.  Denote the complement by $X_D:=X\setminus D$.  
For every $x\in D$, fix a \emph{weight} $\theta_x$, 
namely, a point of the Bruhat--Tits apartment of $G(\mathbb{C}((t)))$.  This point determines a facet of the Bruhat--Tits apartment and hence a parahoric subgroup scheme $G_{\theta_x}$ on the formal disc $\mathbb{D}_x:=\operatorname{Spec}\widehat{\mathcal{O}_{X,x}}$.
$\theta_x$. Let us also write $\boldsymbol\theta:=\{\theta_x\}_{x\in D}$, for a collection of weights for each point in $D$.

Following \cite[Section 2]{BS} we now glue together
\[
\mathcal{G}_{\boldsymbol\theta}\bigl|_{X_D}=G\times X_D,
\qquad 
\mathcal{G}_{\boldsymbol\theta}\bigl|_{\mathbb{D}_x}=G_{\theta_x}\;(x\in D)
\]
to obtain the \emph{parahoric Bruhat--Tits group scheme}
\[
\mathcal{G}_{\boldsymbol\theta}\longrightarrow X.
\]
This scheme is smooth, affine, of finite type, and flat over $X$ \cite[Lemma 3.18]{ChGP}.

\begin{defn}[Generically split torsor {\cite[Section 3]{BS}}]\label{def:gen_split}
A $\mathcal{G}_{\boldsymbol\theta}$-torsor $\mathcal{E}\to X$ is called \emph{generically split} if its restriction to the generic point $q=\operatorname{Spec}\mathbb{C}(X)$ becomes trivial:
\[
\mathcal{E}\bigl|_{q}\;\cong\;\mathcal{G}_{\boldsymbol\theta}\bigl|_{q}\;\simeq\;G\times q .
\]
\end{defn}

\begin{rem}\label{rem:abuse}
Except when we explicitly state otherwise (see Example~\ref{parahoricexmp}), we shall abuse terminology throughout the article and call a generically split $\mathcal{G}_{\boldsymbol\theta}$-torsor simply a \emph{parahoric $\mathcal{G}_{\boldsymbol\theta}$-torsor}.
\end{rem}

\paragraph{\textbf{Local description and gluing.}}
Fix a parahoric $\mathcal{G}_{\boldsymbol \theta}$-torsor $\mathcal{E}$ on $X_D$ and, for every $x\in D$, fix a $G_{\theta_x}(K)$-torsor $\mathcal{E}_x$ on the punctured disc $\mathbb{D}_x^\times:=\operatorname{Spec} K$, where $K=\mathbb{C}((t))$.  
A parahoric $\mathcal{G}_{\boldsymbol\theta}$-torsor $\mathcal{E}$ on $X$ is given by gluing $\mathcal{E}$ and the $\{\mathcal{E}_x\}$ through isomorphisms
\begin{equation}\label{eq:glue}\tag{2.1}
\Theta_x:
\mathcal{E}_x\bigl|_{\mathbb{D}_x^\times}\;\longrightarrow\;
\mathcal{E}\bigl|_{\mathbb{D}_x^\times} ,
\end{equation}
one for each $x\in D$.  
Each $\Theta_x$ corresponds to a loop $g_x(t)\in G(K)$.

\smallskip
Two sets of data $(\mathcal{E},\{\mathcal{E}_x\},\{\Theta_x\})$ and $(\mathcal{E}',\{\mathcal{E}_x'\},\{\Theta_x'\})$ determine \emph{isomorphic} parahoric $\mathcal{G}_{\boldsymbol\theta}$-torsors provided there are isomorphisms
for every $x \in D$
\[
\psi:\mathcal{E}\rightarrow \mathcal{E}' \text{ and }
\psi_x:\mathcal{E}_x\rightarrow \mathcal{E}_x',
\]
such that $\psi\circ\Theta_x=\Theta_x'\circ\psi_x$ for every $x\in D$:
\[
\begin{tikzcd}
\mathcal{E}_x \arrow[d,"\Theta_x"'] \arrow[r,"\psi_x"] & \mathcal{E}_x' \arrow[d,"\Theta_x'"]\\
\mathcal{E} \arrow[r,"\psi"'] & \mathcal{E}'.
\end{tikzcd}
\]

\begin{rem}[Beauville--Laszlo gluing for generically split torsors {\cite{BL,BS}}]\label{prop:BL}
The Beauville--Laszlo gluing construction provides a method to construct global torsors on the curve $X$. In the specific setting described above, consider the case of generically split torsors. If we begin with the \emph{trivial} $\mathcal{G}_{\boldsymbol\theta}$-torsor on $X_D = X \setminus D$ and the \emph{trivial} $G_{\theta_x}(K)$-torsor on each punctured disc $\mathbb{D}_x^\times$ (for $x \in D$), then a given collection of gluing isomorphisms $\Theta_x$ (which correspond to loops $g_x(t) \in G(K)$) defines a parahoric $\mathcal{G}_{\boldsymbol\theta}$-torsor $\mathcal{E}$ on $X$. The existence and uniqueness of this resulting torsor $\mathcal{E}$ are ensured by the general principles of torsor gluing, as detailed in \cite[Remark 5.2.3]{BS}. This result for torsors is a consequence of the more general Beauville--Laszlo framework \cite{BL}, and is analogous to the gluing of group schemes themselves as described in \cite[Lemma 5.2.2]{BS}. Thus, for generically split data, the gluing process is well-defined and yields a unique torsor on $X$.
\end{rem}

\begin{rem}[Uniformization for generically split torsors]\label{rem:unif_gensplit}
Continuing from Remark~\ref{prop:BL}, the isomorphism classes of generically split parahoric \(\mathcal{G}_{\boldsymbol\theta}\)-torsors, constructed via the Beauville--Laszlo gluing of trivial torsors, can be classified using the uniformization theorem of Heinloth \cite{Heinloth}. A generically split \(\mathcal{G}_{\boldsymbol\theta}\)-torsor \(\mathcal{E} \to X\) (as per Definition~\ref{def:gen_split} and Remark~\ref{rem:abuse}) is trivial at the generic point \( q = \operatorname{Spec} \mathbb{C}(X) \), i.e., \(\mathcal{E}|_q \cong G \times q\). Locally, \(\mathcal{E}\) is described via the Beauville--Laszlo gluing construction: it is obtained by gluing a \( G \)-torsor \(\mathcal{E}|_{X_D}\) on \( X_D \) with \(\mathcal{G}_{\theta_x}\)-torsors \(\mathcal{E}_x\) on each \(\mathbb{D}_x\), using isomorphisms
\[
\Theta_x: \mathcal{E}_x|_{\mathbb{D}_x^\times} \to \mathcal{E}|_{\mathbb{D}_x^\times},
\]
where \(\mathbb{D}_x^\times = \operatorname{Spec} \mathbb{C}((t))\), and each \(\Theta_x\) corresponds to a gluing element \( g_x(t) \in G(K) \). Since \(\mathcal{E}\) is generically split, we may assume (after choosing a trivialization) that \(\mathcal{E}|_{X_D} \cong G \times X_D\), making \(\mathcal{E}|_{\mathbb{D}_x^\times} \cong G \times \mathbb{D}_x^\times\), so \( g_x(t) \) defines the transition from the trivial torsor on \(\mathbb{D}_x\).

Two sets of gluing data \( g_\bullet = \{ g_x \in G(K) \}_{x \in D} \) and \( g_\bullet' = \{ g_x' \in G(K) \}_{x \in D} \) define isomorphic \(\mathcal{G}_{\boldsymbol\theta}\)-torsors if there exist isomorphisms \(\psi: \mathcal{E}|_{X_D} \to \mathcal{E}'|_{X_D}\) and \(\psi_x: \mathcal{E}_x \to \mathcal{E}_x'\) satisfying \(\psi \circ \Theta_x = \Theta_x' \circ \psi_x\), for all \( x \in D \). In the trivial case, \(\psi\) corresponds to an element \( h \in \mathcal{G}_{\boldsymbol\theta}(X_D) = G(X_D) \), and each \(\psi_x\) corresponds to \( k_x \in \mathcal{G}_{\theta_x}(\mathbb{D}_x) = G_{\theta_x}(K) \). The compatibility condition translates to:
\[
g_x' = h|_{\mathbb{D}_x^\times} \, g_x \, k_x, \quad \text{for all } x \in D,
\]
where \( h|_{\mathbb{D}_x^\times} \) is the restriction of \( h \in G(X_D) \) to \(\mathbb{D}_x^\times\), an element of \( G(K) \), and \( k_x \in G_{\theta_x}(K) \) acts on the right.

Thus, the isomorphism classes of generically split \(\mathcal{G}_{\boldsymbol\theta}\)-torsors are parameterized by the quotient
\[
\mathrm{Bun}_{\mathcal{G}_{\boldsymbol\theta}}\cong G(X_D) \backslash \prod_{x \in D} \left( G(K) / G_{\theta_x}(K) \right),
\]
where:

\begin{itemize}
\item \( G(K) / G_{\theta_x}(K) \) is the affine Grassmannian (or partial affine flag variety) at \( x \), denoted by \(\mathrm{Gr}_{\mathcal{G}, x}\),
\item \( G(X_D) \) acts on the left via the natural maps \( G(X_D) \to G(K)\), for each \( x \in D \),
\item The product \(\prod_{x \in D} \mathrm{Gr}_{\mathcal{G}, x}\) is the multi-point affine Grassmannian \(\mathrm{Gr}_{\mathcal{G}, D}\).
\end{itemize}

This classification is a direct application of the uniformization theorem for the stack \(\mathrm{Bun}_{\mathcal{G}_{\boldsymbol\theta}}\) of \(\mathcal{G}_{\boldsymbol\theta}\)-torsors by Heinloth \cite{Heinloth} and is explicitly stated by Damiolini and Hong in \cite[Section 2.1]{DH1}.
\end{rem}

\begin{rem}[Beyond the generically split case]\label{rem:DH}
For parahoric $\mathcal{G}_{\boldsymbol\theta}$-torsors that are not necessarily generically split, Damiolini and Hong in \cite[Theorem 6.2.2]{DH2} present a construction for general parahoric Bruhat--Tits group schemes, which also applies to their torsors, accommodating non-generically split cases. They establish that any such parahoric Bruhat--Tits group scheme can be realized from a $(\Gamma,G)$-bundle associated with a finite, tamely ramified cover $\widetilde X \to X$. Within this framework, the gluing procedure and the classification of isomorphism classes are addressed through a $\Gamma$-equivariant adaptation of the uniformization principle. The structure of these torsors, especially their local aspects at points $x \in D$ (which may correspond to ramified points in the cover), is then classified using non-abelian group cohomology, typically involving classes in $H^{1}(\Gamma_x, G(K_x))$, where $\Gamma_x$ represents the local Galois group (or a relevant inertia subgroup) at $x$.
\end{rem}

\begin{exmp}\label{parahoricexmp}
\leavevmode
We now list a few fundamental examples of parahoric $\mathcal{G}_{\theta}$-torsors:
\begin{enumerate}
\item \label{ex:par_princ} \textbf{Parabolic principal $G$-bundles.}

For each $x \in D$, let $\alpha_x \in \mathfrak{t}$ be a parabolic weight defining a parabolic subgroup $P_{\alpha_x} \subseteq G$ containing a Borel subgroup. Set the parahoric weight $\theta_x = \alpha_x$. The parahoric subgroup scheme $G_{\theta_x}$ over $\mathbb{D}_x = \operatorname{Spec} \mathbb{C}[[t_x]]$ is defined as $G_{\theta_x} = \mathrm{ev}_0^{-1}(P_{\alpha_x})$, where $\mathrm{ev}_0: G(\mathbb{C}[[t_x]]) \to G(\mathbb{C})$ maps $g(t_x) \mapsto g(0)$. Thus, sections of $G_{\theta_x}$ are maps $g: \mathbb{D}_x \to G$ with $g(0) \in P_{\alpha_x}$.

A $\mathcal{G}_{\boldsymbol{\theta}}$-torsor $\mathcal{E} \to X$ is a principal $G$-bundle over $X \setminus D$ and a $G_{\theta_x}$-torsor over each $\mathbb{D}_x$, locally isomorphic to $G_{\theta_x} \times \mathbb{D}_x$. A generically split $\mathcal{G}_{\boldsymbol{\theta}}$-torsor corresponds to a principal $G$-bundle $\mathcal{P}$ over $X$ such that, in any local trivialization over $\mathbb{D}_x$, sections $s: \mathbb{D}_x \to \mathcal{P}$ correspond to $g(t_x) \in G(\mathbb{C}[[t_x]])$ with $g(0) \in P_{\alpha_x}$. This defines a parabolic structure at each $x \in D$, recovering the classical notion of parabolic $G$-bundles.

\item \label{ex:par_vec} \textbf{Parabolic vector bundles.}

For $G = \mathrm{GL}_n$, a parabolic vector bundle of rank $n$ on $X$ is a vector bundle $V \to X$ equipped with a flag of subspaces in the fiber $V_x$ for each $x \in D$:
\[
V_x = F^1_x \supset F^2_x \supset \dots \supset F^k_x \supset \{0\},
\]
defining a parabolic subgroup $P_x \subset \mathrm{GL}_n$ as the stabilizer of the flag. Let $\alpha_x \in \mathfrak{t}$ (the Lie algebra of diagonal matrices) be a parabolic weight corresponding to $P_x$.

The frame bundle $\mathrm{Fr}(V)$ is a principal $\mathrm{GL}_n$-bundle over $X$. The flags induce parabolic structures on $\mathrm{Fr}(V)$ with parabolic subgroups $P_x$. Setting $\theta_x = \alpha_x$, then $\mathrm{Fr}(V)$ becomes a generically split parahoric $\mathcal{G}_{\boldsymbol{\theta}}$-torsor, where, over $\mathbb{D}_x$, the sections satisfy the flag condition via $P_x$.

\end{enumerate}
\end{exmp}

\subsection{Parahoric Lie Algebra}\label{sec:parh_Lie_algebra}
Given a weight $\theta$, then associated to $G_\theta(K)$, we have the parahoric Lie algebra
\[
\mathfrak{g}_\theta(K) = \mathfrak{t}(R) \oplus \bigoplus_{r \in \mathcal{R}} z^{m_r(\theta)} \mathfrak{g}_r(R),
\]
with the Killing form providing the identification 
\[
\mathfrak{g}_\theta(K)^\perp = \mathfrak{t}(R) \oplus \bigoplus_{r \in \mathcal{R}} z^{-m_r(\theta)} \mathfrak{g}_{-r}(R).
\]

Recall also from \cite[Section 2.6]{MoyPra}) that each parahoric subgroup $G_\theta(K) \subset G(K)$ admits an exact sequence of group schemes. There exists an exact sequence
\[
1 \to G_\theta^+(K) \to G_\theta(K) \to L_{\theta} \to 1,
\]
where $G_{\theta}^+(K)$ is the \emph{pro-unipotent radical} of $G_\theta(K)$ and $L_{\theta}$ is the reductive \textit{Levi quotient}. Since $G_\theta^+(K)$ is a normal subgroup, this sequence defines the group $L_\theta$ abstractly as the quotient $G_\theta(K) / G_\theta^+(K)$. The Lie algebra of the pro-unipotent radical, $\mathfrak{g}_\theta^+ = \text{Lie}(G_\theta^+(K))$, is an ideal in $\mathfrak{g}_\theta(K)$ given by
\begin{equation}\label{G_Theta_plus}
\mathfrak{g}_\theta^+ = z \mathfrak{t}(R) \oplus \bigoplus_{r \in \mathcal{R}} z^{\lceil -m_r(\theta) \rceil+1} \mathfrak{g}_r(R).
\end{equation}
The Lie algebra of the Levi quotient, which we will denote by $\hat{\mathfrak{l}}_\theta$, is the quotient of the corresponding Lie algebras
\[
\hat{\mathfrak{l}}_\theta := \mathfrak{g}_\theta/\mathfrak{g}_\theta^+.
\]
This is a reductive Lie algebra over the residue field $k$. The Levi decomposition theorem states that this abstract quotient can be realized as a concrete subgroup. Let $L_\theta$ be the finite-dimensional reductive group over the residue field $k$ whose root data corresponds to the integer eigenspaces of $\text{ad}_\theta$:
\begin{equation}\label{defn:Levi_at_theta}
L_\theta = C_G(e^{2\pi i \theta}) = \langle T(k^\times), U_r(k) \mid r \in \mathcal{R}, m_r(\theta) \in \mathbb{Z} \rangle.
\end{equation}
The concrete Levi subgroup $\hat{L}_\theta \subset G(K)$ is obtained by conjugating $L_\theta$:
\[
\hat{L}_\theta = z^{-\theta} L_\theta z^\theta = \langle T(k^\times), U_r(z^{-m_r(\theta)}k) \mid r \in \mathcal{R}, m_r(\theta) \in \mathbb{Z} \rangle.
\]
The corresponding Lie algebra $\hat{\mathfrak{l}}_\theta = \text{Lie}(\hat{L}_\theta)$ is a subalgebra of the loop algebra $\mathfrak{g}(K)$ given by
\[
\hat{\mathfrak{l}}_\theta = \text{Ad}(z^{-\theta})(\text{Lie}(L_\theta)) = \mathfrak{t}(k) \oplus \bigoplus_{r \in \mathcal{R}, \, m_r(\theta) \in \mathbb{Z}} z^{-m_r(\theta)}\mathfrak{g}_r(k).
\]
The Lie algebra $\mathfrak{g}_\theta(K)$ admits a decomposition (the Levi decomposition) as a semidirect product:
\[
\mathfrak{g}_\theta(K) = \hat{\mathfrak{l}}_\theta \ltimes \mathfrak{g}_\theta^+.
\]

The Lie algebra $\hat{\mathfrak{l}}_\theta$ is isomorphic to the finite-dimensional reductive Lie algebra $\mathfrak{l}_{\theta} = \text{Lie}(L_\theta)$ (cf. \cite[Section 2.2]{Boalch}), which is the subalgebra of $\mathfrak{g}(k)$:
\[
\mathfrak{l}_{\theta} = \mathfrak{t}(k) \oplus \bigoplus_{r \in \mathcal{R}, \, m_r(\theta) \in \mathbb{Z}} \mathfrak{g}_r(k).
\]
The isomorphism is given by the evaluation map $\iota: \hat{\mathfrak{l}}_\theta \to \mathfrak{l}_{\theta}$ defined by setting $z=1$ after conjugating back by $z^\theta$:
\[
\iota(X(z)) = \left. \text{Ad}(z^\theta)(X(z)) \right|_{z=1}.
\]
We thus have canonical isomorphisms:
\[
L_\theta\cong \hat{L}_\theta,\quad \mathfrak{l}_\theta \cong \hat{\mathfrak{l}}_\theta. \]
This allows us to treat the abstract Levi quotient as the concrete and familiar reductive Lie group $L_\theta$ of the group $G$, and its Lie algebra as the corresponding Lie algebra $\mathfrak{l}_{\theta}$ over $k$.

\begin{defn}[Parahoric adjoint sheaf]\label{defn:parh-adj}
For a parahoric $\mathcal{G}_{\boldsymbol \theta}$-torsor $\mathcal{E}$, we define the \textit{adjoint sheaf} $\mathcal{E}(\mathfrak{g})$ as the sheaf of infinitesimal automorphisms of $\mathcal{E}$. Its structure is best understood by describing its local sections:
\begin{itemize}
    \item On any open set $U \subset X \setminus D$, $\mathcal{E}$ is a standard $G$-bundle, and the sheaf $\mathcal{E}(\mathfrak{g})$ restricts to the standard adjoint vector bundle $(\mathcal{E}|_U) \times_G \mathfrak{g}$.
    \item Over the formal disk $\mathrm{Spec}(\hat{\mathcal{O}}_{X,x})$ at a point $x \in D$, sections of $\mathcal{E}(\mathfrak{g})$ are identified with sections of the parahoric Lie algebra $\mathfrak{g}_{\theta_x}$.
\end{itemize}
Thus, $\mathcal{E}(\mathfrak{g})$ is a coherent sheaf of $\mathcal{O}_X$-modules whose generic fiber is the complex Lie algebra $\mathfrak{g}$, but whose stalks over the points $x_i \in D$ are the full parahoric Lie algebras $\mathfrak{g}_{\theta_{x_i}}(\mathcal{O}_{X,x_i})$.

The local Lie algebra decompositions at the points $x_i \in D$ allow us to define another sheaf that inherits this structure: The \textit{pro-unipotent radical sheaf}, denoted by $\mathcal{E}(\mathfrak{g}^+)$, is the coherent subsheaf of $\mathcal{E}(\mathfrak{g})$ whose sections over the formal disk at each $x_i \in D$ correspond to elements of the pro-unipotent radical Lie algebra $\mathfrak{g}_{\theta_{x_i}}^+(\mathcal{O}_{X,x_i})$.
\end{defn}

\subsection{Parahoric Degree}\label{subsect:alg-parah-deg}

Let $G$ be a connected complex reductive Lie group. Fix a maximal torus $T\subset G$ and let 
\[
\theta\in Y(T)\otimes_{\mathbb{Z}}\mathbb{Q}
\]
be a rational weight. Denote by $G_\theta(K)\subseteq G(K)$ the parahoric subgroup corresponding to $\theta$. Recall that a parabolic subgroup $P\subset G$ (with Lie algebra $\mathfrak{p}$) is determined by a subset of roots $\mathcal{R}_P\subseteq \mathcal{R}$. In this context, we define the following subgroup of $P(K)$
\[
P_\theta(K) := \Big\langle T(A),\; U_r\Big(z^{m_r(\theta)}A\Big), \; r\in\mathcal{R}_P \Big\rangle.
\]
Let $\mathcal{P}_\theta$ be the corresponding group scheme on the formal disc $\mathbb{D}=\operatorname{Spec}(R)$ (see, e.g., \cite{HaiRap} for further details on this construction.) Furthermore, if 
\[
G_\theta(K)\longrightarrow G
\]
is the evaluation map, then its image is a parabolic subgroup $P_\theta\subset G$, whose inverse image is exactly $P_\theta(K)$.

\medskip

We now describe the global situation. Let $X$ be a smooth projective curve over $\mathbb{C}$ and fix a reduced effective divisor $D\subset X$.
For each point $x\in  D$, let $\theta_x$ be a weight and denote by $\mathcal{P}_{\theta_x}$ the corresponding local group scheme. One defines a global group scheme $\mathcal{P}_{\boldsymbol\theta}$ on $X$ by gluing the local data:
\[
\begin{cases}
\mathcal{P}_{\boldsymbol\theta}\vert_{X\setminus  D} \cong P\times (X\setminus  D),\\[1mm]
\mathcal{P}_{\boldsymbol\theta}\vert_{\mathbb{D}_x} = \mathcal{P}_{\theta_x}, \quad x\in D.
\end{cases}
\]
Accordingly, we next define the pairing that will be used in the definition of our parahoric degree on the Lie algebra of the torus (Definition \ref{defn:alg-deg}).  The case in which
$
\mathcal{P}_{\boldsymbol\theta}\big|_{\mathbb{D}_x}\;\cong\;\mathcal{P}_{\theta_x}
$
is handled in exactly the same way; the only difference is that the parahoric degree can be then defined after conjugating into the Lie algebra of a maximal torus.

By \cite[Lemma~3.18]{ChGP}, the group scheme $\mathcal{P}_{\boldsymbol\theta}$ is smooth, affine of finite type, and flat over $X$. Moreover, one has an inclusion
\[
\mathcal{P}_{\boldsymbol\theta}\subset \mathcal{G}_{\boldsymbol\theta},
\]
where $\mathcal{G}_{\boldsymbol\theta}$ is the Bruhat--Tits group scheme associated to $\boldsymbol\theta$.

\medskip

Let $\mathcal{E}$ be a $\mathcal{G}_{\boldsymbol\theta}$-torsor on $X$. A reduction of structure group of $\mathcal{E}$ to $\mathcal{P}_{\boldsymbol\theta}$ is given by a section
\[
\varsigma: X\longrightarrow \mathcal{E}/\mathcal{P}_{\boldsymbol\theta}.
\]
Denote by $\mathcal{E}_\varsigma$ the corresponding $\mathcal{P}_{\boldsymbol\theta}$-torsor obtained from the Cartesian diagram
\[
\begin{tikzcd}
\mathcal{E}_\varsigma \ar[r,dotted] \ar[d,dotted] & \mathcal{E} \ar[d] \\
X \ar[r, "\varsigma"] & \mathcal{E}/\mathcal{P}_{\boldsymbol\theta}.
\end{tikzcd}
\]
Let $\kappa:\mathcal{P}_{\boldsymbol\theta}\to\mathbb{G}_m$ be a group scheme morphism (a character of $\mathcal{P}_{\boldsymbol\theta}$). There is a natural one-to-one correspondence (see \cite[Lemma~4.2]{KSZparh})
\[
\operatorname{Hom}(\mathcal{P}_{\boldsymbol\theta},\mathbb{G}_m) \cong \operatorname{Hom}(P,\mathbb{C}^*).
\]
Denote by $\chi:P\to\mathbb{C}^*$ the character corresponding to $\kappa$. For any weight $\theta$, we define the pairing
\[
\langle \theta,\kappa\rangle := \langle \theta,\chi\rangle.
\]
Returning to the parahoric torsor $\mathcal{E}_\varsigma$, the pushforward via $\kappa$ defines a line bundle on $X$, which we denote by
\[
L(\varsigma,\kappa) := \kappa_*(\mathcal{E}_\varsigma).
\]
In the special case when $P=G$ (and if the reduction $\varsigma: X\to \mathcal{E}/\mathcal{G}_{\boldsymbol\theta}$ is trivial), one writes $L(\kappa):=\kappa_*\mathcal{E}$.

\medskip

The following notion of \emph{parahoric degree} was introduced in \cite{KSZparh}.

\begin{defn}\cite[Definition~4.2]{KSZparh}\label{defn:alg-deg}
Let $\mathcal{E}$ be a $\mathcal{G}_{\boldsymbol\theta}$-torsor on $X$, and let $\varsigma$ be a reduction of structure group to $\mathcal{P}_{\boldsymbol\theta}$. For a character $\kappa:\mathcal{P}_{\boldsymbol\theta}\to\mathbb{G}_m$, the \emph{parahoric degree} of $\mathcal{E}$ with respect to $\varsigma$ and $\kappa$ is defined by
\[
\parhdeg \mathcal{E}(\varsigma,\kappa) = \deg L(\varsigma,\kappa) + \sum_{x\in  D} \langle \theta_x,\kappa\rangle.
\]
If $\varsigma$ is a trivial reduction of $\mathcal{E}$ to $\mathcal{G}_{\boldsymbol\theta}$, we define
\[
\parhdeg \mathcal{E}(\kappa) = \deg L(\kappa) + \sum_{x\in  D} \langle \theta_x,\kappa\rangle.
\]
\end{defn}

Using this notion of parahoric degree, we define stability for parahoric $\mathcal{G}_{\boldsymbol\theta}$-torsors as follows: 
\begin{defn}\cite[Definition~4.3]{KSZparh}\label{defn:alg-stab-cond}
A parahoric $\mathcal{G}_{\boldsymbol\theta}$-torsor $\mathcal{E}$ on $X$ is called \emph{stable} (resp. \emph{semistable}) if for every proper parabolic subgroup $P\subset G$, every reduction of structure group 
\[
\varsigma: X\to \mathcal{E}/\mathcal{P}_{\boldsymbol\theta}
\]
and every nontrivial anti-dominant character $\kappa: \mathcal{P}_{\boldsymbol\theta}\to\mathbb{G}_m$, one has
\[
\parhdeg \mathcal{E}(\varsigma,\kappa)>0 \quad (\text{resp. } \ge 0).
\]
\end{defn}

\begin{rem}
 A moduli space of semistable parahoric $\mathcal{G}_{\boldsymbol\theta}$-torsors on $X$ is constructed in \cite{BS}, where $\boldsymbol\theta = (\theta_x)_{x\in D}$ involves rational weights $\theta_x \in Y(T)\otimes_{\mathbb{Z}}\mathbb{Q}$. The choice of rational weights is generally sufficient for the theory of these moduli spaces. This is because the isomorphism class of the local parahoric group scheme $\mathcal{G}_{\theta_x}$ over $\operatorname{Spec}(\mathbb{C}[\![z_x]\!])$ (which in turn determines the structure of $\mathcal{P}_{\theta_x}$) depends only on the open facet of the affine apartment (in the Bruhat--Tits building associated with $G(K_x)$) to which $\theta_x$ belongs. Since each such facet necessarily contains rational points, any distinct parahoric group scheme structure defined by this framework can be represented by a rational weight lying within the same facet. Consequently, the moduli theory developed using rational weights comprehensively addresses the range of these algebro-geometric structures.
\end{rem}

\subsection{Stable Parahoric Torsors and Associated Vector Bundles}

 Let $G$ be a connected complex reductive group and
\[
\rho\colon G \to \mathrm{GL}(V)
\]
be a rational representation on a finite-dimensional complex vector space $V$. Suppose that $V$ decomposes into a direct sum of irreducible $G$-modules
\[
V = V_1 \oplus \cdots \oplus V_r,
\]
so that $\rho = \rho_1 \oplus \cdots \oplus \rho_r$, where $\rho_i\colon G \to \mathrm{GL}(V_i)$ is the representation on $V_i$.
Following \cite[Section 3]{Rama}, we define
\[
S := \rho(G) \subset \mathrm{GL}(V), \quad \text{and}
\]
\[
C := \{(\lambda_1\rho_1(g),\dots,\lambda_r\rho_r(g)) \in \mathrm{GL}(V) \mid \lambda_i\in\mathbb{C}^*,\; g\in G\}.
\]
Since $C$ is the image of the rational homomorphism $((\lambda_i), g) \mapsto (\lambda_1\rho_1(g),\dots,\lambda_r\rho_r(g))$ from $\mathbb{C}^{*r} \times G$ to $\mathrm{GL}(V)$, it is a constructible, hence locally closed, subset of $\mathrm{GL}(V)$. We denote by $\overline{C}$ its Zariski closure in $\mathrm{End}(V)$. The sets $S, C,$ and $\overline{C}$ are invariant under left and right multiplication by elements of $S$.

Let $\mathcal{E}$ be a parahoric torsor for a parahoric group scheme $\mathcal{G}_{\boldsymbol\theta}$ over a smooth projective variety $X$.
Via the representation $\rho$, we obtain an associated vector bundle $\mathcal{E}(V) = \mathcal{E} \times^{\mathcal{G}_{\boldsymbol\theta}} V$. For each irreducible component $V_i$, $\mathcal{E}(V_i)$ is the corresponding vector subbundle. Similarly, using the adjoint representation $\mathrm{Ad}: G \to \mathrm{GL}(\mathfrak{g})$, where $\mathfrak{g} = \mathrm{Lie}(G)$, we obtain the Lie algebra bundle $\mathcal{E}(\mathfrak{g})$.

We next extend classical arguments for principal $G$-bundles, particularly those of Ramanathan \cite{Rama}, to the setting of parahoric torsors. We analyze the structure of certain endomorphisms, stabilizers of subspaces, and establish slope inequalities using the parahoric degree. These results culminate in a uniqueness property for homomorphisms between associated vector bundles (Proposition~\ref{prop:parahoric_3.1}) and a characterization of global sections of the adjoint bundle associated to a stable parahoric torsor (Proposition~\ref{prop:parahoric_3.2}).

\begin{lem}\label{lem:parahoric_3.1}
Let $\rho\colon G\to\mathrm{GL}(V)$ be an irreducible representation. Let \[\mathcal{C}_{\rho} := \overline{\{\mu\,\rho(g)\mid \mu\in\mathbb{C}^*,\,g\in G\}} \subset \mathrm{End}(V).\]
Any endomorphism $T \in \mathcal{C}_{\rho}$ is of the form
\[
\lambda\,\rho(k_1)\,p_{\mathcal{R}}\, \rho(k_2),
\]
for some $\lambda\in\mathbb{C}$, elements $k_1, k_2 \in G$, and a projection $p_{\mathcal{R}}$ onto $V_{\mathcal{R}} := \bigoplus_{\alpha \in D(\mathcal{R})} V_\alpha$. Here, $D(\mathcal{R})$ is the set of weights $\alpha = \Lambda - \sum_{\alpha_i \in \mathcal{R}} m_i \alpha_i$ (where $\Lambda$ is the highest weight of $V$, $m_i \ge 0$ integers, and $\{\alpha_j\}$ is a fixed system of simple roots of $\mathfrak{g}$), for some subset $\mathcal{R}$ of the simple roots.
Conversely, every endomorphism of this form lies in $\mathcal{C}_{\rho}$.
\end{lem}

\begin{proof}
 Let $T \in \mathcal{C}_{\rho}$. Then $T=\lim_{n\to\infty}\mu_n\,\rho(g_n)$, for some $\mu_n\in\mathbb{C}^*$ and $g_n\in G$.
Let $K$ be a maximal compact subgroup of $G$. By the Cartan decomposition $G=KAK$, we can write $g_n = k_{1,n} a_n k_{2,n}$ where $k_{1,n}, k_{2,n} \in K$ and $a_n = \exp(h_n)$ for $h_n \in i\mathrm{Lie}(A_0)$, where $A_0$ is a maximal $\mathbb{R}$-split torus whose Lie algebra is $\mathfrak{a}_0$. By passing to a subsequence, we may assume that $k_{1,n} \to k_1 \in K$ and $k_{2,n} \to k_2 \in K$.
Then $T = \rho(k_1) \left( \lim_{n\to\infty} \mu_n \rho(a_n) \right) \rho(k_2)$.
Let $h_n \in i\mathfrak{a}$ (where $\mathfrak{a}$ is the Lie algebra of a maximal torus $A \subset K$, so $i\mathfrak{a}$ is a Cartan subalgebra of $\mathfrak{g}^{\mathbb{C}}$ up to conjugation) such that $\exp(h_n)$ is conjugate to $a_n$. By passing to a further subsequence and conjugating $A$ if necessary, we can assume all $h_n$ lie in the closure of a fixed Weyl chamber, and $\alpha_j(h_n)$ (for simple roots $\alpha_j$) have defined limiting behaviors.
The operator $\lim \mu_n \rho(a_n)$ acts on each weight space $V_\alpha$ by multiplication by $C_\alpha = \lim \mu_n e^{\alpha(h_n)}$.
Let $\Lambda$ be the highest weight of $V$. Then $C_\alpha = \lim \mu_n e^{(\alpha-\Lambda)(h_n)} e^{\Lambda(h_n)}$.
For $C_\Lambda = \lim \mu_n e^{\Lambda(h_n)}$ to be finite and non-zero (or $C_\Lambda=0$ if $\lambda=0$), we analyze $(\alpha-\Lambda)(h_n)$.
There is a subset of simple roots $\mathcal{R} \subseteq \{\alpha_1, \dots, \alpha_l\}$ such that $(\alpha-\Lambda)(h_n)$ tends to a finite limit if $\alpha = \Lambda - \sum_{\alpha_i \in \mathcal{R}} m_i \alpha_i$, for $m_i \ge 0$, and tends to $-\infty$ if $\alpha$ involves simple roots not in $\mathcal{R}$ with negative coefficients in $\Lambda - \alpha$.
Thus, $\lim \mu_n \rho(a_n)$ becomes $\lambda' p_{\mathcal{R}}$, for some scalar $\lambda'$, and the projection $p_{\mathcal{R}}$ onto $V_{\mathcal{R}} = \bigoplus_{\alpha \in D(\mathcal{R})} V_\alpha$.
Therefore, $T = \lambda \rho(k_1) p_{\mathcal{R}} \rho(k_2')$, for some $k_2'$ (after absorbing $\rho(k_1)^{-1}\rho(k_1)$ and adjusting $k_2$).

For the converse, as in \cite[Lemma 3.1]{Rama}, choose $h_m \in i\mathfrak{a}$ such that $\alpha_i(h_m)=0$ for $\alpha_i \in \mathcal{R}$ and $\alpha_j(h_m) = m$ for $\alpha_j \notin \mathcal{R}$. Let $\mu_m = \lambda e^{-\Lambda(h_m)}$. Then $\mu_m \rho(\exp h_m)$ converges to $\lambda p_{\mathcal{R}}$. Pre- and post-multiplying by $\rho(k_1)$ and $\rho(k_2)$ gives the desired form.
\end{proof}

\begin{lem}\label{lem:parahoric_3.2}
Let $V_{\mathcal{R}}$ be a subspace as defined in Lemma~\ref{lem:parahoric_3.1}, corresponding to an irreducible representation $\rho$ and a subset of simple roots $\mathcal{R}$. The stabilizer $P_{\mathcal{R}} = \mathrm{Stab}_G(V_{\mathcal{R}})$ is a parabolic subgroup of $G$. If $Z_0$ is the identity component of the center of $G$, and $P_{\mathcal{R}} = Z_0 P'_{\mathcal{R}}$, then the character $\chi_{\mathcal{R}}$ of $P_{\mathcal{R}}$ given by its action on $\det(V_{\mathcal{R}})$ has its restriction to $P'_{\mathcal{R}}$ dominant. Moreover, if $V_{\mathcal{R}} \neq 0$ and $V_{\mathcal{R}} \neq V$, then $\chi_{\mathcal{R}}$ (or a related character used for stability) is non-trivial in the appropriate sense for stability arguments.
\end{lem}
\begin{proof}
 The Lie algebra of $\mathrm{Stab}_G(V_{\mathcal{R}})$ contains the Borel subalgebra $\mathfrak{b}$ (corresponding to the choice of simple roots) plus all root spaces $\mathfrak{g}^{\alpha}$ such that $\alpha$ is a sum of simple roots where those not in $\mathcal{R}$ appear with non-negative coefficients. This defines a parabolic subalgebra, so $P_{\mathcal{R}}$ is parabolic. The assertion concerning the character $\chi_{\mathcal{R}}$ follows by considering the $P_{\mathcal{R}}$-action on the line $\bigwedge^{\dim V_{\mathcal{R}}} V_{\mathcal{R}}$, whose weight is a sum of weights in $D(\mathcal{R})$, and is dominant for $P'_{\mathcal{R}}$. If $V_{\mathcal{R}}$ is a proper non-zero $P_{\mathcal{R}}$-invariant subspace, then $P_{\mathcal{R}}$ is a proper parabolic, and such characters are non-trivial.
\end{proof}

\begin{defn}\label{def:type_R_subbundle}
Let $\rho: G \to \mathrm{GL}(V)$ be an irreducible representation, and $\mathcal{R}$ be a subset of simple roots. A vector subbundle $W \subset \mathcal{E}(V)$ is said to be of \emph{type $\mathcal{R}$} if for a generic point $x \in X$, the fiber $W_x \subset \mathcal{E}(V)_x \cong V$ is $G$-conjugate to the subspace $V_{\mathcal{R}} = \bigoplus_{\alpha \in D(\mathcal{R})} V_\alpha$.
\end{defn}

\begin{lem}\label{lem:parahoric_3.3}
Let $\mathcal{E}$ be a stable (resp. semistable) parahoric $\mathcal{G}_{\boldsymbol{\theta}}$-torsor on $X$, and let $\rho\colon G\to\mathrm{GL}(V)$ be an irreducible representation. If $W \subset \mathcal{E}(V)$ is a proper, non-zero subbundle of type $\mathcal{R}$ for some $\mathcal{R} \subset \{\alpha_1,\dots,\alpha_l\}$, then
\[
\frac{\parhdeg W}{\mathrm{rk} W} < \frac{\parhdeg \mathcal{E}(V)}{\mathrm{rk} V} \quad \left(\text{resp. } \frac{\parhdeg W}{\mathrm{rk} W} \le \frac{\parhdeg \mathcal{E}(V)}{\mathrm{rk} V}\right).
\]
\end{lem}
\begin{proof}
A subbundle $W$ of type $\mathcal{R}$ corresponds to a section $\sigma \colon X \to \mathcal{E}(G/P_{\mathcal{R}})$, where $P_{\mathcal{R}} = \mathrm{Stab}_G(V_{\mathcal{R}})$ is a proper parabolic subgroup (Lemma~\ref{lem:parahoric_3.2}). This section $\sigma$ defines a reduction of the structure group of $\mathcal{E}$ from $\mathcal{G}_{\boldsymbol{\theta}}$ to a parahoric subgroup scheme $\mathcal{P}_{\boldsymbol{\theta},\mathcal{R}}$ whose generic fiber is $P_{\mathcal{R}}$.
Let $\chi$ be the character of $P_{\mathcal{R}}$ given by its action on $\det(V_{\mathcal{R}})$. We form a character $\mu$ of $P_{\mathcal{R}}$, trivial on $Z_0(G)$ and dominant on $P'_{\mathcal{R}}$ (e.g., $\mu = (\det|_{V_{\mathcal{R}}})^{\mathrm{rk} V} \otimes (\det|_V)^{-\mathrm{rk} V_{\mathcal{R}}}$, restricted to $P'_{\mathcal{R}}$ and extended trivially on $Z_0(G)$).
The definition of (semi)stability for the parahoric torsor $\mathcal{E}$ implies that the parahoric degree of the associated line bundle $\mathcal{L}_\mu = \sigma^*\mathcal{E}(\mu)$ satisfies $\parhdeg \mathcal{L}_\mu < 0$ (resp. $\leq 0$).
The parahoric degree of $\mathcal{L}_\mu$ is given by
\[ \parhdeg \mathcal{L}_\mu = (\mathrm{rk} V) \parhdeg W - (\mathrm{rk} W) \parhdeg \mathcal{E}(V_{\mathcal{R}}), \]
where $V_{\mathcal{R}}$ is identified with $W$ via the section. More directly, \[\parhdeg \mathcal{L}_\mu = (\mathrm{rk} V)(\mathrm{rk} W) \left( \frac{\parhdeg W}{\mathrm{rk} W} - \frac{\parhdeg \mathcal{E}(V)}{\mathrm{rk} V} \right).\]
The inequality on $\parhdeg \mathcal{L}_\mu$ then translates directly to the stated slope inequality for $W$.
\end{proof}

\begin{prop}\label{prop:parahoric_3.1}
Let $\mathcal{E}$ and $\mathcal{E}'$ be parahoric $\mathcal{G}_{\boldsymbol{\theta}}$-torsors on $X$ having the same topological type, that is, $\parhdeg \mathcal{E}(V_i) = \parhdeg \mathcal{E}'(V_i)$, for any representation $V_i$. Suppose $\mathcal{E}$ is stable and $\mathcal{E}'$ is semistable. Let
\[
s=(s_1,\dots,s_r)\in \bigoplus_{i=1}^{r} H^0\Bigl(X,\,\mathrm{Hom}\bigl(\mathcal{E}(V_i),\mathcal{E}'(V_i)\bigr)\Bigr)
\]
be a section such that for each $x \in X$, $s(x) \in \overline{C}$ (when viewed as an endomorphism in $\mathrm{End}(V)$ after trivializing $\mathcal{E}_x(V)$ and $\mathcal{E}'_x(V)$).
Then, for each $i$, the induced homomorphism $s_i\colon \mathcal{E}(V_i)\to \mathcal{E}'(V_i)$ is either identically zero or an isomorphism.
Moreover, if $s_i\neq 0$ for every $i$ and if $\rho = \bigoplus \rho_i$ is faithful, then there exist scalars $\lambda_i\in\mathbb{C}^*$, $1\le i\le r$, such that $(\lambda_1 s_1,\dots,\lambda_r s_r)$ is induced by an isomorphism of parahoric torsors $\mathcal{E}\stackrel{\sim}{\to}\mathcal{E}'$.
\end{prop}
\begin{proof}
Suppose  that for some $i$, the homomorphism $s_i \colon \mathcal{E}(V_i) \to \mathcal{E}'(V_i)$ is not identically zero and not an isomorphism. Then its image $W_i = \mathrm{Im}(s_i)$ and its kernel $U_{1,i} = \mathrm{Ker}(s_i)$ are proper, non-zero subbundles (on a suitable open set).

\noindent Since $s_i(x) \in \overline{C}_i$, the image subbundle $W_i \subset \mathcal{E}'(V_i)$ is of type $\mathcal{R}_i$ for some subset of simple roots $\mathcal{R}_i$. As $\mathcal{E}'$ is semistable, Lemma~\ref{lem:parahoric_3.3} implies that $\mu_{\text{par}}(W_i) \le \mu_{\text{par}}(\mathcal{E}'(V_i))$. From the short exact sequence $0 \to U_{1,i} \to \mathcal{E}(V_i) \to W_i \to 0$ and the fact that $\mu_{\text{par}}(\mathcal{E}(V_i)) = \mu_{\text{par}}(\mathcal{E}'(V_i))$, this leads to the inequality
\[\mu_{\text{par}}(U_{1,i}) \ge \mu_{\text{par}}(\mathcal{E}(V_i)).\]
Now, we analyze the kernel subbundle $U_{1,i} \subset \mathcal{E}(V_i)$. For any point $x \in X$, the fiber $(U_{1,i})_x$ is the kernel of the endomorphism $s_i(x) \in \overline{C}_i$. By Lemma~\ref{lem:parahoric_3.1}, an endomorphism in $\overline{C}_i$ has a kernel that is $G$-conjugate to a standard subspace of the form $V_{\mathcal{R}'}$ (the sum of weight spaces not in the image of the projection $p_{\mathcal{R}}$). Therefore, the kernel subbundle $U_{1,i}$ is also of type $\mathcal{R}'$ for some $\mathcal{R}'$.

\noindent Since $\mathcal{E}$ is stable, we can apply Lemma~\ref{lem:parahoric_3.3} directly to the proper, non-zero subbundle $U_{1,i} \subset \mathcal{E}(V_i)$. This gives the strict slope inequality
\[\mu_{\text{par}}(U_{1,i}) < \mu_{\text{par}}(\mathcal{E}(V_i)).\]
This is a direct contradiction to our earlier finding that $\mu_{\text{par}}(U_{1,i}) \ge \mu_{\text{par}}(\mathcal{E}(V_i))$. Therefore, each homomorphism $s_i$ must be either identically zero or an isomorphism.

If $s_i$ is a generic isomorphism, then $\det(s_i)$ is a holomorphic section of $\det(\mathcal{E}(V_i)^*) \otimes \det(\mathcal{E}'(V_i))$. Let $k_i = \mathrm{rk} V_i$, and then this line bundle is $\bigwedge^{k_i} \mathcal{E}(V_i)^* \otimes \bigwedge^{k_i} \mathcal{E}'(V_i)$. Since $\mathcal{E}$ and $\mathcal{E}'$ have the same topological type, this line bundle has parahoric degree zero. A non-zero holomorphic section of a degree-zero line bundle over a projective variety $X$ must be nowhere vanishing (if $H^0(X, \mathcal{O}_X) = \mathbb{C}$). Thus, $\det(s_i)$ is nowhere zero, and $s_i$ is an isomorphism.
Therefore, each $s_i$ is either zero or an isomorphism.

For the second part, assume that each $s_i \neq 0$, so each $s_i$ is an isomorphism. Then $s(X) \subset \mathcal{E}(C)$ (where $\mathcal{E}(C)$ denotes sections whose values at each point $x$ lie in $C$ relative to the fiber $\mathcal{E}_x(V)$). This means that, for each $x \in X$, there exist $\lambda_i(x) \in \mathbb{C}^*$ and $g(x) \in G$ such that $s_j(x) = \lambda_j(x) \rho_j(g(x))$, for all $j=1,\dots,r$. The argument from \cite[p.~136]{Rama} involving characters $\chi_k$ on $\mathbb{C}^{*r}$ shows that $\chi_k(\lambda_1(x), \dots, \lambda_r(x))$ are constant. This allows finding constants $\lambda_j \in \mathbb{C}^*$ such that $(\lambda_1 s_1, \dots, \lambda_r s_r)$ takes values in $\mathcal{E}(S)$. If $\rho$ is faithful, this implies that $(\lambda_1 s_1, \dots, \lambda_r s_r)$ is induced by an isomorphism of parahoric torsors $\mathcal{E} \stackrel{\sim}{\to} \mathcal{E}'$.
\end{proof}

\begin{rem}
The condition $s(X) \subset \mathcal{E}(\overline{C})$ means that for any $x \in X$, $s(x)$, viewed as an element of $\mathrm{Hom}(V,V)$ via trivializations of $\mathcal{E}_x(V)$ and $\mathcal{E}'_x(V)$, lies in $\overline{C}$. If $s_i(x)$ is not an isomorphism, it corresponds to an element in $\overline{C}_i \setminus C_i$.
If $s(X)$ were not contained in $\mathcal{E}(\overline{C})$, then the arguments relying on Lemma \ref{lem:parahoric_3.1} (about the structure of elements in $\overline{C}$) would not directly apply. Note that the assumption $s(X) \subset \mathcal{E}(\overline{C})$ is also used in \cite[Proposition 3.1]{Rama} for the case of principal $G$-bundles.
\end{rem}

\begin{prop}\label{prop:parahoric_3.2}
Let $\mathcal{E}$ be a stable parahoric $\mathcal{G}_{\boldsymbol{\theta}}$-torsor on $X$. Then the space of global sections of the adjoint bundle $H^0\bigl(X,\mathcal{E}(\mathfrak{g})\bigr)$ is isomorphic to the center $\mathfrak{z}$ of the Lie algebra $\mathfrak{g}$. In particular, if $G$ is semisimple (so that $\mathfrak{z}=0$), then $H^0\bigl(X,\mathcal{E}(\mathfrak{g})\bigr)=0$, and consequently, the group of automorphisms $\mathrm{Aut}(\mathcal{E})$ that are extensions of the identity on $X$ is discrete (and finite if $X$ is projective).
\end{prop}
\begin{proof}
Let $\operatorname{Aut}_X(\mathcal{E})$ be the group of automorphisms of $\mathcal{E}$ covering the identity on $X$. Its Lie algebra is $\mathrm{Lie}(\operatorname{Aut}_X(\mathcal{E})) \cong H^0(X, \mathcal{E}(\mathfrak{g}))$.
Let $\rho = \bigoplus_{i=1}^r \rho_i: G \to \mathrm{GL}(V = \bigoplus V_i)$ be a faithful representation, where each $V_i$ is irreducible.
An automorphism $\alpha \in \operatorname{Aut}_X(\mathcal{E})$ induces isomorphisms $s_i(\alpha): \mathcal{E}(V_i) \to \mathcal{E}(V_i)$ for each $i$. By Proposition~\ref{prop:parahoric_3.1} (with $\mathcal{E}'=\mathcal{E}$, which is stable and thus semistable), each $s_i(\alpha)$ is an isomorphism.
As in \cite[Proof of Prop 3.2]{Rama}, the induced sections $(s_1(\alpha), \dots, s_r(\alpha))$ define a map from $\operatorname{Aut}_X(\mathcal{E})$ to a product of projective spaces of sections (after projectivizing and considering $H^0(X, \mathcal{E}(\mathrm{End} V_i))$). The image of this map is finite. More precisely, there is a homomorphism $\Phi: \operatorname{Aut}_X(\mathcal{E}) \to I$, where $I$ is a finite group.
This implies that the connected component of the identity, $\operatorname{Aut}_X^0(\mathcal{E})$, is contained in $\text{Ker} \Phi$.
Thus $\mathrm{Lie}(\operatorname{Aut}_X(\mathcal{E})) = \mathrm{Lie}(\operatorname{Aut}_X^0(\mathcal{E})) \subseteq \mathrm{Lie}(\text{Ker} \Phi)$.
An element $h \in H^0(X, \mathcal{E}(\mathfrak{g}))$ corresponds to a 1-parameter subgroup $\exp(th)$ in $\operatorname{Aut}_X^0(\mathcal{E})$. If $\exp(th) \in \text{Ker} \Phi$, then for each $i$, the induced isomorphism $s_i(\exp(th))$ on $\mathcal{E}(V_i)$ must correspond to the identity element in the relevant factor of $I$.
The isomorphism $s_i(\exp(th))$ is given by $\exp(t \cdot \rho_i(h))$ acting on the fibers (where $\rho_i(h)$ here denotes the action derived from the adjoint action, i.e., $\mathcal{E}(\rho_i(h))$).
For $\exp(t \cdot \rho_i(h))$ to effectively be trivial in $I$ for all $t$, $\rho_i(h)$ must correspond to a scalar endomorphism for each $i$. Since $\rho_i$ is an irreducible representation of $G$, if $\rho_i(h)$ is a scalar matrix, and $h \in \mathfrak{g}$, then by Schur's Lemma, $\rho_i(h) = c_i \cdot \operatorname{Id}_{V_i}$ for $c_i \in \mathbb{C}$. If $h$ is in $\mathfrak{g}$, then $\rho_i(h)$ is an endomorphism that comes from $\mathfrak{g}$. For $\rho_i(h)$ to be scalar for all $i$ in a faithful representation $\rho=\bigoplus \rho_i$, $h$ must lie in the center $\mathfrak{z}$ of $\mathfrak{g}$.
Thus, $H^0(X, \mathcal{E}(\mathfrak{g})) = \mathrm{Lie}(\operatorname{Aut}_X^0(\mathcal{E})) \subseteq \mathfrak{z}$.
Conversely, if $h \in \mathfrak{z}$, then $\exp(th)$ is a 1-parameter subgroup of $Z(G)$ (center of $G$). These act as global automorphisms of $\mathcal{E}$ and correspond to sections in $H^0(X, \mathcal{E}(\mathfrak{z})) \cong H^0(X, \mathcal{O}_X) \otimes \mathfrak{z} \cong \mathfrak{z}$ (if $H^0(X, \mathcal{O}_X) = \mathbb{C}$).
Therefore, $H^0(X, \mathcal{E}(\mathfrak{g})) \cong \mathfrak{z}$.

If $G$ is semisimple, then $\mathfrak{z} = 0$, and so $H^0(X, \mathcal{E}(\mathfrak{g}))=0$. This implies that $\operatorname{Aut}_X^0(\mathcal{E})$ is trivial, so $\operatorname{Aut}_X(\mathcal{E})$ is discrete. If $X$ is projective, it is known that such an automorphism group is finite (e.g., it is an algebraic group scheme).
\end{proof}

\section{\texorpdfstring{$D$-level Structure on Parahoric Torsors}{D-level Structure on Parahoric Torsors}}

In this section we introduce the concept of a $D$-level structure on a parahoric $\mathcal{G}_{\boldsymbol \theta}$-torsor over a curve, generalizing the notion of $D$-level structure on vector bundles given by Seshadri in \cite{Seshadri}. We then construct their moduli space as an irreducible, normal, projective variety.  

\subsection{\texorpdfstring{$D$-level structures and stability}{D-level structures and stability}}

As in the previous section, let $X$ be a smooth complex projective curve and let $K_X$ represent the holomorphic cotangent bundle of $X$. We consider a reduced effective divisor $D$ on $X$ and a connected complex reductive Lie group $G$.

A parahoric $\mathcal{G}_{\boldsymbol\theta}$-torsor $\pi: \mathcal{E} \to X$ over $X$ is a holomorphic fiber bundle over $X$ equipped with a holomorphic right-action of the group $\mathcal{G}_{\boldsymbol\theta}$, $q: \mathcal{E} \times \mathcal{G}_{\boldsymbol\theta} \to \mathcal{E}$, satisfying:
 $$\pi(q(z, g)) = \pi(z),\mbox{ for all } (z, g) \in \mathcal{E} \times \mathcal{G}_{\boldsymbol\theta}.  $$

For notational convenience, the point $q(z, g) \in \mathcal{E}$ will be denoted by $z\cdot g$. For any $x \in X$, the fiber $\pi^{-1}(x) \subset \mathcal{E}$ will be also denoted by $\mathcal{E}_x$.

The following definition extends to the case of parahoric $\mathcal{G}_{\boldsymbol\theta}$-torsors the fundamental notion of a \emph{$D$-level structure} on vector bundles introduced by Seshadri in \cite[Quatri\`{e}me Partie]{Seshadri}.
\begin{defn}{($D$-level structure)}\label{defn:D_level_structure} 
A $D$-level structure of parahoric type $(\theta_x)_{x\in D}$ on a parahoric $\mathcal{G}_{\boldsymbol\theta}$-torsor $\mathcal{E}$ over $X$ is a choice of section $\eta_x:\mathbb{D}_x \to G_{\theta_x}/G_{\theta_x}^+$, for each $x\in D$. We will denote a parahoric $\mathcal{G}_{\boldsymbol\theta}$-torsor equipped with a $D$-level structure as a pair $(\mathcal{E}, \eta)$.
\end{defn}

We say that two pairs $(\mathcal{E}_1,\eta_1)$ and $(\mathcal{E}_2,\eta_2)$ are \emph{equivalent} if there exists an isomorphism $h:\mathcal{E}_1\to \mathcal{E}_2$ that sends $G_{1,\theta_x}/G_{1,\theta_x}^+$ to $G_{2,\theta_x}/G_{2,\theta_x}^+$ and such that $h\circ \eta_{1,x}=\eta_{2,x}$.

Note that for each point $x\in D$, the $D$-level structure defines locally over a formal disk $\mathbb{D}_x$ a reduction of structure group of $G_{\theta_x}$-torsor to a $G_{\theta_x}^+$-torsor. For each $x\in D$, there is a natural quotient map $q_x:G_{\theta_x}\to G_{\theta_x}/G_{\theta_x}^+$. Let also
\[E_{G_{\theta_x}^+}:=q_x^{-1}(\eta_x(\mathbb{D}_x))\subset G_{\theta_x}\]
as a principal $G_{\theta_x}^+$-bundle, and we define $\mathcal{E}_{G_{\theta_x}^+}(\mathfrak{g}_{\theta_x}^+)\subset \mathcal{E}_x(\mathfrak{g}_{\theta_x})$. Locally, a choice of $\eta_x$ chooses a Lie subalgebra $\mathfrak{g}_{\theta_x}^+$ of $\mathfrak{g}_{\theta_x}$. As so we have the short exact sequence
\[0\to \mathcal{E}(\mathfrak{g})_\eta\to \mathcal{E}(\mathfrak{g})\to \bigoplus_{x\in D} \mathcal{E}(\mathfrak{g}_{\theta_x}/\mathfrak{g}_{\theta_x}^+)\to 0.\]
Here $\mathcal{E}(\mathfrak{g})_\eta$ is the associated adjoint sheaf to a new torsor, $\mathcal{E}_\eta$, which is obtained from $\mathcal{E}$ by reducing the structure group from $\mathcal{G}_{\theta_x}$ to $\mathcal{G}_{\theta_x}^+$ at each point 
$x\in D$.
\begin{rem}
In \cite[Section 4.2]{Yun}, Yun defines bundles with a notion of  \emph{parahoric level structure} at a varying point as a reduction to the Borel subgroup. This notion was introduced in order to introduce parahoric versions of the Hitchin stacks. 
\end{rem}

\begin{defn}\label{defn:stable_level_str}
We call a $D$-level structure on a parahoric $\mathcal{G}_{\boldsymbol \theta}$-torsor \emph{stable} (resp. \emph{semistable}) if the underlying parahoric $\mathcal{G}_{\boldsymbol \theta}$-torsor is stable (resp. semistable).
\end{defn}

\subsection{From Parahoric to Equivariant Level Structures}

In order to establish the construction of a moduli space of stable pairs $(\mathcal{E}, \eta)$, we first show the equivalence of the notion of $D$-level structure on a parahoric $\mathcal{G}_{\boldsymbol \theta}$-torsor to a notion of level structure on an equivariant principal $G$-bundle with respect to a cyclic group action; we refer to \cite{BS}, \cite{KSZparh} for more details on the correspondence between parahoric $\mathcal{G}_{\boldsymbol \theta}$-torsors and equivariant $G$-bundles. 

Let $Y$ be a smooth algebraic curve over $\mathbb{C}$ and let $\Gamma$ be a cyclic group of order $d$ together with an action on $Y$. Denote
by $R$ a finite set of points of $Y$ (also a divisor), such that the stabilizer group $\Gamma_y$ is nontrivial for any $y \in R$. Let \(G\) be a connected complex reductive group with maximal torus \(T\) and root system \(\mathcal{R}\). Fix a rational cocharacter 
\[
\theta\in Y(T)\otimes\mathbb{Q}
\]
and let \(d>0\) be the smallest integer so that
\(
d\,\theta\in Y(T).
\)
Define
\[
\Delta:=d\,\theta\in Y(T).
\]

Assume that the group \(\Gamma\) acts on the formal disc \(\mathbb{D}_y\) via
\[
\gamma\cdot \omega = \zeta\,\omega,
\]
with \(\zeta\) a primitive \(d\)th root of unity, and that the representation
\[
\rho:\Gamma\to T\subset G,\quad \rho(\gamma)=\Delta(\zeta)
\]
has been defined.
As so we may define the subgroup of \(G\) fixed by the \(\Gamma\)-action,
\[
G^\Gamma :=C_G\bigl(\rho(\gamma)\bigr)= \{g\in G\mid \rho(\gamma)\,g\,\rho(\gamma)^{-1}=g\}.
\]
In what follows the associated parahoric subgroup \(G_\theta\) (with pro-unipotent radical \(G_\theta^+\)) has (generalized) Levi subgroup
\[
L_\theta = \langle T(k), U_r(k), r \in \mathcal{R}_{\theta} \rangle.
\]
Moreover, set
\[
\mathcal{R}_{\theta} := \{ r\in \mathcal{R} \mid r(\theta)\in \mathbb{Z}\}.
\]
When \(\theta\) is small, i.e. \(|r(\theta)|<1\) for all \(r\), one recovers the classical Levi subgroup.

\begin{lem}[Centralizer and Levi Factor]\label{lem:centralizer}
We have
\[
C_G\bigl(\rho(\gamma)\bigr) = L_\theta = \langle T(k),\, U_r(k) \mid r\in \mathcal{R}_{\theta} \rangle.
\]
\end{lem}

\begin{proof}
An element \(g\in G\) commutes with \(\rho(\gamma)\) if and only if
\[
\rho(\gamma)\,g\,\rho(\gamma)^{-1}=g.
\]
Since \(\rho(\gamma)\in T\) and \(T\) is abelian, every \(t\in T\) lies in \(C_G\bigl(\rho(\gamma)\bigr)\). For a unipotent element \(u_r(x)\in U_r(k)\), the conjugation formula yields
\[
\rho(\gamma)\,u_r(x)\,\rho(\gamma)^{-1}=u_r\Bigl(\zeta^{d\,r(\theta)}\,x\Bigr).
\]
Thus, \(u_r(x)\) commutes with \(\rho(\gamma)\) if and only if \(\zeta^{d\,r(\theta)}=1\). Since \(\zeta\) is a primitive \(d\)th root of unity, this holds if and only if
\[
d\,r(\theta)\in d\,\mathbb{Z}\quad\Longleftrightarrow\quad r(\theta)\in \mathbb{Z}.
\]
Hence,
\[
C_G\bigl(\rho(\gamma)\bigr)=\langle T(k),\, U_r(k) \mid r\in \mathcal{R}_{\theta} \rangle,
\]
which by definition is equal to \(L_\theta\).
\end{proof}

\begin{thm}[Equivalence of \(D\)-Level Structures and Equivariant Maps]\label{thm:equivlevel}
Let \(x\in D\) be a point on a smooth projective curve \(X\) with formal disc \(\mathbb{D}_x\). Suppose that \(\mathcal{E}\) is a parahoric $\mathcal{G}_{\boldsymbol \theta}$-torsor equipped with a \(D\)-level structure at \(x\), i.e. a section
\[
\eta_x:\mathbb{D}_x\to G_{\theta_x}/G_{\theta_x}^+.
\]
Let
\[
p:\mathbb{D}_y\to \mathbb{D}_x
\]
be the Galois covering with group \(\Gamma\) (acting by \(\gamma\cdot\omega=\zeta\,\omega\)). Then, the level structure \(\eta_x\) (which, via the identification \(G_{\theta_x}/G_{\theta_x}^+ \cong L_{\theta_x}\), takes values in \(L_{\theta_x}\)) is equivalent to a \(\Gamma\)-invariant map
  \[
  \tilde{\eta}:\mathbb{D}_y\to G
  \]
  satisfying
  \[
  \tilde{\eta}(\gamma\cdot y)=\tilde{\eta}(y), \quad \text{for all } \gamma\in\Gamma,
  \]
  so that the image of \(\tilde{\eta}\) lies in \(G^\Gamma\cong L_\theta\).
\end{thm}

\begin{proof}
Since \(G_{\theta_{x}} /G_{\theta_{x}}^+\cong L_{\theta_{x}}\) by definition (see Section 2.3), a \(D\)-level structure \(\eta_x\) provides a reduction of the structure group of \(\mathcal{E}\) over \(\mathbb{D}_x\) to \(L_{\theta_{x}}\). By standard descent theory for the Galois covering \(p:\mathbb{D}_y\to\mathbb{D}_x\), such a reduction is equivalent to the existence of a unique \(\Gamma\)-invariant lift
\[
\tilde{\eta}:\mathbb{D}_y\to G
\]
satisfying \(\tilde{\eta}(\gamma\cdot y)=\tilde{\eta}(y)\), for all \(\gamma\in\Gamma\). Conversely, any such \(\Gamma\)-invariant map descends to a section \(\eta_x:\mathbb{D}_x\to G^\Gamma\cong L_{\theta_{x}}\), thereby defining a \(D\)-level structure.
\end{proof}

In the light of the previous theorem, we now introduce the following: 

\begin{defn}[Equivariant $D$-level structure]
We shall call a \emph{$\Gamma$-equivariant $D$-level structure} on a $\Gamma$-principal bundle to be a \(\Gamma\)-invariant map
  \[
  \tilde{\eta}:\mathbb{D}_y\to G
  \]
  satisfying
  \[
  \tilde{\eta}(\gamma\cdot y)=\tilde{\eta}(y), \quad \text{for all } \gamma\in\Gamma,
  \]
  so that the image of \(\tilde{\eta}\) lies in \(G^\Gamma\cong L_{\theta_{x}}\).  
\end{defn}

\subsection{Moduli space}\label{sec:moduli_space_U_constr}
There is a correspondence between a (semi)stable parahoric $\mathcal{G}_{\boldsymbol \theta}$-torsor with $D$-level structure and a (semi)stable equivariant $(\Gamma, G)$-bundle equipped with an equivariant $D$-level structure. Indeed, in the light of Definition \ref{defn:stable_level_str}, the (semi)stability of $\mathcal{E}$ equipped with a $D$-level structure $\eta$ amounts to the (semi)stability of $\mathcal{E}$, and from \cite[Theorem 4.12]{KSZparh} (following \cite[Theorem 6.3.5]{BS}) we have that a  parahoric torsor is (semi)stable if and only if the corresponding equivariant $G$-bundle is.

The \emph{local type} of \(E\) at \(y\) is defined to be the conjugacy class (in \(G\)) of the representation \(\rho_y\). We denote by \(\boldsymbol{\tau}\) the collection of such local types at all points of \(Y\) lying over the branch locus of the covering \(Y\to X\).

A standard approach (following Balaji--Seshadri \cite[Section 8]{BS}) is to first fix a faithful representation 
\[
G\hookrightarrow \mathrm{GL}(n),
\]
and then consider a parameter scheme \(Q^{\boldsymbol\tau}_{(\Gamma,\mathrm{GL}(n))}\) which classifies \(\Gamma\)-equivariant vector bundles on a suitable ramified cover \(Y\) (of \(X\)) that are semistable and of fixed local type \(\boldsymbol\tau\). In particular, the points of \(Q^{\boldsymbol\tau}_{(\Gamma,\mathrm{GL}(n))}\) correspond to \(\Gamma\)-semistable principal \((\Gamma,\mathrm{GL}(n))\)-bundles.

Next, one defines
\[
Q^{\boldsymbol\tau}_{(\Gamma,G)}\subset Q^{\boldsymbol\tau}_{(\Gamma,\mathrm{GL}(n))}
\]
to be the subscheme parameterizing those bundles which admit a reduction of structure group to \(G\). This subscheme has the local universal property for families of semistable \((\Gamma,G)\)-bundles of local type \(\boldsymbol\tau\).

The extra data of a \(D\)-level structure is now imposed as follows. For each \(x\in D\), consider the formal disc \(\mathbb{D}_y\) over \(y\) and the fixed equivariant map
\[
\tilde{\eta}_x: \mathbb{D}_y \to G
\]
(which, by definition, has image in \(G^{\Gamma_x}\), where $\Gamma_x$ denotes the Galois group acting for the particular parahoric point $x$). The section $\tilde{\eta}_x$ being a trivialization over \(\mathbb{D}_y\) of the universal family of \((\Gamma,G)\)-bundles  is a closed condition. Hence, one obtains a closed (or locally closed) subscheme
\[
Q^{\boldsymbol\tau}_{(\Gamma,G,D)}\subset Q^{\boldsymbol\tau}_{(\Gamma,G)}
\]
parameterizing those \((\Gamma,G)\)-bundles together with the prescribed \(D\)-level structure.

Finally, one forms the good quotient (in the sense of Geometric Invariant Theory) by a suitable reductive group \(\mathcal{H}\) acting on \(Q^{\boldsymbol\tau}_{(\Gamma,G,D)}\) to obtain the coarse moduli space
\[
\mathcal{U}(X,\mathcal{G}_{\boldsymbol \theta}):= Q^{\boldsymbol\tau}_{(\Gamma,G,D)}//\mathcal{H}.
\]
Remember that the moduli space of parahoric $\mathcal{G}_{ \theta}$-torsors was constructed in \cite[Section 8]{BS} via the correspondence to $\Gamma$-equivariant $G$-bundles. Since the extra \(D\)-level structure on a parahoric $\mathcal{G}_{\boldsymbol \theta}$-torsor (equivalently, the equivariant $D$-level structure on a $\Gamma$-equivariant $G$-bundle) is a rigid condition, all the standard arguments (regarding the existence of the quotient, its irreducibility, normality, and projectivity) carry through unchanged.

Let \(\mathcal{E} \to S \times X\) be a flat family of \(\Gamma\)-equivariant principal \(G\)-bundles on \(X\), each equipped with its prescribed $\Gamma$-equivariant $D$-level structure \(\tilde{\eta}\).  Suppose there is a reference \((\Gamma,G)\)-bundle \(\mathcal{E}_0\) (with the same level $D$-level structure \(\tilde{\eta}\)) such that, for a dense open subset \(U \subset S\), every fiber \(\mathcal{E}_s\) with \(s \in U\) is \(\Gamma\)-equivariantly isomorphic to \(\mathcal{E}_0\) in a way compatible with the given level structure.  Then we say that \(\mathcal{E}\) is \emph{S-equivalent} to \(\mathcal{E}_0\).  

We summarize the previous analysis to the following:

\begin{thm}[Existence of Moduli Space]\label{thm:moduli-d-level}
Let \(X\) be a smooth projective curve over \(\mathbb{C}\) and \(D\subset X\) a reduced effective divisor. Let \(\mathcal{G}_{\boldsymbol \theta}\) be the parahoric Bruhat--Tits group scheme on \(X\) corresponding to a collection \({\boldsymbol \theta}:=\{\theta_x\}_{x\in D}\) of rational weights. Then, the moduli functor which assigns to any scheme \(S\) the set of \(S\)-equivalent classes of semistable parahoric \(\mathcal{G}_{\boldsymbol \theta}\)-torsors on \(X\) with fixed \(D\)-level structure is corepresented by an irreducible, normal, projective variety.
\end{thm}

We thus introduce the following
\begin{defn}    
We will denote by \(
\mathcal{U}(X,\mathcal{G}_{\boldsymbol \theta}),
\) the \emph{moduli space of stable  parahoric $\mathcal{G}_{\boldsymbol \theta}$-torsors over $X$
with a $D$-level structure} in the sense of Definition \ref{defn:stable_level_str}.
\end{defn}

\begin{rem}
One can also see how the definition of a parahoric $\mathcal{G}_{\boldsymbol \theta}$-torsor with $D$-level structure reduces to Seshadri's definition of a $D$-level structure on a vector bundle from \cite[D\'{e}finition 1, p. 92]{Seshadri}. Firstly, in the absence of a parahoric structure and for structure group $G=\mathrm{GL}(n, \mathbb{C})$, a parahoric $\mathcal{G}_{\boldsymbol \theta}$-torsor reduces to a vector bundle.   Now, a $D$-level structure at a point $x\in D$ is a section $\eta_x:\mathbb{D}_x \to G_{\theta_x}/G_{\theta_x}^+$, for each $x\in D$. Restricting to the fiber over the point $x$, then a level structure is just the trivialization of the fiber over $x$. This is exactly the notion as in Seshadri's definition as an isomorphism $\eta: E\vert _{D} \to \oplus_{i=1}^r \mathcal{O}_D$ for a rank $r$ vector bundle $E$ over $X$. Indeed, locally for each $x \in D$, when $G_{\theta_x}$ is $\mathrm{GL}(\mathbb{C}[\![t]\!])$, for local coordinate $t$, and $G_{\theta_x}^+$ is trivial, then the section $\eta_x$ is a section of $\mathrm{GL}(\mathbb{C}[\![t]\!])$, thus after extending to a formal disk $\mathbb{D}_x$, the section $\eta_x$ is just an isomorphism of the fiber above $x$.
\end{rem}

\section{\texorpdfstring{Deformations of $D$-level structures and singularities of parahoric torsors}{Deformations of D-level structures and singularities of parahoric torsors}}
The deformation theory of \(D\)-level structures is studied in this Section. We study the tangent space of the moduli space $\mathcal{U}(X, \mathcal{G}_{\boldsymbol \theta})$ and its singular points. 

\subsection{\texorpdfstring{Deformation of $D$-level structure}{Deformation of D-level structure}}

We start with the following proposition.

\begin{prop}\label{prop:tangent_U}
Let \(\mathcal{U}(X,\mathcal{G}_{\boldsymbol\theta})\) be the moduli space of stable parahoric \(\mathcal{G}_{\boldsymbol\theta}\)-torsors with a \(D\)-level structure. For any representative \([(\mathcal{E},\eta)] \in \mathcal{U}(X,\mathcal{G}_\theta)\), the tangent space is
\[
T_{[(\mathcal{E},\eta)]} \mathcal{U}(X,\mathcal{G}_{\boldsymbol\theta}) \cong H^1(X, \mathcal{E}(\mathfrak{g})_\eta).
\]
\end{prop}

\begin{proof}
Choose a sufficiently fine open cover \(\{U_i\}_{i \in I}\) of \(X\) where \(\mathcal{E}\) trivializes on each \(U_i\). If \(U_i\) contains a marked point \(x_j \in D\), we need to ensure that the trivialization near \(x_j\) is compatible with the parahoric reduction \(\eta_{x_j}\). Concretely, over \(U_i\), we have \(\mathcal{E}|_{U_i} \cong U_i \times G\). On each non-empty intersection \(U_i \cap U_j\), the torsor is determined by transition functions \(g_{ij} \colon U_i \cap U_j \to G\) satisfying \(g_{ij}g_{jk} = g_{ik}\) on triple overlaps.

\noindent An infinitesimal deformation of \(\mathcal{E}\) involves deforming \(g_{ij}\) to first order in a parameter \(\varepsilon\) with \(\varepsilon^2 = 0\) as
\[
g_{ij} \mapsto g_{ij}(\mathrm{Id} + \varepsilon\,\alpha_{ij}),
\]
where each \(\alpha_{ij}\) is a section of the adjoint bundle \(\mathcal{E}(\mathfrak{g})\) over \(U_i \cap U_j\). Trivializing \(\mathcal{E}(\mathfrak{g})\) on \(U_i\) identifies these with elements of \(\Gamma(U_i, \mathfrak{g})\). The cocycle condition on triple overlaps becomes
\[
\alpha_{ij} + \mathrm{Ad}_{g_{ij}}(\alpha_{jk}) = \alpha_{ik},
\]
so \(\{\alpha_{ij}\}\) forms a Čech 1-cocycle in \(\check{Z}^1(\{U_i\}, \mathcal{E}(\mathfrak{g}))\).

\noindent Two cocycles \(\{\alpha_{ij}\}\) and \(\{\alpha'_{ij}\}\) differ by an infinitesimal gauge transformation if there exists a Čech 0-cochain \(\{\beta_i\}\) in \(\mathcal{E}(\mathfrak{g})\) such that
\[
\alpha'_{ij} = \alpha_{ij} + \beta_j - \mathrm{Ad}_{g_{ij}}(\beta_i).
\]
In Čech terms, \(\{\alpha'_{ij}\} = \{\alpha_{ij}\} + \delta(\{\beta_i\})\), where \(\delta\) is the Čech differential. The infinitesimal deformations of \(\mathcal{E}\) are classified by
\[
\check{H}^1(\{U_i\}, \mathcal{E}(\mathfrak{g})) \cong H^1(X, \mathcal{E}(\mathfrak{g})).
\]

\noindent To incorporate the \(D\)-level structure, each marked point \(x_j \in D\) imposes a parahoric reduction near \(x_j\). Gauge transformations must preserve this reduction, so if \(x_j \in U_i\), then \(\beta_i\) must lie in the subalgebra \(\mathfrak{g}_{\theta_{x_j}}^+\) stabilizing the parahoric flag. Globally, this is encoded by the subsheaf \(\mathcal{E}(\mathfrak{g})_\eta \subset \mathcal{E}(\mathfrak{g})\), whose sections preserve the parahoric structure at all points of \(D\). The allowed infinitesimal gauge transformations form \(\check{C}^0(\{U_i\}, \mathcal{E}(\mathfrak{g})_\eta)\).

\noindent Thus, the space of infinitesimal deformations of a pair \((\mathcal{E}, \eta)\) is given by Čech 1-cocycles in \(\mathcal{E}(\mathfrak{g})\) modulo these restricted 0-cochains:
\[
\check{H}^1(\{U_i\}, \mathcal{E}(\mathfrak{g})_\eta)\cong  \frac{\{\alpha_{ij}\} \in \check{Z}^1(\mathcal{E}(\mathfrak{g}))}{\delta(\{\beta_i\}) \text{ where } \{\beta_i\} \in \check{C}^0(\mathcal{E}(\mathfrak{g})_\eta)} .
\]
Since \(\{U_i\}\) is sufficiently fine, this is identified with \(H^1(X, \mathcal{E}(\mathfrak{g})_\eta)\). By the standard correspondence between first-order deformations and the tangent space, we have
\begin{equation}\label{tangentU_H1}
T_{[(\mathcal{E},\eta)]} \mathcal{U}(X, \mathcal{G}_{\boldsymbol\theta}) \cong H^1(X, \mathcal{E}(\mathfrak{g})_\eta),
\end{equation}
as claimed.
\end{proof}

Taking dual vector spaces, observe that locally at a formal disk around each point $x \in D$, it is
\[
\mathcal{E}(\mathfrak{g})_\eta|_{\mathbb{D}_x} \cong \mathfrak{g}_\theta^+ = \mathfrak{g}_\theta^\perp \otimes \mathcal{O}(-x).
\]
Dualizing, we obtain
\[
\mathcal{E}(\mathfrak{g})_\eta^\vee|_{\mathbb{D}_x} \cong \mathfrak{g}_\theta \otimes \mathcal{O}(x),
\]
and by Serre duality we have
\[
H^1(X, \mathcal{E}(\mathfrak{g})_\eta)^\vee \cong H^0(X, \mathcal{E}(\mathfrak{g}) \otimes K(D)).
\]
Therefore, the tangent space corresponds to the space of \emph{logarithmic Higgs fields} as introduced in \cite{KSZparh}:
\[H^0(X, \mathcal{E}(\mathfrak{g})\otimes K(D)).\]

We recall the following:
\begin{defn}\cite[Definition 3.1]{KSZparh}\label{defn alg parah Higgs}
A \emph{logahoric $\mathcal{G}_{\boldsymbol\theta}$-Higgs torsor} on a smooth complex algebraic curve $X$ is defined as a pair $(\mathcal{E},\varphi)$, where
\begin{itemize}
	\item $\mathcal{E}$ is a parahoric $\mathcal{G}_{\boldsymbol\theta}$-torsor on $X$;
	\item $\varphi \in H^0(X, \mathcal{E}(\mathfrak{g}) \otimes K(D))$ is a section called a \emph{logarithmic Higgs field}.
\end{itemize}
\end{defn}
In \cite{KSZparh}, a moduli space of logahoric $\mathcal{G}_{\boldsymbol\theta}$-Higgs torsors over a smooth complex algebraic curve was constructed as a quasi-projective variety. We will denote this moduli space by \(\mathcal{M}_H(X,\mathcal{G}_{\boldsymbol\theta})\).

\subsection{Regular Stability and Singularities for Parahoric Torsors}
 We next study the singular points of the moduli space \(\mathcal{U}(X,\mathcal{G}_{\boldsymbol\theta})\). We begin with the following:
 
\begin{defn}[Regularly stable parahoric torsor]
\label{def:parahoric-regular-stable}
Let $G$ be a connected complex reductive Lie group with center $Z(G)$. A stable parahoric $\mathcal{G}_{\boldsymbol\theta}$-torsor \(\mathcal{E}\) over $X$
is said to be \emph{regularly stable} if
\(\mathrm{Aut}(\mathcal{E})=Z(G).\)
Equivalently, a stable parahoric $\mathcal{G}_{\boldsymbol\theta}$-torsor with $D$-level structure is \emph{regularly stable} precisely when the underlying torsor has no extra automorphisms apart from the ones in $Z(G)$. We shall denote the moduli of regularly stable parahoric $\mathcal{G}_{\boldsymbol\theta}$-torsors with $D$-level structures by \(\mathcal{U}^{rs}(X,\mathcal{G}_{\boldsymbol\theta})\). This is an open subvariety of \(\mathcal{U}(X,\mathcal{G}_{\boldsymbol\theta})\). 
\end{defn}

\begin{rem}
Note that the notion of a regularly stable parahoric $\mathcal{G}_{\boldsymbol\theta}$-torsor extends the notion of a $\delta$-stable vector bundle of Seshadri \cite{Seshadri}, as well as the notion of a regularly stable principal $G$-bundle of Biswas--Hoffmann \cite{BiHo}. 
\end{rem}

\begin{prop}
\label{prop:parahoric-singular}
Let $[(\mathcal{E},\eta)]$ represent an isomorphism class of stable but \emph{not} regularly stable $\mathcal{G}_{\boldsymbol\theta}$-torsors  on $X$ with $D$-level structure.  
Then the corresponding point
\(
[(\mathcal{E},\eta)]\;\in\;\mathcal{U}(X,\mathcal{G}_{\boldsymbol\theta})
\)
is a singular point.
\end{prop}

\begin{proof}
The deformation theory for parahoric $\mathcal{G}_{\boldsymbol\theta}$-torsors, combined with Luna’s \'{e}tale slice theorem, identifies a Zariski neighborhood of $[(\mathcal{E},\eta)]$ in $\mathcal{U}(X,\mathcal{G}_{\boldsymbol\theta})$ with the GIT quotient
\[
H^{1}\bigl(X,\,\mathcal{E}({\mathfrak{g}_{\boldsymbol \theta}})_\eta\bigr)\;\big/\;\mathrm{Aut}(\mathcal{E})
\]
near the origin.

\noindent Because $\mathcal{E}$ is stable but \emph{not} regularly stable, there exists a non-trivial automorphism of finite order
\[
f \in \mathrm{Aut}(\mathcal{E})\setminus Z(G).
\]
The local structure of the moduli space is therefore a quotient of the vector space $V = H^{1}\bigl(X,\,\mathcal{E}({\mathfrak{g}_{\boldsymbol \theta}})_\eta\bigr)$ by the action of the finite non-trivial group $\langle f \rangle$.

\noindent We apply the Chevalley--Shephard--Todd singularity criterion: Let a finite group $H$ act linearly on a complex vector space $V$. If for some non-trivial $h\in H$, the fixed-point subspace $V^{h}$ has codimension $\ge2$, then the affine quotient $V/H$ is singular at the origin.

\noindent The automorphism $f$ acts on the torsor $\mathcal{E}$. This induces a semi-simple linear action on the fiber of the associated Lie algebra bundle over any generic point $x \in X \setminus D$, which is the vector space $\mathfrak{g} = \mathrm{Lie}(G)$. This action on $\mathfrak{g}$ induces an eigenspace decomposition $\mathfrak{g} = \mathfrak{g}_1 \oplus \mathfrak{g}_{\neq 1}$, where $\mathfrak{g}_1$ is the trivial eigenspace (the subspace fixed by $f$). Because $f \notin Z(G)$, its adjoint action on $\mathfrak{g}$ is non-trivial. The root spaces of $\mathfrak{g}$ on which $f$ acts non-trivially come in pairs (corresponding to a root and its negative). This pairing ensures that the vector space $\mathfrak{g}_{\neq 1}$ has dimension at least 2, i.e., $\dim_{\mathbb{C}}(\mathfrak{g}_{\neq 1}) \ge 2$.

\noindent Since $f$ is an automorphism of the parahoric torsor, its action preserves the parahoric structure at each point $x_i \in D$. Consequently, the decomposition of $\mathfrak{g}$ extends to a decomposition of the entire sheaf of Lie algebras ${\mathfrak{g}_{\boldsymbol \theta}}$ associated with the group scheme $\mathcal{G}_{\boldsymbol\theta}$, thus giving a direct sum of sheaves
\[
{\mathfrak{g}_{\boldsymbol \theta}} = {\mathfrak{g}_{\boldsymbol \theta}}_1 \oplus {\mathfrak{g}_{\boldsymbol \theta}}_{\neq 1}.
\]
This decomposition lifts to the associated coherent sheaf on $X$
\[
\mathcal{E}({\mathfrak{g}_{\boldsymbol \theta}})_\eta = \mathcal{E}({\mathfrak{g}_{\boldsymbol \theta}}_1)_\eta \oplus \mathcal{E}({\mathfrak{g}_{\boldsymbol \theta}}_{\neq 1})_\eta,
\]
which in turn induces a direct sum decomposition on the first cohomology group
\[
H^{1}\bigl(X,\,\mathcal{E}({\mathfrak{g}_{\boldsymbol \theta}})_\eta\bigr) \cong H^{1}\bigl(X,\,\mathcal{E}({\mathfrak{g}_{\boldsymbol \theta}}_1)_\eta\bigr) \oplus H^{1}\bigl(X,\,\mathcal{E}({\mathfrak{g}_{\boldsymbol \theta}}_{\neq 1})_\eta\bigr).
\]
The action of $f$ is trivial on the first summand and non-trivial on the second. Thus, the fixed-point subspace is precisely $V^f \cong H^{1}\bigl(X,\,\mathcal{E}({\mathfrak{g}_{\boldsymbol \theta}}_1)_\eta\bigr)$.

\noindent The codimension of this fixed-point subspace is therefore $\dim H^{1}\bigl(X,\,\mathcal{E}({\mathfrak{g}_{\boldsymbol \theta}}_{\neq 1})_\eta\bigr)$. Let us denote the sheaf $\mathcal{F} := \mathcal{E}({\mathfrak{g}_{\boldsymbol \theta}}_{\neq 1})_\eta$. The rank of $\mathcal{F}$ (the dimension of its stalk at a generic point of $X$) is $\dim_{\mathbb{C}}(\mathfrak{g}_{\neq 1}) \ge 2$.

\noindent Since the torsor $\mathcal{E}$ is stable, the associated sheaf $\mathcal{E}({\mathfrak{g}_{\boldsymbol \theta}})$ is semistable and has parahoric degree 0. This property is inherited by its direct summands, so $\mathcal{F}$ is a semistable coherent sheaf of rank $\ge 2$ and degree 0. For such a sheaf on a projective curve, it is a standard result that its first cohomology group has dimension at least 2. A more detailed analysis using Serre duality shows that $\dim H^1(X, \mathcal{F}) = \dim H^0(X, \mathcal{F}^* \otimes K_X)$, and the properties of $\mathcal{F}^* \otimes K_X$ guarantee that this dimension is at least 2 across all genera for a non-regularly stable torsor.

\noindent Therefore, the codimension of the fixed subspace $V^f$ in $V$ is at least 2. The singularity criterion applies, and we conclude that the quotient is singular at the origin. Hence, the point $[(\mathcal{E},\eta)]$ is a singular point of the moduli space $\mathcal{U}(X,\mathcal{G}_{\boldsymbol\theta})$.
\end{proof}

\section{Poisson action and moment map}\label{sec:Poisson_moment}

The first part of this Section includes the necessary preliminaries from Symplectic and Poisson geometry over smooth algebraic varieties that will be useful for establishing our main results. We then introduce a level group and study its action on the moduli space $\mathcal{U}(X, \mathcal{G}_{\boldsymbol \theta})$ of stable parahoric $\mathcal{G}_{\boldsymbol \theta}$-torsors with $D$-level structure. This action is shown to be inducing a Poisson action on the cotangent $T^*\mathcal{U}(X, \mathcal{G}_{\boldsymbol \theta})$, thus providing a canonical moment map.

\subsection{Hamiltonian group actions on smooth algebraic varieties}\label{subs:Hamiltonian_gp_actions}

In this subsection, we gather those basic notions from Poisson geometry and completely integrable systems which will be used in the rest of the article. Standard references in the context of smooth algebraic varieties include \cite[Chapter I]{Vanhaecke} or \cite[Section 2.3]{DoMa}.

\begin{defn}
Let $X$ be a smooth algebraic variety of dimension $n$. A \textit{symplectic form} (or \textit{symplectic structure}) on $X$ is an algebraic 2-form $\omega \in \Gamma(X, \Omega_X^{2})$ such that:
\begin{enumerate}
    \item $\omega$ is \textit{closed}: $d\omega = 0$, where $d$ is the exterior derivative.
    \item $\omega$ is \textit{non-degenerate}: for every point $x \in X$, the map $T_xX \rightarrow T_x^*X$ defined by $v \mapsto \omega_x(v, \cdot)$ is an isomorphism of vector spaces. This implies that $n$ must be even.
\end{enumerate}
A smooth algebraic variety $X$ equipped with a symplectic form $\omega$ is called a \textit{symplectic algebraic variety} $(X, \omega)$.
\end{defn}

\begin{defn}
Let $X$ be a smooth algebraic variety. A \emph{Poisson bracket} on $X$ is a Lie bracket 
\begin{align*}
\left\{ \text{ } , \text{ } \right\}: \mathcal{O}_X \times \mathcal{O}_X \rightarrow \mathcal{O}_X
\end{align*}
satisfying the Leibniz rule $\left\{ f,gh \right\}=\left\{ f,g \right\}h+g\left\{ f,h \right\}$, for  $f,g,h \in \mathcal{O}_X$. Poisson brackets bijectively correspond to bi-vector fields
\[\Pi \in \Gamma \left( {{\wedge }^{2}}TX \right)\]
such that
\[\left[ \Pi ,\Pi  \right]=0.\]
The Poisson bracket ${{\left\{ , \right\}}_{\Pi }}$ that corresponds to such a bi-vector field $\Pi $ is given by ${{\left\{ f,g \right\}}_{\Pi }}=\Pi \left( df,dg \right)$.    
\end{defn}

\begin{rem}
Every symplectic algebraic variety $(X, \omega)$ is naturally a Poisson algebraic variety. The non-degenerate 2-form $\omega$ induces an isomorphism $\omega^\sharp: TX \rightarrow \Omega_X^{1}$ (the cotangent sheaf). Its inverse $(\omega^\sharp)^{-1}: \Omega_X^{1} \rightarrow TX$ allows us to define the Poisson bi-vector $\Pi \in \Gamma(X, \wedge^2 TX)$ by $\Pi(\alpha, \beta) = \omega((\omega^\sharp)^{-1}(\alpha), (\omega^\sharp)^{-1}(\beta))$, for local sections $\alpha, \beta$ of $\Omega_X^{1}$. The Poisson bracket is then given by $\{f,g\} = \Pi(df,dg) = \omega(V_f, V_g)$, where $V_f$ is the Hamiltonian vector field of $f$ (see Definition \ref{defn:Ham_vec_field} below). The condition $d\omega=0$ ensures that $[\Pi, \Pi]=0$.
\end{rem}

\begin{exmp}[Cotangent Bundles]\label{exm:cot_bun}
Let $Q$ be a smooth algebraic variety of dimension $m$. Its cotangent bundle $X = T^*Q$ is a smooth algebraic variety of dimension $2m$. Let $\pi: T^*Q \to Q$ be the canonical projection. There exists a canonical 1-form $\theta \in \Gamma(T^*Q, \Omega_{T^*Q}^{1})$, called the \textit{Liouville form} (or \textit{tautological 1-form}). If $(q_1, \ldots, q_m)$ are local coordinates on an open subset $U \subseteq Q$, and $(p_1, \ldots, p_m)$ are the corresponding fiber coordinates on $T^*U \cong U \times k^m$, then $\theta = \sum_{i=1}^m p_i dq_i$.
The 2-form $\omega = -d\theta = \sum_{i=1}^m dq_i \wedge dp_i$ is a symplectic form on $T^*Q$. This is known as the \textit{canonical symplectic structure} on the cotangent bundle (cf. \cite{AG}). Thus, $(T^*Q, \omega)$ is a symplectic algebraic variety.
\end{exmp}

\begin{exmp}[Kostant--Kirillov structures]
For a Lie group $G$ with Lie algebra $\mathfrak{g}$, there exists a canonical Poisson structure  on the dual vector space $\mathfrak{g}^*$ called the \emph{Kostant--Kirillov Poisson structure} defined by the bracket
\[\{F,G\}(\xi):=\langle \xi, [d_{\xi}F, d_{\xi}G]\rangle,\]
for $F,G \in C^{\infty}(\mathfrak{g}^*)$. This is obtained by extending on $\mathfrak{g}^*$ the symplectic structures on coadjoint orbits of $\mathfrak{g}$. Note here that $d_{\xi}F$ is identified with an element of $\mathfrak{g}=\mathfrak{g}^{**}$. On the dual $\mathfrak{g}^*$ the symplectic leaves of the Kostant--Kirillov Poisson structure are precisely the coadjoint orbits. Moreover, the rank of $\mathfrak{g}$ is equal to the smallest codimension of a coadjoint orbit.  

\end{exmp}

\begin{defn}\label{defn:Ham_vec_field}
Let $(X, \{\cdot,\cdot\})$ be a Poisson algebraic variety (or $(X,\omega)$ be a symplectic algebraic variety). For a regular function $h \in \mathcal{O}_X$, called a \textit{Hamiltonian function}, the \textit{Hamiltonian vector field} $V_h$ is the unique vector field such that $V_h(f) = \{f,h\}$, for all $f \in \mathcal{O}_X$.
If $(X,\omega)$ is symplectic, then $V_h$ is equivalently defined by the condition $i_{V_h}\omega = -dh$, where $i_{V_h}\omega$ is the interior product of $V_h$ with $\omega$.
\end{defn}

\begin{defn}
Let $G$ be an algebraic group with Lie algebra $\mathfrak{g} = T_eG$. Let $G$ act on a symplectic algebraic variety $(X, \omega)$ via a morphism $\Phi: G \times X \rightarrow X$. For each $A \in \mathfrak{g}$, the action induces a \textit{fundamental vector field} $A_X \in \Gamma(X, TX)$ defined at $x \in X$ by $A_X(x) = \frac{d}{dt}|_{t=0} (\exp(-tA) \cdot x)$, or more algebraically, as the image of $A$ under the map $\mathfrak{g} \to \Gamma(X, TX)$ induced by the action. 

The action of $G$ on $(X, \omega)$ is called \textit{Hamiltonian} if:
\begin{enumerate}
    \item For every $A \in \mathfrak{g}$, the fundamental vector field $A_X$ is Hamiltonian. That is, there exists a regular function $h_A \in \mathcal{O}_X$ such that $A_X = V_{h_A}$ (or, equivalently, $i_{A_X}\omega = -dh_A$).
    \item There exists a $G$-equivariant morphism $\mu: X \rightarrow \mathfrak{g}^*$ (where $\mathfrak{g}^*$ is the dual of the Lie algebra) called the \textit{moment map} (or \textit{momentum map}) such that $h_A(x) = \langle \mu(x), A \rangle$, for all $A \in \mathfrak{g}$ and $x \in X$. The pairing $\langle \cdot, \cdot \rangle$ is the natural pairing between $\mathfrak{g}^*$ and $\mathfrak{g}$, and $G$-equivariance here means that $\mu(g \cdot x) = \mathrm{Ad}_g^*(\mu(x))$, for all $g \in G, x \in X$, where $\mathrm{Ad}^*$ is the coadjoint action of $G$ on $\mathfrak{g}^*$.
\end{enumerate}
Often, an additional condition is imposed: the map $A \mapsto h_A$ from $\mathfrak{g}$ to $(\mathcal{O}_X, \{\cdot,\cdot\})$ is a Lie algebra homomorphism, i.e., $\{h_A, h_B\} = h_{[A,B]}$, for all $A, B \in \mathfrak{g}$.
\end{defn}

\begin{exmp}[Lifted Action on Cotangent Bundles]\label{exm:Ham_lift}
Let $G$ be an algebraic group acting on a smooth algebraic variety $Q$. This action can be lifted to an action on the cotangent bundle $T^*Q$. Let $\rho_q : G \to Q$ be the orbit map $g \mapsto g \cdot q$, for $q \in Q$. For each $\alpha \in T_q^*Q$, the induced action on $(q,\alpha) \in T^*Q$ admits a moment map
\[ \mu : T^*Q \longrightarrow \mathfrak{g}^* \]
characterized by
\[ \langle \mu(q,\alpha), A \rangle = \langle \alpha, A_Q(q) \rangle, \quad \text{for all } A \in \mathfrak{g}, \]
where $A_Q$ is the fundamental vector field on $Q$ generated by $A$. Equivalently, using the differential of the orbit map $d\rho_q : \mathfrak{g} \to T_qQ$, the moment map can be written as
\[ \mu(q,\alpha) = (d\rho_q)^*(\alpha) \in \mathfrak{g}^*. \]
This lifted action on $T^*Q$ (equipped with its canonical symplectic form) is automatically Hamiltonian.
\end{exmp}

\begin{defn}
Let $(X, \omega)$ be a symplectic algebraic variety of dimension $2m$. A smooth subvariety $L \subset X$ is called:
\begin{itemize}
    \item \textit{Isotropic} if for every $x \in L$, the tangent space $T_xL$ is an isotropic subspace of $T_xX$, that is, $\omega_x(v,w) = 0$, for all $v, w \in T_xL$. This is equivalent to saying that the pullback of $\omega$ to $L$, $\omega|_L$, is zero.
    \item \textit{Lagrangian} if it is isotropic and its dimension is $m = \frac{1}{2} \dim X$.
\end{itemize}
For a Poisson algebraic variety $(X, \Pi)$, an irreducible subvariety $Y \subset X$ is \textit{Lagrangian} if it is generically a Lagrangian subvariety of a \textit{symplectic leaf} of $X$. More precisely, $Y$ is contained in the closure of a symplectic leaf $S \subset X$, and $Y \cap S$ is a Lagrangian subvariety of $S$ (where $S$ is equipped with the symplectic structure induced by the bi-vector field $\Pi$).
\end{defn}

\begin{defn}\label{defn:aciHs}
Let $X$ be a smooth Poisson algebraic variety. An \textit{algebraically completely integrable Hamiltonian system} (often referred to as an \textit{algebraic integrable system}) is typically given by a proper morphism $H: X \to B$ to an algebraic variety $B$ of dimension $m = (\dim X)/2$, such that:
\begin{enumerate}
    \item The components $H_1, \dots, H_m$ of $H$ (if $B \subset k^m$) are in involution, i.e., $\{H_i, H_j\} = 0$, for all $i,j$.
    \item The generic fibers $H^{-1}(b)$, for $b \in B$, are \textit{Lagrangian subvarieties} of $X$.
    \item These generic fibers are (open subsets of) abelian varieties, and the Hamiltonian vector fields $V_{H_i}$ are tangent to the fibers and correspond to translation-invariant vector fields on these abelian varieties.
\end{enumerate}
More generally, an \emph{algebraically completely integrable Hamiltonian system structure} on a family of abelian varieties $H: X \to B$ is a Poisson structure on $X$ with respect to which $H: X \to B$ is a \textit{Lagrangian fibration} (meaning its generic fibers are Lagrangian).
If $X$ is a smooth algebraic variety, $B$ an algebraic variety, $A \subset B$ a proper closed subvariety, and $H: X \to B$ a proper morphism such that the fibers over $B \setminus A$ are (isomorphic to) abelian varieties, then a Poisson structure on $X$ defines an algebraically completely integrable Hamiltonian system if $H: X \to B$ is a Lagrangian fibration over $B \setminus A$.
\end{defn}

\subsection{\texorpdfstring{Action of \(G_{D}\) on Local Trivializations}{Action of GD on Local Trivializations}}

To provide a more explicit understanding of a group action on the moduli space \(\mathcal{U}(X,\mathcal{G}_{\boldsymbol \theta})\) of parahoric \(\mathcal{G}_{\boldsymbol \theta}\)-torsors with \(D\)-level structure, we analyze the local description using open covers (patches) of the curve \(X\).

Let \(D = \{x_1, x_2, \ldots, x_s\}\) be the reduced effective divisor on \(X\), where each \(x_i\) is a distinct point. Choose an open cover \(\{U_i\}_{i=0}^s\) of \(X\) such that:
\begin{itemize}
    \item \(U_0 = X \setminus D\) is the complement of the divisor \(D\).
    \item For each \(1 \leq j \leq s\), \(U_j\) is a small open disc \(\mathbb{D}_{x_j}\) around the point \(x_j \in D\), equipped with a local coordinate \(z_j\) centered at \(x_j\), i.e., \(z_j(x_j) = 0\).
\end{itemize}

Over each \(\mathbb{D}_{x_j}\), the parahoric \(\mathcal{G}_{\boldsymbol \theta}\)-torsor \(\mathcal{E}\) trivializes
\[
\mathcal{E}\big|_{\mathbb{D}_{x_j}} \cong G_{\theta_{x_j}}.
\]
The transition functions on the overlaps \(U_i \cap U_j\) (\(1\leq i, j \leq s\)) encode the gluing data, respecting the parahoric reductions at each \(x_j \in D\).

\begin{defn}[Level group]\label{defn:level_group}
Given the data introduced above, we define the \emph{level group}
\[
G_{D} = \frac{L_{\theta_1} \times \cdots \times L_{\theta_s}}{Z},
\]
where each \(L_{\theta_j}\) is the Levi subgroup as in (\ref{defn:Levi_at_theta})  at \(x_j \in D\), $j=1,...,s$ and \(Z\) is the center of \(G\).
\end{defn}
An element \(g \in G_{D}\) can be represented by a tuple \((g_1, g_2, \ldots, g_s) \in L_{\theta_1} \times L_{\theta_2} \times \cdots \times L_{\theta_s}\), modulo the diagonal action of the center \(Z\).

\begin{prop}\label{prop:localAction}
The action of \(g \in G_{D}\) on the moduli space \(\mathcal{U}(X,\mathcal{G}_{\boldsymbol \theta})\) is induced locally by the action of each \(g_j \in L_{\theta_j}\) on the corresponding local trivialization \(\mathbb{D}_{x_j} \times G\). Specifically, in the local coordinates \(z_j\) around each \(x_j \in D\), the action is given by:
\[
g \cdot (e_j(z_j), \varphi_j(z_j)) = \left( g_j \cdot e_j(z_j), \, \mathrm{Ad}(g_j) \cdot \varphi_j(z_j) \right),
\]
where:
\begin{itemize}
    \item \(e_j(z_j)\) represents a local section of the parahoric $\mathcal{G}_{\boldsymbol \theta}$-torsor over \(\mathbb{D}_{x_j}\).
    \item \(\varphi_j(z_j)\) is a local section of \(\mathcal{E}(\mathfrak{g}) \otimes K(D)\) over \(\mathbb{D}_{x_j}\), viewed as an element of the cotangent space \(T^*_{[\mathcal{E}, \eta]} \mathcal{U}(X,\mathcal{G}_{\boldsymbol \theta})\).
    \item \(\mathrm{Ad}(g_j)\) denotes the adjoint action of \(g_j\) on \(\mathfrak{g}\).
\end{itemize}
\end{prop}

\begin{proof}
The action of \(G_{D}\) on the moduli space \(\mathcal{U}(X,\mathcal{G}_{\boldsymbol \theta})\) is defined globally by modifying the level structures at each point \(x_j \in D\). Locally, near each \(x_j\), this corresponds to acting by the element \(g_j \in L_{\theta_j}\) on the trivialization \(\mathbb{D}_{x_j} \times G\).

\noindent Given a local trivialization, an element \(g_j \in L_{\theta_j}\) acts on a local section \(e_j(z_j) \in G\) by multiplication:
\[
e_j(z_j) \mapsto g_j \cdot e_j(z_j).
\]
For the cotangent vectors, which are represented by Higgs fields \(\varphi_j(z_j) \in \mathcal{E}(\mathfrak{g}) \otimes K(D)\), the action is via the adjoint representation:
\[
\varphi_j(z_j) \mapsto \mathrm{Ad}(g_j) \cdot \varphi_j(z_j).
\]
Since the action preserves the parahoric structure, \(g_j\) lies in the Levi subgroup \(L_{\theta_j}\), as so the level structure at \(x_j\) remains intact.

\noindent Finally, because \(G_{D}\) is defined modulo the center \(Z\), the overall action respects the identification under \(Z\). This completes the local description.
\end{proof}

We examine the infinitesimal action of the Lie algebra \(\mathfrak{g}_D:= \text{Lie}(G_D)\) on the cotangent space.

\begin{lem}\label{lem:infinitesimalAction}
Let \((X_1, X_2, \ldots, X_s) \in \mathfrak{g}_D\), where each \(X_j \in \mathfrak{l}_{\theta_j}\) is an element of the Levi subalgebra corresponding to \(x_j \in D\). The infinitesimal action of \((X_1, X_2, \ldots, X_s)\) on a cotangent vector \(\varphi \in H^0(X, \mathcal{E}(\mathfrak{g}) \otimes K(D))\) is given locally near each \(x_j\) by:
\[
\delta_X \varphi_j(z_j) = [X_j, \varphi_j(z_j)].
\]
\end{lem}

\begin{proof}
The infinitesimal action of \(X_j \in \mathfrak{l}_{\theta_j}\) on the Higgs field \(\varphi_j(z_j)\) is induced by the adjoint action:
\[
\delta_X \varphi_j(z_j) = \left. \frac{d}{d\varepsilon} \right|_{\varepsilon=0} \mathrm{Ad}\left(\exp(\varepsilon X_j)\right) \varphi_j(z_j) = [X_j, \varphi_j(z_j)].
\]
This commutator arises naturally from the linearization of the adjoint action at the identity.
\end{proof}

We may now determine the action of the level group on parahoric $\mathcal{G}_{\boldsymbol \theta}$-torsors. 
\begin{thm}{(Freeness)}
The level group $G_D$ acts \emph{freely} on the regularly stable moduli space $\mathcal{U}^{rs}(X,\mathcal{G}_{\boldsymbol \theta})$ of parahoric $\mathcal{G}_{\boldsymbol \theta}$-torsors over $X$ 
with $D$-level structure. 
\end{thm}

\begin{proof}
Let $(\mathcal{E}, \eta)$ be a point in $\mathcal{U}(X,\mathcal{G}_{\boldsymbol \theta})$. For each $i=1,...,s$, the level structure is a reduction of structure group from $G_{\theta_{x_i}}$ to $G_{\theta_{x_i}}^+$ over a formal disk $\mathbb{D}_{x_i}$, and the Levi factor $L_{\theta_i}$ naturally changes this local reduction. Hence $\prod_{i=1}^{s} L_{\theta_i}$ acts on such data, and this action factors through 
$G_D=(\prod_{i=1}^{s} L_{\theta_i})/Z$ because the center $Z$ acts trivially.

\noindent To show the action is free, let $g\in G_D$ fix a point $(\mathcal{E},\eta) \in \mathcal{U}^{rs}(X,\mathcal{G}_{\boldsymbol \theta})$. Lift $g$ to some element $\tilde{g} = (g_1, \dots, g_s) \in \prod_{i=1}^s L_{\theta_i}$. The action of $\tilde{g}$ on $(\mathcal{E}, \eta)$ by multiplying $g_i$ on the left of each $\eta_i$ yields a new pair $(\mathcal{E}', \eta')$. The fact that $g$ fixes $(\mathcal{E}, \eta)$ means there exists a global isomorphism of parahoric torsors $\phi: \mathcal{E} \to \mathcal{E}'$ that is compatible with the level structures, i.e., $\phi_*\eta = \eta'$.

\noindent The action of $\tilde{g}$ is defined locally. On each formal disk $\mathbb{D}_{x_i}$, the element $g_i \in L_{\theta_i}$ modifies the local trivialization of the torsor, which in turn defines the new torsor $\mathcal{E}'$ and level structure $\eta'$ via gluing. The condition that $\tilde{g}$ preserves the isomorphism class of $(\mathcal{E}, \eta)$ means that the newly constructed torsor $(\mathcal{E}', \eta')$ is isomorphic to the original one. This isomorphism $\phi: \mathcal{E} \to \mathcal{E}' \cong \mathcal{E}$ is an automorphism of the parahoric $\mathcal{G}_{\boldsymbol{\theta}}$-torsor $\mathcal{E}$.

\noindent  Because $\mathcal{E}$ is regularly stable, all of its global automorphisms 
lie in the center $Z \subset G$ by definition.

\noindent If $\phi\in Z$, then on  
$\prod_{i=1}^s L_{\theta_i}$, the element $\tilde{g}$ is also in $Z$ (viewed diagonally).  
Thus $g\in G_D$ is the identity element of $G_D$.  
Therefore any $g\in G_D$ fixing a point of $\mathcal{U}(X,\mathcal{G}_{\boldsymbol \theta})$ is the identity.  
This proves the action of $G_D$ on $\mathcal{U}^{rs}(X,\mathcal{G}_{\boldsymbol \theta})$ is \emph{free}. 
\end{proof}

\subsection{\texorpdfstring{Poisson Action of $G_D$ and the Moment Map}{Poisson Action of GD and the Moment Map}}

We now prove that the natural action of \(G_{D}\) on the moduli space
$\mathcal{U}(X,\mathcal{G}_{\boldsymbol \theta})$ extends to a Poisson action on its cotangent bundle \(T^*\mathcal{U}(X,\mathcal{G}_{\boldsymbol \theta})\).  In particular, we identify a \emph{canonical moment map}, showing explicitly that it arises from a (co)residue pairing when viewed through Serre duality.

We can perform a Hamiltonian lifting of the \(G_D\)-action from \(\mathcal{U}(X,\mathcal{G}_{\boldsymbol \theta})\) to \(T^*\mathcal{U}(X,\mathcal{G}_{\boldsymbol \theta})\) as in Example \ref{exm:Ham_lift}. 
Hence we will get a Poisson \(G_D\)-action on the cotangent immediately once we exhibit a moment map
\[
  \mu : T^*\mathcal{U}(X,\mathcal{G}_{\boldsymbol \theta}) \;\longrightarrow\;\mathfrak{g}_D^*.
\]
In this direction, let \(\mathcal{E}\) be a parahoric torsor and consider the short exact sequence
\[
0\to \mathcal{E}(\mathfrak{g})_\eta\longrightarrow \mathcal{E}(\mathfrak{g})\stackrel{\pi}{\longrightarrow}\bigoplus_{x_j\in D}\mathcal{E}(\mathfrak{g}_{\theta_{x_j}}/\mathfrak{g}_{\theta_{x_j}}^+)\to 0.
\]
Taking cohomology yields
\[
0\to H^0\bigl(X,\mathcal{E}(\mathfrak{g})_\eta\bigr)\longrightarrow H^0\bigl(X,\mathcal{E}(\mathfrak{g})\bigr)
\longrightarrow \bigoplus_{x_j\in D}H^0\Bigl(\mathbb{D}_{x_j},\mathcal{E}(\mathfrak{g}_{\theta_{x_j}}/\mathfrak{g}_{\theta_{x_j}}^+)\Bigr)
\]
\[
\longrightarrow H^1\bigl(X,\mathcal{E}(\mathfrak{g})_\eta\bigr)
\stackrel{H^1(\pi)}{\longrightarrow} H^1\bigl(X,\mathcal{E}(\mathfrak{g})\bigr)\to 0.
\]
By Serre duality, we have the commutative diagram:
\begin{center}
\begin{tikzcd}[column sep=6em]
H^1\bigl(X,\mathcal{E}(\mathfrak{g})_\eta\bigr)
	\arrow[r, "H^1(\pi)"]
	\arrow[d, "\cong"']
& H^1\bigl(X,\mathcal{E}(\mathfrak{g})\bigr)
	\arrow[d, "\cong"]\\[1mm]
H^0\Bigl(X,\mathcal{E}(\mathfrak{g})^*\otimes K_X(D)\Bigr)^*
	\arrow[r, "{(H^0(\pi\otimes\mathrm{id}_{K_X}))^*}", shorten <=1em, shorten >=1em]
& H^0\Bigl(X,\mathcal{E}(\mathfrak{g})^*\otimes K_X\Bigr)^*.
\end{tikzcd}
\end{center}
Thus, one obtains canonical isomorphisms
\[
\Biggl[\frac{\bigoplus_{x_j\in D}H^0\bigl(\mathbb{D}_{x_j},\mathcal{E}(\mathfrak{g}_{\theta_{x_j}}/\mathfrak{g}_{\theta_{x_j}}^+)\bigr)}{H^0\bigl(X,\mathcal{E}(\mathfrak{g})\bigr)}\Biggr]
\cong\mathrm{Ker}\Bigl[H^0\bigl(\pi\otimes\operatorname{id}_{K_X}\bigr)^*\Bigr].
\]
Now, identifying the cokernel of \(H^0(\pi\otimes\operatorname{id}_{K_X})\) with the dual of the kernel above, we define a homomorphism
\[
\mu_{\mathcal{E}}: H^0\Bigl(X,\mathcal{E}(\mathfrak{g})^*\otimes K_X(D)\Bigr)
\longrightarrow \Biggl[\frac{\bigoplus_{x_j\in D}H^0\Bigl(\mathbb{D}_{x_j},\mathcal{E}(\mathfrak{g}_{\theta_{x_j}}/\mathfrak{g}_{\theta_{x_j}}^+)\Bigr)}{H^0\bigl(X,\mathcal{E}(\mathfrak{g})\bigr)}\Biggr]^*.
\]
More precisely, after composing with the natural projection and injection, one obtains
\[
\mu_{\mathcal{E}}:\; H^0\Bigl(X,\mathcal{E}(\mathfrak{g})^*\otimes K_X(D)\Bigr)
\longrightarrow \Biggl[\frac{\bigoplus_{x\in D}H^0\bigl(x,\mathcal{E}(\mathfrak{g}_{\theta_x}/\mathfrak{g}_{\theta_x}^+)\bigr)}{\mathbb{C}}\Biggr]^*.
\]
We thus introduce the following:
\begin{defn}\label{defn:moment_map}
Define the moment map
\[
\mu:T^*\mathcal{U}(X,\mathcal{G}_{\boldsymbol{\theta}})\longrightarrow \mathfrak{g}_D^*
\]
by sending a point represented by \((\mathcal{E},\varphi,\eta)\) to
\[
\mu(\bigl[(\mathcal{E},\eta)\bigr],\varphi) = \eta\circ\bigl(\mu_{\mathcal{E}}(\varphi)\bigr)\circ \eta^{-1}.
\]
Since $\eta$ is a $D$-level structure, it identifies each fiber of $\mathcal{E}$ at $x_j\in D$ with $G_{\theta_{x_j}}$.  Hence we compose with the appropriate adjoint-conjugation $\eta \circ (-)\circ \eta^{-1}$ so that everything is \emph{intrinsically} defined in $\mathfrak{g}_D^*$, independent of the choice of local trivialization.  Symbolically, we get \(\eta\circ\mu_{\mathcal{E}}(\varphi)\circ\eta^{-1}\).
Notice also that although the level structure \(\eta\) is defined only up to the natural \(\operatorname{Aut}(\mathcal{E})\)-action, the composition \(\eta\circ\mu_{\mathcal{E}}(\varphi)\circ\eta^{-1}\) is well-defined.
\end{defn}

The homomorphism \(\mu_{\mathcal{E}}\) above is induced by the natural pairing
\begin{equation}\label{rem:residue_pairing}
H^0\Bigl(X,\mathcal{E}(\mathfrak{g})^*\otimes K_X(D)\Bigr)
\otimes \bigoplus_{x\in D}H^0\Bigl(\mathbb{D}_x,\mathcal{E}(\mathfrak{g}_{\theta_x}/\mathfrak{g}_{\theta_x}^+)\Bigr)
\longrightarrow H^0\bigl(X,K_X(D)\bigr)
\stackrel{\operatorname{Res}}{\longrightarrow}\mathbb{C}.
\end{equation}
This pairing provides the link to the local structure. At a point $x_j \in D$, let $\phi_j$ be a local representative of the Higgs field $\varphi$. The Lie algebra $\mathfrak{g}_{\theta_j}$ has the Levi decomposition $\mathfrak{g}_{\theta_j} = \hat{\mathfrak{l}}_{\theta_j} \oplus \mathfrak{g}_{\theta_j}^{+}$, which is an orthogonal direct sum with respect to the Killing form, which we denote by $(\cdot, \cdot)$. The term $\mathcal{E}(\mathfrak{g}_{\theta_j})/\mathcal{E}(\mathfrak{g}_{\theta_j}^+)$ is the bundle associated to the Levi quotient $\mathfrak{l}_{\theta_j} \cong \hat{\mathfrak{l}}_{\theta_j}$. When pairing $\phi_j$ with an element $Y_j \in \hat{\mathfrak{l}}_{\theta_j}$, the component of $\phi_j$ in $\mathfrak{g}_{\theta_j}^{+}$ gives a zero contribution. Thus, the pairing only sees the projection of $\phi_j$ onto $\hat{\mathfrak{l}}_{\theta_j}$. We call this projection the \emph{residue}, $\operatorname{Res}_{x_j}(\varphi)$. The pairing $(\phi_j, Y_j)$ becomes $(\operatorname{Res}_{x_j}(\varphi), Y_j)$. 

The resulting functional on $\hat{\mathfrak{l}}_{\theta_j}$ is the \emph{coresidue} of $\varphi$ at $x_j$, an element of $\hat{\mathfrak{l}}_{\theta_j}^*$. The moment map $\mu$ is therefore the collection of these coresidues over all $x_j \in D$.

\begin{thm}\label{thm:PoissonGD}
The group \(G_{D}\) acts \emph{Poisson} on \(T^*\mathcal{U}(X,\mathcal{G}_{\boldsymbol \theta})\).  Moreover, the \emph{canonical} moment map
\[
  \mu \colon T^*\mathcal{U}(X,\mathcal{G}_{\boldsymbol \theta})\;\longrightarrow\;\mathfrak{g}_D^*
\]
is given by dualizing the infinitesimal action and can be explicitly described via coresidues at the divisor \(D\). Its image is the element of $\mathfrak{g}_D^* = \bigoplus_{j=1}^s \hat{\mathfrak{l}}_{\theta_j}^*$ given by the direct sum of the coresidues at each point $x_j \in D$:
\[
\mu([(\mathcal{E},\eta)],\varphi) = \bigoplus_{j=1}^s \operatorname{CoRes}_{x_j}(\varphi).
\]
Explicitly, for any element $Y = (Y_1, \dots, Y_s) \in \mathfrak{g}_D = \bigoplus_{j=1}^s \hat{\mathfrak{l}}_{\theta_j}$, the pairing is given by the sum of Killing form pairings over the divisor $D$:
\[
\langle \mu([(\mathcal{E},\eta)],\varphi), Y \rangle = \sum_{j=1}^s (\operatorname{Res}_{x_j}(\varphi), Y_j).
\]
\end{thm}
\begin{proof}
Let us denote here for convenience \(U := \mathcal{U}(X,\mathcal{G}_{\boldsymbol \theta})\). For a stable parahoric $\mathcal{G}_{\boldsymbol \theta}$-torsor \(\mathcal{E}\) with \(D\)-level structure \(\eta\), we have an isomorphism
$T_{[(\mathcal{E}, \eta)]}U \cong H^1(X, \mathcal{E}(\mathfrak{g})_\eta)$,
where \(\mathcal{E}(\mathfrak{g})_\eta\subset \mathcal{E}(\mathfrak{g})\) denotes the subsheaf consisting of local adjoint-valued sections that preserve the $D$-level structure.

\noindent Dualizing and invoking Serre duality (together with the twist by \(D\) at the points supporting the parahoric structure), one obtains
$T^*_{[(\mathcal{E}, \eta)]}U \cong H^1(X, \mathcal{E}(\mathfrak{g})_\eta)^* \simeq H^0(X, \mathcal{E}(\mathfrak{g})\otimes K(D))$.
Hence a point of \(T^*U\) can be represented by a \emph{logarithmic Higgs field} $\varphi$.

\noindent To identify \(\mu\) explicitly, we must describe how \(\mathfrak{g}_D\) acts infinitesimally on \(U\). Let
\[
  \sigma \colon \mathfrak{g}_D \;\longrightarrow\; T_{[(\mathcal{E}, \eta)]}U
  \;\cong\;
  H^1(X, \mathcal{E}(\mathfrak{g})_\eta)
\]
be the differential of the map \(\rho_{[(\mathcal{E}, \eta)]} \colon G_D \to U\) sending \(g\mapsto g\cdot [(\mathcal{E},\eta)]\). Concretely, each element of \(\mathfrak{g}_D\) corresponds to an infinitesimal transformation of the \(D\)-level structure at each parahoric point, and these local transformations glue to give a global 1-cocycle in \(\mathcal{E}(\mathfrak{g})_\eta\).

\noindent By general principles of Hamiltonian actions on cotangent bundles, the moment map \(\mu\) of Definition \ref{defn:moment_map} is precisely the dual of \(\sigma\). For a cotangent vector $([(\mathcal{E},\eta)], \varphi)$ and an element $X \in \mathfrak{g}_D$, we have $\mu([(\mathcal{E}, \eta)], \varphi)(X) = \langle \varphi, \sigma(X)\rangle$. As established in the discussion following (\ref{rem:residue_pairing}), this pairing precisely computes the coresidue.

\noindent To show equivariance, we must prove that for $g \in G_D$ and $\varphi' = \operatorname{Ad}(g)\varphi$, the moment map transforms via the coadjoint action, i.e., $\mu(\varphi') = \operatorname{Ad}^*(g)\mu(\varphi)$. It is sufficient to verify this at each point $x_j \in D$. Let us denote the duality pairing between $\hat{\mathfrak{l}}_{\theta_j}^*$ and $\hat{\mathfrak{l}}_{\theta_j}$ by $\langle \cdot, \cdot \rangle$, and the Killing form on $\hat{\mathfrak{l}}_{\theta_j}$ by $(\cdot, \cdot)$. We need to show that $\operatorname{CoRes}_{x_j}(\varphi') = \operatorname{Ad}^*(g_j)\operatorname{CoRes}_{x_j}(\varphi)$, but this is immediate by the definition of coresidue and its action.

\noindent Since we have shown that $\langle \operatorname{CoRes}_{x_j}(\varphi'), Y_j \rangle = \langle \operatorname{Ad}^*(g_j)\operatorname{CoRes}_{x_j}(\varphi), Y_j \rangle$, for all $Y_j \in \hat{\mathfrak{l}}_{\theta_j}$, we conclude that $\operatorname{CoRes}_{x_j}(\varphi') = \operatorname{Ad}^*(g_j)\operatorname{CoRes}_{x_j}(\varphi)$. The total moment map transforms accordingly, proving the theorem.
\end{proof}

\subsection{The Infinitesimal Deformation Complex}

Let \(X\) be a smooth complex  projective curve.
In \cite{KSZparh} we have constructed the moduli space $\mathcal{M}_H(X, \mathcal{G}_{\boldsymbol \theta})$  of $R$-semistable logahoric $\mathcal{G}_{\boldsymbol \theta}$-Higgs torsors $(\mathcal{E}, \varphi)$ over $X$. We now consider isomorphism classes of triples $(\mathcal{E}, \varphi, \eta)$, where
\begin{itemize}
    \item \(\mathcal{E}\) is a stable parahoric \(\mathcal{G}_{\boldsymbol{\theta}}\)-torsor over \(X\),
    \item \(\varphi\in H^0\bigl(X,\mathcal{E}(\mathfrak{g})\otimes K(D)\bigr)\) is a logarithmic Higgs field, and 
    \item \(\eta\) is a \(D\)-level structure on \(\mathcal{E}\).
\end{itemize}
One can construct a coarse moduli space \(M_{LH}(X, \mathcal{G}_{\boldsymbol \theta})\) of triples $(\mathcal{E}, \varphi, \eta)$ as above using a similar approach as the one in Section \ref{sec:moduli_space_U_constr} for constructing the moduli space $\mathcal{U}(X, \mathcal{G}_{\boldsymbol \theta})$. In particular, this involves the correspondence between pairs $(\mathcal{E}, \varphi)$ over $X$ and equivariant logarithmic $G$-Higgs bundles on a Galois cover of $X$; see \cite[Section 3]{KSZparh} for this correspondence. Then, equipping the pairs $(\mathcal{E}, \varphi)$ with a $D$-level structure not  affecting the stability condition of the logahoric $\mathcal{G}_{\boldsymbol \theta}$-Higgs torsor allows us to construct a good quotient as in Section \ref{sec:moduli_space_U_constr}. We call the moduli space \(M_{LH}(X, \mathcal{G}_{\boldsymbol \theta})\), the \emph{moduli space of leveled logahoric $\mathcal{G}_{\boldsymbol \theta}$-Higgs torsors} over $X$. Note that then the cotangent \(T^*\mathcal{U}_X(G,D)\) of the moduli space \(\mathcal{U}_X(G,D)\) is an open subset of this moduli space \(M_{LH}(X, \mathcal{G}_{\boldsymbol \theta})\).

We would like to consider the deformation theory for this moduli space. Given a logarithmic Higgs field
\[
\varphi\in H^0\bigl(X,\mathcal{E}(\mathfrak{g})\otimes K(D)\bigr),
\]
we define the \emph{adjoint action}
\[
\operatorname{ad}\varphi: \mathcal{E}(\mathfrak{g})\to \mathcal{E}(\mathfrak{g})\otimes K(D),\quad \text{ with }\operatorname{ad}\varphi(\psi)=\varphi\circ\psi-\psi\circ\varphi.
\]
Then, a natural two-term complex governing the deformations of the triple \((\mathcal{E},\varphi,\eta)\) is
\[
\mathcal{K}_{\mathcal{E},\varphi,\eta}\colon\quad \mathcal{E}(\mathfrak{g}) \xrightarrow{\partial_{\varphi,\eta}} \bigl(\mathcal{E}(\mathfrak{g})\otimes K(D)\bigr)\oplus \bigoplus_{x_j\in D} \left[\mathcal{E}(\mathfrak{g}_{\theta_{x_j}}/\mathfrak{g}_{\theta_{x_j}}^+)\right],
\]
where the differential is given by
\[
\partial_{\varphi,\eta}(\psi) = \operatorname{ad}\varphi(\psi) \oplus  m_\eta(\psi).
\]
Here, the term \(m_\eta(\psi)\) arises from the level structure. In practice, the level structure \(\eta\) defines a reduction of the fibers of \(\mathcal{E}(\mathfrak{g})\) over \(D\) (via, say, a choice of splitting or a canonical quotient) and induces a homomorphism
\begin{equation}\label{dfn:mh}
m_\eta\colon \mathcal{E}(\mathfrak{g})\to \bigoplus_{x_j\in D} \left[\mathcal{E}(\mathfrak{g}_{\theta_{x_j}}/\mathfrak{g}_{\theta_{x_j}}^+)\right].
\end{equation}
We will prove in Lemma \ref{quasi-isom} later on that this complex is quasi-isomorphic to the complex 
\[
\mathcal{K}_{\mathcal{E},\varphi}\colon\quad \mathcal{E}(\mathfrak{g}^+) \xrightarrow{\partial} \mathcal{E}(\mathfrak{g})\otimes K(D),
\]
by the mapping
$\partial(\varphi)=\operatorname{ad}(\varphi)$, for the Lie algebra bundle $\mathcal{E}(\mathfrak{g}^+)$ defined by the gluing of the Lie algebra bundles $\mathcal{E}(\mathfrak{g}^+_{\theta_{x_j}})$ from Section \ref{sec:parh_Lie_algebra}.

To prove that this complex is indeed the deformation, we perform a local calculation. Choose a \v{C}ech covering $\{U_{\alpha}\}$ of $X$ and let $\mathcal{W}':=\{W_{\alpha}\}$ be the covering of $S \times X$, for $W_{\alpha}=S \times U_{\alpha}$, where $S=\mathrm{Spec}(\mathbb{C}[\epsilon]/(\epsilon^2))$ is an infinitesimal family. We have the following:
\begin{lem}
Fix a logahoric $\mathcal{G}_{\boldsymbol \theta}$-Higgs torsor \((\mathcal{E}_0, \varphi_0)\) with a \(D\)-level structure \(\eta_0\). A cochain 
\[
(\hat{f},\hat{\varphi},\hat{\eta}):=((\hat{f}_{\alpha\beta}),\, \hat{\varphi}_\alpha,\, \hat{\eta}_\alpha)
\]
in \(C^0(\mathcal{W}', \mathcal{K}_{\mathcal{E},\varphi,\eta})\) is a cocycle if and only if:
\begin{enumerate}
    \item The cochain \(\hat{f} = (\hat{f}_{\alpha\beta})\) is a cocycle in \(Z^1(\mathcal{U}, \mathcal{E}_0(\mathfrak{g}))\), that is, it defines an infinitesimal deformation of the parahoric $\mathcal{G}_{\boldsymbol \theta}$-torsor \(\mathcal{E}_0\).
    \item The perturbed Higgs field
    \[
    \varphi_0 + \epsilon\cdot \hat{\varphi} \in C^0\Bigl(\mathcal{W}',\, p^*\bigl(\mathcal{E}_0(\mathfrak{g})\otimes K(D)\bigr)\Bigr)
    \]
    is a global section of \(\mathcal{E}(\mathfrak{g})\otimes p^*{K(D)}\), where \(\mathcal{E}\) is the infinitesimal family of parahoric $\mathcal{G}_{\boldsymbol \theta}$-torsors over \(S\times X\) defined by the new gluing transformations 
    \[
    f_{\alpha\beta} = \operatorname{id} + \epsilon \cdot \hat{f}_{\alpha\beta}
    \]
    for \(p^*\mathcal{E}_0\), where $p$ is the projection on $X$.
    \item The perturbed level structure
    \[
    \eta_0 + \epsilon\cdot \hat{\eta} \in C^0\Bigl(\mathcal{W}',\, p^*\Bigl(\bigoplus_{x_j\in D}\Bigl[\mathcal{E}(\mathfrak{g}_{\theta_{x_j}}/\mathfrak{g}_{\theta_{x_j}}^+)\Bigr]\Bigr)\Bigr)
    \]
    is a global section \(\eta\) of the sheaf defining the \(D'\)-level structure (with the fibers \(V_{x_j}\) obtained from the natural quotients associated to the parahoric groups at \(x_j\in D'\)).
\end{enumerate}
Moreover, two cocycles 
\[
((\hat{f}_{\alpha\beta}),\, \hat{\varphi}_\alpha,\, \hat{\eta}_\alpha) \quad\text{and}\quad ((\hat{f}'_{\alpha\beta}),\, \hat{\varphi}'_\alpha,\, \hat{\eta}'_\alpha)
\]
represent the same hypercohomology class if and only if the corresponding infinitesimal families 
\[
(\mathcal{E},\varphi,\eta)\quad \text{and}\quad (\mathcal{E}',\varphi',\eta')
\]
are isomorphic.
\end{lem}

\begin{proof}
A cochain \((\hat{f},\hat{\varphi},\hat{\eta})\) in the complex \(\mathcal{K}_{\mathcal{E},\varphi,\eta}\) is a cocycle if and only if the following conditions hold:
\begin{enumerate}
    \item The \v{C}ech differential \(\delta(\hat{f})\) vanishes, i.e., 
    \[
    \delta(\hat{f})=0,
    \]
    so that \(\hat{f}\) defines an infinitesimal deformation of the parahoric torsor (more precisely, of its adjoint bundle) via the usual cocycle condition in \(Z^1(\mathcal{U}, \mathcal{E}_0(\mathfrak{g}))\).
    \item The compatibility between the deformed Higgs field and the new gluing is expressed by
    \[
    -\operatorname{ad}_{\varphi_0}(\hat{f}) = \delta(\hat{\varphi}),
    \]
    which ensures that the perturbed Higgs field \(\varphi_0+\epsilon\hat{\varphi}\) is well--defined on the deformed torsor. To see this, note that on triple overlaps one must have 
    \[
    f_{\beta\gamma}\varphi_0 f_{\alpha\beta}^{-1} = \varphi_\alpha,
    \]
    and writing 
    \[
    f_{\alpha\beta} = \operatorname{id} + \epsilon\hat{f}_{\alpha\beta},\quad \varphi_\alpha = \varphi_0 + \epsilon\hat{\varphi}_\alpha,
    \]
    one checks (to first order in \(\epsilon\)) that the condition is equivalent to 
    \[
    -\operatorname{ad}_{\varphi_0}(\hat{f}_{\alpha\beta}) = \hat{\varphi}_\alpha - \hat{\varphi}_\beta,
    \]
    that is, \(\delta(\hat{\varphi}) = -\operatorname{ad}_{\varphi_0}(\hat{f})\).
    \item Similarly, the compatibility of the deformed level structure is encoded by the equation
    \[
    -\eta_0\circ \hat{f} = \delta(\hat{\eta}).
    \]
    Indeed, on overlaps the new level structure must satisfy 
    \[
    \eta_z = (\eta_0+\epsilon\hat{\eta})\circ (\operatorname{id}+\epsilon\hat{f}) = \eta_0 + \epsilon\Bigl(\hat{\eta} + \eta_0\circ \hat{f}\Bigr),
    \]
    so that the cocycle condition for \(\hat{\eta}\) becomes \(\delta(\hat{\eta}) = -\eta_0\circ \hat{f}\).
\end{enumerate}
Finally, one verifies that a cocycle \((\hat{f},\hat{\varphi},\hat{\eta})\) represents a trivial deformation (i.e. a coboundary) if and only if there exists a 0-cochain \(\hat{g} = (\hat{g}_\alpha)\) in \(C^0(\mathcal{U}, \mathcal{E}_0(\mathfrak{g}))\) such that:
\begin{enumerate}[label=\alph*)]
    \item \(\hat{f}_{\alpha\beta} = \hat{g}_\alpha - \hat{g}_\beta\),
    \item \(\hat{\varphi}_\alpha = -\operatorname{ad}_{\varphi_0}(\hat{g}_\alpha)\),
    \item \(\hat{\eta}_\alpha = -\eta_0\circ \hat{g}_\alpha\).
\end{enumerate}
This precisely corresponds to the statement that the infinitesimal families \((\mathcal{E},\varphi,\eta)\) and \((\mathcal{E}',\varphi',\eta')\) are isomorphic. 
\end{proof}

This leads to the following characterization:
\begin{prop}
Let \((\mathcal{E},\varphi,\eta)\) be a point in \(\mathcal{M}_{LH}(X,\mathcal{G}_\theta)\). Then, \((\mathcal{E},\varphi,\eta)\) determines a canonical isomorphism
\[
T_{(\mathcal{E},\varphi,\eta)} \mathcal{M}_{LH}(X,\mathcal{G}_\theta) \cong \mathbb{H}^1\bigl(\mathcal{K}_{\mathcal{E},\varphi,\eta}\bigr)
\]
between the Zariski tangent space at \((\mathcal{E},\varphi,\eta)\) and the hypercohomology of the complex \(\mathcal{K}_{\mathcal{E},\varphi,\eta}\).
\end{prop}

We finally get a relationship between the complexes \(\mathcal{K}_{\mathcal{E},\varphi}\) and \(\mathcal{K}_{\mathcal{E},\varphi,\eta}\):

\begin{lem}\label{quasi-isom}
Let \([(\mathcal{E}, \varphi, \eta)]_S\) be a flat family of leveled parahoric $\mathcal{G}_{\boldsymbol \theta}$-Higgs torsors parametrized by a scheme \(S\). There exists a canonical isomorphism
\[
R^i_{{P_{S^*}}} \bigl(\mathcal{K}_{\mathcal{E}, \varphi, \eta}\bigr) \cong R^i_{{P_{S^*}}} \bigl(\mathcal{K}_{\mathcal{E}, \varphi}\bigr)
\]
for \(i =0,1\), induced by a canonical quasi-isomorphism.
\end{lem}

\begin{proof}
We define a morphism of complexes
\[
qi\colon \mathcal{K}_{\mathcal{E},\varphi,\eta} \longrightarrow \mathcal{K}_{\mathcal{E},\varphi}
\]
by specifying its components in degrees \(0\) and \(1\).

\medskip

\noindent\textbf{Degree 0:}  
The level structure \(\eta\) forces the reduction $\mathcal{E}(\mathfrak{g})$ to lie in the subbundle \(\mathcal{E}(\mathfrak{g}^+)\). Therefore, we define
\[
qi_0\colon   \mathcal{E}(\mathfrak{g}) \rightarrow\mathcal{E}(\mathfrak{g}^+)
\]
to be the natural inclusion.

\medskip

\noindent\textbf{Degree 1:}  
In degree \(1\) we have
\[
\mathcal{K}_{\mathcal{E},\varphi,\eta}^1 = \mathcal{E}(\mathfrak{g})\otimes K(D)\oplus \Biggl(\bigoplus_{x_j\in D} \Bigl[\mathcal{E}(\mathfrak{g}_{\theta_{x_j}}/\mathfrak{g}_{\theta_{x_j}}^+)\Bigr]\Biggr)
\]
and
\[
\mathcal{K}_{\mathcal{E},\varphi}^1 = \mathcal{E}(\mathfrak{g})\otimes K(D).
\]
We define
\[
qi_1\colon \mathcal{E}(\mathfrak{g})\otimes K(D)\oplus \Biggl(\bigoplus_{x_j\in D} \Bigl[\mathcal{E}(\mathfrak{g}_{\theta_{x_j}}/\mathfrak{g}_{\theta_{x_j}}^+)\Bigr]\Biggr) \longrightarrow \mathcal{E}(\mathfrak{g})\otimes K(D)
\]
by projecting onto the first summand:
\[
qi_1(\alpha,\beta)=\alpha.
\]

\medskip

\noindent It is straightforward to check that \(qi=(qi_0,qi_1)\) defines a morphism of complexes.

\medskip

\noindent We now show that \(qi\) is a quasi-isomorphism, i.e., that it induces isomorphisms on the cohomology sheaves.

\medskip

\noindent \textbf{Degree 0:}  
The degree \(0\) cohomology of \(\mathcal{K}_{\mathcal{E},\varphi}\) is
\[
\mathcal{H}^0 \bigl(\mathcal{K}_{\mathcal{E},\varphi}\bigr) = \mathrm{Ker}\Bigl(\operatorname{ad}\varphi\colon \mathcal{E}(\mathfrak{g}^+)\to \mathcal{E}(\mathfrak{g})\otimes K(D)\Bigr).
\]
For the complex \(\mathcal{K}_{\mathcal{E},\varphi,\eta}\) we have
\[
\mathcal{H}^0 \bigl(\mathcal{K}_{\mathcal{E},\varphi,\eta}\bigr) = \mathrm{Ker}\Bigl(\operatorname{ad}\varphi\oplus m_{\eta}\colon \mathcal{E}(\mathfrak{g})\to \bigl(\mathcal{E}(\mathfrak{g})\otimes K(D)\bigr)\oplus \bigoplus_{x_j\in D} \Bigl[\mathcal{E}(\mathfrak{g}_{\theta_{x_j}}/\mathfrak{g}_{\theta_{x_j}}^+)\Bigr]\Bigr),
\]
for $m_{\eta}$ as defined in (\ref{dfn:mh}). Since the condition \(m_{\eta}(\psi)=0\) forces \(\psi\) to lie in \(\mathcal{E}(\mathfrak{g}^+)\), we deduce that
\[
\mathcal{H}^0 \bigl(\mathcal{K}_{\mathcal{E},\varphi,\eta}\bigr) = \mathrm{Ker}\Bigl(\operatorname{ad}\varphi\colon \mathcal{E}(\mathfrak{g}^+)\to \mathcal{E}(\mathfrak{g})\otimes K(D)\Bigr).
\]
Thus, \(qi_0\) induces an isomorphism on degree \(0\) cohomology.

\medskip

\noindent \textbf{Degree 1:}  
The degree \(1\) cohomology of \(\mathcal{K}_{\mathcal{E},\varphi}\) is
\[
\mathcal{H}^1 \bigl(\mathcal{K}_{\mathcal{E},\varphi}\bigr) = \operatorname{coker}\Bigl(\operatorname{ad}\varphi\colon \mathcal{E}(\mathfrak{g}^+)\to \mathcal{E}(\mathfrak{g})\otimes K(D)\Bigr).
\]
Similarly, one computes
\[
\mathcal{H}^1 \bigl(\mathcal{K}_{\mathcal{E},\varphi,\eta}\bigr) = \frac{\mathcal{E}(\mathfrak{g})\otimes K(D)\oplus \Biggl(\bigoplus_{x_j\in D} \Bigl[\mathcal{E}(\mathfrak{g}_{\theta_{x_j}}/\mathfrak{g}_{\theta_{x_j}}^+)\Bigr]\Biggr)}{\operatorname{Im}\Bigl(\operatorname{ad}\varphi\oplus m_{\eta}\Bigr)}.
\]
A local analysis (analogous to the vector bundle case) shows that the extra summand coming from the level structure exactly compensates for passing from \(\mathcal{E}(\mathfrak{g})\) to \(\mathcal{E}(\mathfrak{g}^+)\). Consequently, we obtain an isomorphism
\[
\mathcal{H}^1 \bigl(\mathcal{K}_{\mathcal{E},\varphi,\eta}\bigr) \cong \operatorname{coker}\Bigl(\operatorname{ad}\varphi\colon \mathcal{E}(\mathfrak{g}^+)\to \mathcal{E}(\mathfrak{g})\otimes K(D)\Bigr) = \mathcal{H}^1 \bigl(\mathcal{K}_{\mathcal{E},\varphi}\bigr).
\]
Since \(qi_1\) is the projection onto the first factor (acting as the identity on \(\mathcal{E}(\mathfrak{g})\otimes K(D)\)), it follows that \(qi_1\) induces an isomorphism on degree \(1\) cohomology.

\medskip

\noindent Thus, the morphism \(qi\) is a quasi-isomorphism. In particular, for each \(i=0,1\), we have the canonical isomorphism
\[
R^i_{{P_{S^*}}} \bigl(\mathcal{K}_{\mathcal{E},\varphi,\eta}\bigr) \cong R^i_{{P_{S^*}}} \bigl(\mathcal{K}_{\mathcal{E},\varphi}\bigr).
\]
This completes the proof.
\end{proof}

\subsection{Extension of Poisson Structure}
Finally, we would like to consider the Poisson structure on the moduli space of logahoric $\mathcal{G}_{\boldsymbol \theta}$-Higgs torsors.

Let \((E,\varphi)\) be a stable logahoric $\mathcal{G}_{\boldsymbol \theta}$-Higgs torsor in \(\mathcal{M}_H(X,\mathcal{G}_\theta)\). Its tangent space is naturally identified with the first hypercohomology group of the two‐term complex
\[
\bar{\mathcal{K}}_{\mathcal{E},\varphi}:\mathcal{E}(\mathfrak{g}) \xrightarrow{\, [\,\cdot\,,\varphi]\,} \mathcal{E}(\mathfrak{g})\otimes K(D),
\]
where \(\mathcal{E}(\mathfrak{g})\) denotes the sheaf of endomorphisms of \(\mathcal{E}\) that preserve the parahoric structure, and \(K(D)\) is the canonical bundle twisted by the divisor \(D\). To describe a Poisson bracket we next consider the dual complex obtained by tensoring with \(K(D)\) and reversing the sign of the differential. This yields
\[
\bar{\mathcal{K}}_{\mathcal{E},\varphi}^\vee:\mathcal{E}(\mathfrak{g}^+) \xrightarrow{\, -[\cdot,\varphi]\,} \mathcal{E}(\mathfrak{g}^+)\otimes K(D).
\]
There is then a natural injection from this dual complex into the original one; that is, we have a commutative diagram
\[
\xymatrix{
\mathcal{E}(\mathfrak{g}^+)\ar[d]_{-[\cdot,\Phi]} \ar@{^{(}->}[r]^{\mathrm{id}} & \mathcal{E}(\mathfrak{g})\ar[d]^{[\cdot,\Phi]}\\
\mathcal{E}(\mathfrak{g}^+)\otimes K(D)\ar@{^{(}->}[r]^{-\,\mathrm{id}\otimes\mathrm{id}_{K(D)}} & \mathcal{E}(\mathfrak{g})\otimes K(D).
}
\]
Using Serre duality for hypercohomology, this inclusion produces an antisymmetric linear map
\begin{equation}\label{defn:sharp}
\sharp \;:\; T^*_{(\mathcal{E},\varphi)}\mathcal{M}_H(X,\mathcal{G}_\theta) \cong \mathbb{H}^1\Bigl(\bar{\mathcal{K}}_{\mathcal{E},\varphi}^\vee\Bigr) \longrightarrow \mathbb{H}^1\Bigl(\bar{\mathcal{K}}_{\mathcal{E},\varphi}\Bigr) \cong T_{(\mathcal{E},\varphi)}\mathcal{M}_H(X,\mathcal{G}_\theta).
\end{equation}
The specific choices of signs guarantee that \(\sharp\) is antisymmetric, and we will show that it indeed defines a Poisson structure on the moduli space \(\mathcal{M}_H(X,\mathcal{G}_\theta)\) inheriting the Poisson structure on \(\mathcal{M}_{LH}(X,\mathcal{G}_\theta)\). 

Recall that we have the following:

\begin{align*}
T_{(\mathcal{E},\varphi,\eta)}\mathcal{M}_{LH}(X,\mathcal{G}_\theta)&
\cong \mathcal{K}_{\mathcal{E},\varphi}\colon\quad \mathcal{E}(\mathfrak{g}^+) \xrightarrow{\, [\,\cdot\,,\varphi]\,} \mathcal{E}(\mathfrak{g})\otimes K(D),\\
T_{(\mathcal{E},\varphi,\eta)}^*\mathcal{M}_{LH}(X,\mathcal{G}_\theta)&
\cong \mathcal{K}^\vee_{\mathcal{E},\varphi}\colon\quad \mathcal{E}(\mathfrak{g}^+) \xrightarrow{\,- [\,\cdot\,,\varphi]\,} \mathcal{E}(\mathfrak{g})\otimes K(D),\\
T_{(\mathcal{E},\varphi)}\mathcal{M}_{H}(X,\mathcal{G}_\theta)&
\cong\bar{\mathcal{K}}_{\mathcal{E},\varphi}\colon\quad\mathcal{E}(\mathfrak{g}) \xrightarrow{\, [\,\cdot\,,\varphi]\,} \mathcal{E}(\mathfrak{g})\otimes K(D),\\
T_{(\mathcal{E},\varphi)}^*\mathcal{M}_{H}(X,\mathcal{G}_\theta)&
\cong\bar{\mathcal{K}}^\vee_{\mathcal{E},\varphi}\colon\quad\mathcal{E}(\mathfrak{g^+}) \xrightarrow{\, [\,\cdot\,,\varphi]\,} \mathcal{E}(\mathfrak{g^+})\otimes K(D).
\end{align*}
Then, the next theorem provides the Poisson structure on $\mathcal{M}_H(X, \mathcal{G}_{\boldsymbol \theta})$:
\begin{thm}\label{thm:forgetful_Poisson}
The forgetful map
\[
\ell: \mathcal{M}_{LH}(X,\mathcal{G}_\theta)\longrightarrow \mathcal{M}_H(X,\mathcal{G}_\theta)
\]
defined by forgetting the $D$-level structure is a Poisson map. The induced Poisson structure on $\mathcal{M}_H(X, \mathcal{G}_{\boldsymbol \theta})$ is given by the map $\sharp$ defined in (\ref{defn:sharp}).
\end{thm}

\begin{proof}
Let 
\[
\mathbf{z}:=(\mathcal{E},\varphi,\eta)\in \mathcal{M}_{LH}(X,\mathcal{G}_\theta)
\]
be a leveled logahoric $\mathcal{G}_{\boldsymbol \theta}$-Higgs torsor over $X$ and set
\[
\mathbf{y}:=\ell(\mathbf{z})=(\mathcal{E},\varphi)\in \mathcal{M}_H(X,\mathcal{G}_\theta).
\]
The tangent spaces are given by
\[
T_{\mathbf{z}}\mathcal{M}_{LH}(X,\mathcal{G}_\theta)\cong \mathbb{H}^1\Bigl(\bar{\mathcal{K}}_{E,\phi,\varphi}\Bigr)
\]
and
\[
T_{\mathbf{y}}\mathcal{M}_H(X,\mathcal{G}_\theta)\cong \mathbb{H}^1\Bigl(\bar{\mathcal{K}}_{\mathcal{E},\varphi}\Bigr).
\]
Since the $D$-level structure is forgotten by \(\ell\), the differential
\[
d\ell(\mathbf{z}):T_{\mathbf{z}}\mathcal{M}_{LH}(X,\mathcal{G}_\theta)\longrightarrow T_{\mathbf{y}}\mathcal{M}_H(X,\mathcal{G}_\theta)
\]
is induced by the natural inclusion of complexes
\[
\mathcal{K}_{\mathcal{E},\varphi}\ \hookrightarrow\ \bar{\mathcal{K}}_{\mathcal{E},\varphi}.
\]

\noindent By Serre duality in hypercohomology, the cotangent spaces are identified:
\[
T^*_{\mathbf{z}}\mathcal{M}_{LH}(X,\mathcal{G}_\theta)\cong \mathbb{H}^1\Bigl(\mathcal{K}_{\mathcal{E},\varphi}^\vee\Bigr) \quad \text{and} \quad 
T^*_{\mathbf{y}}\mathcal{M}_H(X,\mathcal{G}_\theta)\cong \mathbb{H}^1\Bigl(\bar{\mathcal{K}}_{\mathcal{E},\varphi}^\vee\Bigr).
\]
The dual differential
\[
d\ell(\mathbf{z})^*: T^*_{\mathbf{y}}\mathcal{M}_H(X,\mathcal{G}_\theta)\longrightarrow T^*_{\mathbf{z}}\mathcal{M}_{LH}(X,\mathcal{G}_\theta)
\]
arises from the inclusion
\[
\bar{\mathcal{K}}_{\mathcal{E},\varphi}^\vee\ \hookrightarrow\ \bar{\mathcal{K}}_{\mathcal{E},\varphi}^\vee.
\]

\noindent On \(\mathcal{M}_{LH}(X,\mathcal{G}_\theta)\) the symplectic form \(\omega\) induces the isomorphism
\[
\omega^{-1}_{(\mathcal{E},\varphi,\eta)}: T^*_{\mathbf{z}}\mathcal{M}_{LH}(X,\mathcal{G}_\theta)\stackrel{\sim}{\longrightarrow} T_{\mathbf{z}}\mathcal{M}_{LH}(X,\mathcal{G}_\theta).
\]
Meanwhile, the map $\sharp$ on \(\mathcal{M}_H(X,\mathcal{G}_\theta)\) is 
\[
\sharp: T^*_{\mathbf{y}}\mathcal{M}_H(X,\mathcal{G}_\theta)\longrightarrow T_{\mathbf{y}}\mathcal{M}_H(X,\mathcal{G}_\theta)
\]
which is induced by the inclusion $\bar{\mathcal{K}}_{E,\varphi}^\vee\ \hookrightarrow\ \bar{\mathcal{K}}_{E,\varphi}$. 

\noindent In order to prove that \(\ell\) is Poisson, we must show that for every $w\in T^*_{\mathbf{y}}\mathcal{M}_H(X,\mathcal{G}_\theta)$, the identity
\begin{equation}\label{poisson_id}
d\ell(\mathbf{z})\circ \omega^{-1}_{(\mathcal{E},\varphi,\eta)}\circ d\ell(\mathbf{z})^*(w)=\sharp(w)
\end{equation}
holds or, equivalently, that the diagram
\[
\begin{tikzcd}
T^*_{\mathbf{y}}\mathcal{M}_H(X,\mathcal{G}_\theta) \arrow[r, "\sharp"] \arrow[d, "d\ell(\mathbf{z})^*"']
& T_{\mathbf{y}}\mathcal{M}_H(X,\mathcal{G}_\theta) \\
T^*_{\mathbf{z}}\mathcal{M}_{LH}(X,\mathcal{G}_\theta) \arrow[r, "\omega^{-1}_{(\mathcal{E},\varphi,\eta)}"'] 
& T_{\mathbf{z}}\mathcal{M}_{LH}(X,\mathcal{G}_\theta) \arrow[u, "d\ell(\mathbf{z})"']
\end{tikzcd}
\]
commutes.

\noindent By the functoriality of hypercohomology, the maps in the above diagram are induced by the natural inclusions of the deformation complexes (and their duals). In particular, the map \(\sharp\) is induced by the inclusion
\[
\bar{\mathcal{K}}_{\mathcal{E},\varphi}^\vee\hookrightarrow \bar{\mathcal{K}}_{\mathcal{E},\varphi},
\]
while the maps \(d\ell(\mathbf{z})\) and \(d\ell(\mathbf{z})^*\) come from the inclusions
\[
\mathcal{K}_{\mathcal{E},\varphi}\hookrightarrow \bar{\mathcal{K}}_{\mathcal{E},\varphi} \quad \text{and} \quad \bar{\mathcal{K}}_{\mathcal{E},\varphi}^\vee\hookrightarrow \mathcal{K}_{\mathcal{E},\varphi}^\vee,
\]
respectively. Consequently, the composite map
\[
d\ell(\mathbf{z})\circ \omega^{-1}_{(\mathcal{E},\varphi,\eta)}\circ d\ell(\mathbf{z})^*
\]
is precisely the hypercohomology map induced by the inclusion $\bar{\mathcal{K}}_{\mathcal{E},\varphi}^\vee\hookrightarrow \bar{\mathcal{K}}_{\mathcal{E},\varphi}$, which is exactly the map \(\sharp\). This establishes the identity \eqref{poisson_id}, and shows that \(\ell\) is a Poisson map and that $\sharp$ is Poisson on the open dense subset of $\mathcal{M}_{H}(X, \mathcal{G}_\theta)$. As so, the map $\sharp$ extends to a Poisson structure on all $\mathcal{M}_{H}(X, \mathcal{G}_\theta)$.
\end{proof}

\begin{rem}
Note that in \cite{KSZPoisson}, a Poisson structure on moduli
spaces of twisted stable $G$-Higgs bundles over stacky curves was obtained using an Atiyah sequence. However, that construction was not determining a canonical moment map on the cotangent bundle that could be used in order to describe the symplectic leaves of the Poisson manifold. 
\end{rem}

\vspace{4mm}

\section{Parahoric Hitchin Fibration and Abelianization} 

We next study the Hitchin fibration on the space of logahoric $\mathcal{G}_{\boldsymbol \theta}$-Higgs torsors. The main result here is that the generic fibers of this fibration are Lagrangian with respect to the symplectic leaves of the Poisson space $\mathcal{M}_H(X, \mathcal{G}_{\boldsymbol \theta})$.

\subsection{Hitchin Fibration}
Any invariant homogeneous degree $i$-polynomial naturally defines a map 
\[a_i : H^0(\mathcal{E}(\mathfrak{g})\otimes K_X (D)) \to H^0 (K(D)^i).\]
Equivalence classes of Higgs bundles on the moduli space $\mathcal{M}_H (X, \mathcal{G}_{\boldsymbol\theta})$ of logahoric $\mathcal{G}_{\boldsymbol\theta}$-Higgs torsors  are defined using the adjoint action of $G(K)$ on $\mathfrak{g}(K)$, the loop Lie algebra of $G(K)$, which is induced from the action of the complex reductive group $G$ on $\mathfrak{g}$. If the Lie algebra $\mathfrak{g}(K)$ has rank $l$ and $p_i$ are polynomials of degree $\text{deg}p_i =m_i +1$, for $i=1,...,l$, forming a basis of the algebra of invariant polynomials on the Lie algebra $\mathfrak{g}(K)$, then the corresponding maps $a_i$ combine to give a \textit{Hitchin fibration} 
\[h_{\boldsymbol\theta}: \mathcal{M}_H (X, \mathcal{G}_{\boldsymbol\theta}) \to \mathcal{B}_{\boldsymbol \theta},\]
defined by $h_{\boldsymbol\theta}(\mathcal{E},\varphi)=(p_1 (\varphi),...,p_l (\varphi))$. 
Note that as in the parabolic case, the map $h_{\boldsymbol{\theta}}$ is blind to the parahoric structure at each parahoric point in the divisor $D$, as it only depends on the Higgs field $\varphi \in H^0(\mathcal{E}(\mathfrak{g})\otimes K_X (D))$ and the line bundle $K_X (D)$. 

We set the following definition (cf. \cite{Nguyen}, \cite{Wang}):

\begin{defn}[Hitchin Base and its Image]
The (ambient) \textit{parahoric Hitchin base} is the affine space
\[
\mathcal{B}_{\boldsymbol{\theta}} = \bigoplus_{j=1}^{l} H^0(X, (K_X(D))^{m_j + 1}).
\]
The \textit{image of the Hitchin fibration}, denoted by $\mathcal{A}_{\boldsymbol{\theta}}$, is the subvariety \[\mathcal{A}_{\boldsymbol{\theta}} = h_{\boldsymbol{\theta}}(\mathcal{M}_H(X, \mathcal{G}_{\boldsymbol{\theta}})) \subset \mathcal{B}_{\boldsymbol{\theta}}.\]
\end{defn}

\subsection{Cameral covers} We now consider the construction of cameral covers of $X$ and the generalized Prym varieties adapting the original construction of Donagi for principal $G$-Higgs bundles to the case of logahoric $\mathcal{G}_{{\boldsymbol \theta}}$-Higgs torsors; see \cite[Section 2]{Donagi} for a survey in the principal bundle case. 

Let $T \subset G$ be a maximal torus of $G$.  Chevalley's theorem provides that the restriction map 
\[ \mathbb{C}[\mathfrak{g}]^{G} \to \mathbb{C}[\mathfrak{t}]^{W}\]
is an isomorphism from $\text{Ad}$-invariant polynomial functions on the Lie algebra $\mathfrak{g}$ to $W$-invariant polynomial functions on the Cartan subalgebra $\mathfrak{t} \subset \mathfrak{g}$. We then consider the injective ring homomorphism  
\[\mathbb{C}[\mathfrak{t}]^{W} \cong \mathbb{C}[\mathfrak{g}]^{G} \hookrightarrow \mathbb{C}[\mathfrak{g}]\]
and take the prime spectrum of the rings to define a surjective $G$-invariant morphism of affine varieties
\begin{equation}\label{G_inv_morph}
\mathfrak{g} \twoheadrightarrow  \mathfrak{t}/{W}.   
\end{equation}
Thus, taking fiber product with the quotient map $\mathfrak{t} \to \mathfrak{t}/W$, we get 
\begin{equation}\label{cam_cover_Lie_alg}
\tilde{\mathfrak{g}}:=\mathfrak{g} \times _{\mathfrak{t}/W} \mathfrak{t},    
\end{equation}
and the projection $\pi:\tilde{\mathfrak{g}} \to \mathfrak{g}$ is a finite $W$-Galois morphism. This is called the \emph{cameral cover} of the Lie algebra $\mathfrak{g}$. The fiber $\pi^{-1}(g)$ of a regular semisimple element $g \in \mathfrak{g}$ is identified with the set of chambers in $\mathfrak{t}^{*}$. 

Now let $(\mathcal{E}, \varphi)$ be a logahoric $\mathcal{G}_{{\boldsymbol \theta}}$-Higgs torsor over $X$. The Higgs field $\varphi$ is a holomorphic section 
of $\mathcal{E}(\mathfrak{g})\otimes {{K}_{X}}(D)$. Since the morphism (\ref{G_inv_morph}) is $G$-invariant and $\mathbb{C}^{*}$-equivariant, it can be extended to a morphism 
$\left| \mathcal{E}(\mathfrak{g})\otimes {{K}_{X}}(D) \right| \to \left| \mathfrak{t}\otimes {{K}_{X}}(D) \right|/W$, where $\left| \,\cdot \, \right|$ denotes here the total space. Forming the fiber product with $\mathfrak{t}\otimes {{K}_{X}}(D)$ as in (\ref{cam_cover_Lie_alg}) and pulling-back to $X$ by the Higgs field $\varphi$, we get the following definition analogously to \cite[Definition 2.6]{Donagi}:
\begin{defn}
The \emph{cameral cover} of $X$ determined by the logahoric $\mathcal{G}_{{\boldsymbol \theta}}$-Higgs torsor $(\mathcal{E}, \varphi)$ is defined by the projection $\pi :\tilde{X}\to X$, where
\[\tilde{X}={{\varphi }^{*}}\left( \left| \mathcal{E}(\mathfrak{g})\otimes {{K}_{X}}(D) \right|{{\times }_{\left| \mathfrak{t}\otimes {{K}_{X}}(D) \right|/W}}\left| \mathfrak{t}\otimes {{K}_{X}}(D) \right| \right)\]
and $\pi$ is the projection onto the first factor.
\end{defn}
The cameral cover $\tilde{X} \to X$ is a $W$-Galois cover which generically parameterizes the chambers determined by the Higgs field $\varphi$; we are pulling back by $\varphi$ the covers $\tilde{\mathfrak{g}} \to \mathfrak{g}$ to take covers of open subsets in $X$ over which the bundle is trivialized, and then we glue these covers together. The ramification of $\tilde{X}$ is determined by the order of the zeroes of $\varphi$ and $\tilde{X}$ is a closed subscheme of $\left| \mathcal{E}(\mathfrak{g})\otimes {{K}_{X}}(D) \right|$ that can be singular or non-reduced. The cameral cover inherits from $\left| \mathcal{E}(\mathfrak{g})\otimes {{K}_{X}}(D) \right|$ a $W$-action thus has lots of automorphisms. Generically, $\tilde{X}$ is a non-singular Galois $W$-cover with simple ramification.

\begin{defn}
We shall denote by $\mathcal{A}' \subset \mathcal{A}_{\boldsymbol \theta}$ the open and dense subspace of $\mathcal{A}_{\boldsymbol \theta}$ such that $\tilde{X}_s$ is smooth whenever $s\in \mathcal{A}'$. 
We call the fibers $h_{\boldsymbol \theta}^{-1}(s)$ for $s \in \mathcal{A}'$, the \emph{generic fibers} of the parahoric Hitchin fibration $h_{\boldsymbol \theta}$. 
\end{defn}

\subsection{Generalized Prym varieties} 
We have constructed the cameral cover $\tilde{X}$ as a $W$-Galois cover of $X$. Thus, there is an induced action of $\mathbb{Z}[W]$ on $\tilde{X}$, hence on $H_{*}(\tilde{X},\mathbb{Z})$ and on the Picard group $\text{Pic}(\tilde{X})$. 

\begin{defn}   
For an irreducible $\mathbb{Z}[W]$-module $\Lambda$, the \emph{generalized Prym variety} of $\tilde{X}$ is defined as the set of equivariant maps of the $\mathbb{Z}[W]$-module $\Lambda$ to $\text{Pic}\tilde{X}$
\[Prym_{\Lambda}(\tilde{X}):=\text{Hom}_W(\Lambda, \text{Pic}\tilde{X}).\]
\end{defn}
The generalized Prym variety $Prym_{\Lambda}(\tilde{X})$ is an algebraic group. For a generic point $s\in \mathcal{A}'$, the cameral cover $\tilde{X}$ is a smooth projective curve, in which case then $\text{Pic}\tilde{X}$ is an abelian variety, therefore $Prym_{\Lambda}(\tilde{X})$ is also abelian; we refer to \cite[Section 5]{Donagi} for further information. We thus have:

\begin{prop}
The generalized Prym varieties associated to the cameral cover $X_s$, for generic  $s\in \mathcal{A}'$, are abelian varieties. 
\end{prop}

\subsection{\texorpdfstring{The Hitchin fibration for \(\Gamma\)-equivariant \(G\)-Higgs bundles}{The Hitchin fibration for Gamma-equivariant G-Higgs bundles}}
\label{sec-pre-Y}

Let \(X\) be a smooth complex projective curve. Suppose that a second curve \(Y\) is equipped with an effective action of the finite abelian group
\[
\Gamma=\prod_{i=1}^{r}\Z/n_i\Z,
\]
so that there is a corresponding \(\Gamma\)-Galois cover
\[
\pi\colon Y\to X,\qquad X=Y/\Gamma.
\]
Equivalently, one may work on the root stack over \(X\). In what follows, we study the moduli space $\mathcal{M}^{\Gamma}_H(Y,G)$  of semistable \(\Gamma\)-equivariant  \(G\)-Higgs bundles on \(Y\), where \(G\) is a reductive algebraic group over \(\mathbb{C}\) with Lie algebra \(\mathfrak{g}\). In particular, 
\[
\mathcal{M}^{\Gamma}_H(Y,G)=\Bigl\{(P,s):\,P\mbox{ is a \(\Gamma\)-equivariant \(G\)-bundle on \(Y\), }\, s\in H^0\bigl(Y,\operatorname{ad}P\otimes K_Y(\tilde{D})\bigr)\Bigr\};
\]
we refer to \cite{KSZparh} for the construction of this moduli space. 

Let \(K_Y(\tilde{D})\) denote the canonical bundle on \(Y\) twisted by the parahoric divisor \(\tilde{D}\). By deformation theory and Serre duality, a point in the cotangent bundle $T^*\mathcal{M}^{\Gamma}_H(Y,G)$ is a pair \((P,s)\) with \(P\) a stable \(\Gamma\)-equivariant principal \(G\)-bundle on \(Y\) and
\[
s\in H^0\Bigl(Y,\,\operatorname{ad}P\otimes K_Y(\tilde{D})\Bigr)
\]
a (necessarily \(\Gamma\)-invariant) section.

Let \(h_1,\dots,h_k\) be homogeneous generators of the algebra of invariant polynomials on \(\mathfrak{g}\) with \(\deg h_i=:d_i\). Each \(h_i\) induces a map
\[
\mathcal{H}_i\colon \operatorname{ad}P\otimes K_Y(\tilde{D}) \to \Bigl(K_Y(\tilde{D})\Bigr)^{d_i},
\]
and consequently one defines the Hitchin fibration
\begin{equation}\label{Hitchin map and base_equiv}
\mathcal{H}\colon T^*\mathcal{M}^{\Gamma}_H(Y,G) \to \mathcal{K} := \bigoplus_{i=1}^{k} H^0\Bigl(Y,\Bigl(K_Y(\tilde{D})\Bigr)^{d_i}\Bigr)
\end{equation}
by
\[
(P,s) \mapsto \phi=(\phi_1,\dots,\phi_k),\quad \phi_i=\mathcal{H}_i(s).
\]

Fix a maximal torus \(T\subset G\) with associated root system
\(
\mathcal{R}=\mathcal{R}(G,T)
\)
and Weyl group
\(
W=N_G(T)/T.
\)
Choose a Borel subgroup \(B\supset T\); this determines a set of positive roots \(\mathcal{R}^+\subset \mathcal{R}\). Denote by \(\mathfrak{t}\) the Lie algebra of \(T\). Then the differential of each root \(\alpha\) gives a map
\[
d\alpha\colon \mathfrak{t}\otimes K_Y(\tilde{D}) \to K_Y(\tilde{D}).
\]
Moreover, restricting the invariant polynomials \(h_i\) to \(\mathfrak{t}\) yields homogeneous, \(W\)-invariant polynomials $\sigma_1,\dots,\sigma_k$,
which define the Galois covering
\[
\underline{\sigma}=(\sigma_1,\dots,\sigma_k)\colon \mathfrak{t}\otimes K_Y(\tilde{D}) \to \bigoplus_{i=1}^{k}\Bigl(K_Y(\tilde{D})\Bigr)^{d_i}.
\]
Its discriminant \(\Xi\) is the zero locus of the \(W\)-invariant function
\(\prod_{\alpha\in R} d\alpha.\)

For generic \(\phi\in \mathcal{K}\), we define the cameral cover of \(Y\) by
\[
\widetilde{Y}:=\phi^{*}\Bigl(\mathfrak{t}\otimes K_Y(\tilde{D})\Bigr).
\]
Then \(\widetilde{Y}\) is a ramified covering of \(Y\) with
\(
m=|W|
\)
sheets. Denote by \(Ram\subset Y\) the branch locus; by construction one has
\begin{equation}\label{eq:ramY}
\mathcal{O}(Ram) \cong \Bigl(K_Y(\tilde{D})\Bigr)^{|R|} \equiv \Bigl(K_Y(\tilde{D})\Bigr)^{\dim G-\operatorname{rank}G}.
\end{equation}
Let
\[
\iota\colon \widetilde{Y}\to \mathfrak{t}\otimes K_Y(\tilde{D})
\]
be the natural inclusion. By definition, for each \(w\in W\)
\[
\iota(w\,\eta)=\operatorname{Ad}(n_w)\,\iota(\eta),
\]
where \(n_w\in N_G(T)\) is any representative of \(w\). Moreover, if \(\pi\colon\widetilde{Y}\to Y\) is the projection, then for each \(\alpha\in R\) the composition
\(
d\alpha\circ\iota
\)
is a holomorphic section of \(\pi^*K_Y(\tilde{D})\).

These relations are summarized in the commutative diagram
\[
\begin{array}{rccc}
 & \widetilde{Y} & \stackrel{\iota}{\longrightarrow} & \mathfrak{t}\otimes K_Y(\tilde{D}) \\&
\scriptstyle{\pi}\ \downarrow &  &   \downarrow \\
 & Y & \stackrel{\phi}{\longrightarrow} & \displaystyle\bigoplus_{i=1}^{k}\Bigl(K_Y(\tilde{D})\Bigr)^{d_i}\,.
\end{array}
\]
Under our genericity assumptions the cover \(\widetilde{Y}\) is smooth and irreducible, and each ramification point \(p\in\pi^{-1}(Ram)\) is simple (i.e. the section
\[
\prod_{\alpha\in R^+}\Bigl(d\alpha\circ\iota\Bigr)\colon \widetilde{Y}\to \pi^*K_Y(\tilde{D})^{|R|/2}
\]
has a simple zero at \(p\)).

Let \(X(T)\) denote the group of characters of \(T\). The group of isomorphism classes of holomorphic $\Gamma$-equivariant principal \(T\)-bundles on \(\widetilde{Y}\) is identified with the group
\[
\Pic^\Gamma(\widetilde{Y})\otimes X(T)^*,
\]
where
\[
X(T)^*=\operatorname{Hom}(X(T),\Z)
\]
is the dual group. Similarly, the group of topologically trivial $\Gamma$-equivariant \(T\)-bundles is
\[
J(\widetilde{Y})\otimes X(T)^*,
\]
with \(J^\Gamma(\widetilde{Y})\) the Jacobian of \(\widetilde{Y}\). The Weyl group \(W\) acts naturally on both \(J^{\Gamma}(\widetilde{Y})\) and \(X(T)^*\) (on the latter by conjugation). A $T$-bundle is written as 
\[
\tau = D_1\otimes\chi_1+\cdots+D_l\otimes\chi_l,
\]
where each \(D_i\), $i=1,...,l$ is a divisor on \(\tilde{Y}\) representing a point in the Jacobian, and the $\chi_i$'s are cocharacters. Writing each divisor $D_i$ as
\[
D_i = \sum_{p\in\tilde{Y}} n_p\,p,
\]
the action of an element \(w\in W\) on \(D_i\) is defined by
\[
w\cdot D_i := \sum_{p\in\tilde{Y}} n_p\,\bigl[w(p)\bigr].
\]
In other words, \(w\) acts by sending each point \(p\) in the support of \(D_i\) to \(w(p)\), preserving the multiplicities. This operation induces a natural \(W\)-action on the divisor group \(\operatorname{Div}(\tilde{Y})\) and, hence, on the Picard group \(\Pic(\tilde{Y})\).
On the other hand,
the twisted character \({}^w\chi_i\) is given by
\begin{equation}\label{twisted_char}
{}^w\chi_i(t)=\chi_i\bigl(w^{-1}t\,w\bigr).
\end{equation}
We have for any $w\in W$ an action
\[
{}^w\tau = wD_1\otimes {}^w\chi_1+\cdots+wD_l\otimes {}^w\chi_l.
\]

A critical aspect in order to introduce the definition of the generalized Prym variety $\operatorname{Prym}^\Gamma(Y)$ on $Y$ (Definition \ref{defn:Prym_var_Y} below) and the subsequent abelianization result (Theorem~\ref{thm:abelianization-Gamma}) is the compatibility between the $\Gamma$-equivariant structure on $T$-bundles over $\widetilde{Y}$ and the natural action of the Weyl group $W$ on these bundles. We need to ensure that if a $T$-bundle is $\Gamma$-equivariant, its $W$-transforms are also $\Gamma$-equivariant in a consistent manner.

\begin{lem}\label{lem:gamma_w_actions_on_cameral_cover}
The $\Gamma$-action on $Y$ lifts to an action on the cameral cover $\varpi: \widetilde{Y} \to Y$, denoted by $\bar{\gamma}: \widetilde{Y} \to \widetilde{Y}$ for $\gamma \in \Gamma$. This lifted $\Gamma$-action commutes with the sheet-permuting action of the Weyl group $W$ on $\widetilde{Y}$, denoted by $\bar{w}: \widetilde{Y} \to \widetilde{Y}$ for $w \in W$. That is, for all $p \in \widetilde{Y}$, $\bar{\gamma}(\bar{w}(p)) = \bar{w}(\bar{\gamma}(p))$.
\end{lem}
\begin{proof}
The $\Gamma$-equivariant $G$-Higgs bundle $(P,s)$ on $Y$ features a $\Gamma$-invariant Higgs field $s \in H^0(Y, \mathrm{ad}P \otimes K_Y(\tilde{D}))$. Consequently, the associated characteristic polynomial sections $\phi = (\phi_1, \dots, \phi_k)$, which define the Hitchin map $\mathcal{H}(P,s) = \phi$, are $\Gamma$-invariant: $\delta^*\phi = \phi$ for all $\delta \in \Gamma$, where $\delta^*$ is the pullback on sections over $Y$.
The cameral cover is $\widetilde{Y} = \{ (y, \psi_y) \mid y \in Y, \psi_y \in (\mathfrak{t} \otimes K_Y(\tilde{D}))_y, \underline{\sigma}(\psi_y) = \phi(y) \}$.
For $\delta \in \Gamma$, define its action on $p=(y, \psi_y) \in \widetilde{Y}$ by $\bar{\delta}(p) = (\delta \cdot y, \delta_*\psi_y)$, where $\delta_*\psi_y$ represents the natural transformation of the fiber element $\psi_y$ under the action of $\delta$ (which acts on $Y$ and the bundle $K_Y(\tilde{D})$). Since $\underline{\sigma}$ is $G$-invariant (acting on the $\mathfrak{t}$-component) and $\phi$ is $\Gamma$-invariant, if $(y, \psi_y) \in \widetilde{Y}$, then
\[ \underline{\sigma}(\delta_*\psi_{\delta^{-1}y}) = \delta_*(\underline{\sigma}(\psi_{\delta^{-1}y})) = \delta_*(\phi(\delta^{-1}y)) = (\delta_*\phi)(y) = \phi(y). \]
Hence, $\bar{\delta}(p) \in \widetilde{Y}$, so $\widetilde{Y}$ admits a $\Gamma$-action covering the $\Gamma$-action on $Y$.

\noindent The Weyl group $W$ acts on $p=(y, \psi_y) \in \widetilde{Y}$ by $\bar{w}(p) = (y, w \cdot \psi_y)$, where $w \cdot \psi_y$ is the standard $W$-action on the $\mathfrak{t}$-component of $\psi_y \in \mathfrak{t} \otimes (K_Y(\tilde{D}))_y$. This action fixes $y \in Y$.

\noindent To show commutativity, consider a point $p=(y, \psi_y)$. Then,
\begin{align*}
    (\bar{\delta} \circ \bar{w})(p) &= \bar{\delta}(y, w \cdot \psi_y) = (\delta \cdot y, \delta_*(w \cdot \psi_y)), \text{ and} \\
    (\bar{w} \circ \bar{\delta})(p) &= \bar{w}(\delta \cdot y, \delta_*\psi_y) = (\delta \cdot y, w \cdot (\delta_*\psi_y)).
\end{align*}
Equality holds if $\delta_*(w \cdot \psi_y) = w \cdot (\delta_*\psi_y)$. The term $\psi_y$ can be locally written as $\sum_j c_j t_j \otimes \alpha_j(y)$, where $t_j \in \mathfrak{t}$ and $\alpha_j(y) \in (K_Y(\tilde{D}))_y$. The action $w \cdot \psi_y = \sum_j c_j (w \cdot t_j) \otimes \alpha_j(y)$. Then $\delta_*(w \cdot \psi_y) = \sum_j c_j (w \cdot t_j) \otimes (\delta_*\alpha_j)(\delta \cdot y)$.
Conversely, $\delta_*\psi_y = \sum_j c_j t_j \otimes (\delta_*\alpha_j)(\delta \cdot y)$. Then $w \cdot (\delta_*\psi_y) = \sum_j c_j (w \cdot t_j) \otimes (\delta_*\alpha_j)(\delta \cdot y)$.
The two expressions are identical because the $W$-action only affects the $\mathfrak{t}$-coefficients and the $\delta_*$-action affects the $K_Y(\tilde{D})$ coefficients and the base point. Thus, $\bar{\delta} \circ \bar{w} = \bar{w} \circ \bar{\delta}$ as automorphisms of $\widetilde{Y}$.
\end{proof}

\begin{prop}\label{prop:w_action_preserves_gamma_equivariance}
Let $\mathcal{L}$ be a $\Gamma$-equivariant $T$-bundle on $\widetilde{Y}$. For any $w \in W$, the transformed $T$-bundle ${}^w\mathcal{L}$ carries a natural $\Gamma$-equivariant structure inherited from $\mathcal{L}$. Consequently, the $W$-action is well-defined on the set of isomorphism classes of $\Gamma$-equivariant $T$-bundles on $\widetilde{Y}$.
\end{prop}
\begin{proof}
A $T$-bundle $\mathcal{L}$ on $\widetilde{Y}$ is $\Gamma$-equivariant if for each $\delta \in \Gamma$, there is an isomorphism $u_\delta: (\bar{\delta}^{-1})^*\mathcal{L} \stackrel{\sim}{\to} \mathcal{L}$ such that $u_{\delta_1\delta_2} = u_{\delta_1} \circ (\bar{\delta_1}^{-1})^*u_{\delta_2}$, for all $\delta_1, \delta_2 \in \Gamma$.
The $W$-transform is ${}^w\mathcal{L} = (\bar{w}^{-1})^*\mathcal{L} \otimes \chi_w$, where $\chi_w$ represents the action of $w$ on the characters defining the $T$-structure (denoted by ${}^w\chi_i$ in (\ref{twisted_char})). For simplicity of notation for the geometric part, let $\mathcal{L}_w := (\bar{w}^{-1})^*\mathcal{L}$.

\noindent We define a $\Gamma$-equivariant structure $v_\delta: (\bar{\delta}^{-1})^*\mathcal{L}_w \stackrel{\sim}{\to} \mathcal{L}_w$ for a given $\mathcal{L}_w$.
Using Lemma~\ref{lem:gamma_w_actions_on_cameral_cover}, we have $\bar{\delta}^{-1} \circ \bar{w}^{-1} = \bar{w}^{-1} \circ \bar{\delta}^{-1}$. Thus,
\[ (\bar{\delta}^{-1})^*\mathcal{L}_w = (\bar{\delta}^{-1})^*(\bar{w}^{-1})^*\mathcal{L} = ((\bar{w} \circ \bar{\delta})^{-1})^*\mathcal{L} = ((\bar{\delta} \circ \bar{w})^{-1})^*\mathcal{L} = (\bar{w}^{-1})^*(\bar{\delta}^{-1})^*\mathcal{L}. \]
Define $v_\delta := (\bar{w}^{-1})^*u_\delta : (\bar{w}^{-1})^*(\bar{\delta}^{-1})^*\mathcal{L} \stackrel{\sim}{\to} (\bar{w}^{-1})^*\mathcal{L} = \mathcal{L}_w$.
This map $v_\delta$ is an isomorphism $(\bar{\delta}^{-1})^*\mathcal{L}_w \stackrel{\sim}{\to} \mathcal{L}_w$.
We verify the cocycle condition for $v_\delta$:
\begin{align*}
    v_{\delta_1\delta_2} &= (\bar{w}^{-1})^*u_{\delta_1\delta_2} \\
    &= (\bar{w}^{-1})^*(u_{\delta_1} \circ (\bar{\delta_1}^{-1})^*u_{\delta_2}) \quad (\text{by cocycle condition for } u) \\
    &= ((\bar{w}^{-1})^*u_{\delta_1}) \circ ((\bar{w}^{-1})^*(\bar{\delta_1}^{-1})^*u_{\delta_2}) \\
    &= v_{\delta_1} \circ ((\bar{\delta_1}^{-1})^*(\bar{w}^{-1})^*u_{\delta_2}) \quad (\text{using } \bar{w}^{-1}\bar{\delta_1}^{-1} = \bar{\delta_1}^{-1}\bar{w}^{-1}) \\
    &= v_{\delta_1} \circ (\bar{\delta_1}^{-1})^*v_{\delta_2}.
\end{align*}
Thus, $\{v_\delta\}_{\delta \in \Gamma}$ defines a $\Gamma$-equivariant structure on $\mathcal{L}_w = (\bar{w}^{-1})^*\mathcal{L}$.
The full bundle ${}^w\mathcal{L} = \mathcal{L}_w \otimes \chi_w$ is then also $\Gamma$-equivariant, assuming $\Gamma$ acts trivially on the abstract character group $X(T)$ (which is standard, as $T$ is a fixed group).
Therefore, the condition ${}^w\mathcal{L} \cong \mathcal{L}$ in the definition of $\operatorname{Prym}^\Gamma(Y)$ can be understood as an isomorphism of $\Gamma$-equivariant $T$-bundles.
\end{proof}

We can now introduce the following: 

\begin{defn}\label{defn:Prym_var_Y}
The \emph{generalized Prym variety with respect to the $\Gamma$-action} is defined as
\[
\operatorname{Prym}^\Gamma(Y):=\Bigl[J^\Gamma(\widetilde{Y})\otimes X(T)^*\Bigr]^W,
\]
i.e. the subgroup of those topologically trivial \(T\)-bundles \(\tau\) on \(\widetilde{Y}\) satisfying
\[
{}^w\tau\cong \tau, \quad\text{for all }w\in W.
\]
\end{defn}

Note that $\operatorname{Prym}^\Gamma(Y)$ is an algebraic group whose null connected component $\operatorname{Prym}^\Gamma(Y)_0$ is an abelian variety.

The generic fibers of the Hitchin fibration in the case of stable principal $G$-bundles on a compact Riemann surface were studied by Faltings in \cite{Faltings}. Since we have seen that the $\Gamma$-equivariant structure on $T$-bundles and the natural Weyl group action on these bundles are compatible, the constructions of the cameral cover and the generalized Prym variety proceed as in the non-equivariant principal $G$-bundle case. Moreover, the proof that the generic Hitchin fibers are isomorphic to generalized Prym varieties by Scognamillo \cite{Scog} (simplifying the original proof of Faltings) repeats word by word in the presence of a $\Gamma$-equivariant action on the principal $T$-bundle as above. We thus conclude to the following:

\begin{thm}[Abelianization of \(\Gamma\)-Equivariant $G$-Higgs Bundles]\label{thm:abelianization-Gamma}
Let \((P,s)\) be a stable \(\Gamma\)-equivariant \(G\)-Higgs bundle on \(Y\). Assume that the Hitchin base is generic so that the associated cameral cover $\pi\colon \widetilde{Y}\to Y$ is smooth. Then there exists a canonical construction of a \(\Gamma\)-equivariant \(T\)-bundle
\[
\mathcal{T}(P,s)\in \operatorname{Pic}^\Gamma(\widetilde{Y})\otimes X^*(T)
\]
satisfying
\[
{}^w\mathcal{T}(P,s)\cong \mathcal{T}(P,s), \quad\text{for all }w\in W.
\]
In particular, the assignment $(P,s)\longmapsto \mathcal{T}(P,s)$
defines an injective morphism from (each connected component of) the Hitchin fiber
\[
\mathcal{H}^{-1}(\phi)\subset \mathcal{M}_H^\Gamma(Y,G)
\]
to the generalized Prym variety
\[
\operatorname{Prym}^\Gamma(Y)=\Bigl\{Q\in \operatorname{Pic}^\Gamma(\widetilde{Y})\otimes X^*(T)\; \Bigm|\; {}^wQ\cong Q\text{ for all }w\in W\Bigr\}.
\]
Thus, the generic Hitchin fiber is an abelian torsor.
\end{thm}

In view of the correspondence between the moduli space of logahoric $\mathcal{G}_{\boldsymbol \theta}$-Higgs torsors on $X$ and the $\Gamma$-equivariant $G$-Higgs bundles on $Y$ from \cite[Theorem 3.7]{KSZparh}, we now have the following:
\begin{cor}\label{parh_Hitchin_fiber_abelian}
The generic Hitchin fibers of the parahoric Hitchin fibration $$h_{\boldsymbol\theta}: \mathcal{M}_H (X, \mathcal{G}_{\boldsymbol\theta}) \to \mathcal{A}_{\boldsymbol \theta}$$ are abelian torsors.
\end{cor}

\subsection{Regular Centralizer}

This section establishes the duality that underlies the complete integrability of the parahoric Hitchin system. We begin with a local analysis of the regular centralizer in the parahoric context, which we then globalize to the moduli space of logahoric $\mathcal{G}_{\boldsymbol{\theta}}$-Higgs torsors. The section culminates in showing that the generic fibers of the parahoric Hitchin fibration are Lagrangian.

In the parahoric setting, for a given weight $\theta$, we work with the parahoric group scheme $\mathcal{G}_\theta$ whose Lie algebra is $\mathfrak{g}_\theta$. For an element $\varphi \in \mathfrak{g}_\theta^*$, the \emph{parahoric centralizer} is the subgroup
\[
C_{\mathcal{G}_\theta}(\varphi) := \{ g \in \mathcal{G}_\theta \mid \operatorname{Ad}^*(g)(\varphi)=\varphi \},
\]
with Lie algebra
\[
\mathfrak{g}_\theta^{\varphi} := \{ X \in \mathfrak{g}_\theta \mid \operatorname{ad}^*(X)(\varphi)=0 \}.
\]
These assemble into a group scheme $C_{\mathfrak{g}_\theta} \to \mathfrak{g}_\theta^*$. The regular locus is the intersection $(\mathfrak{g}_\theta^*)^{\mathrm{reg}} := \mathfrak{g}_\theta^* \cap (\mathfrak{g}^*)^{\mathrm{reg}}$. The universal centralizer $J$ over the Chevalley basis $\mathfrak{c}^* := \mathfrak{t}^*/W$ pulls back via the Chevalley map $\chi\colon \mathfrak{g}^* \to \mathfrak{c}^*$ to give an isomorphism over the regular locus:
\[
\chi^*J\big|_{(\mathfrak{g}_\theta^*)^{\mathrm{reg}}} \cong C_{\mathcal{G}_\theta}\big|_{(\mathfrak{g}_\theta^*)^{\mathrm{reg}}}.
\]

\begin{rem}[$\mathbb{G}_m$-action and the moment map derivative]\label{rem:dm-parahoric-revised}
The multiplicative group $\mathbb{G}_m$ acts on $\mathfrak{g}^*$ by scalar multiplication, inducing an action on $\mathfrak{c}^*$. This action preserves centralizers, so the action on $C_{\mathfrak{g}^*}$ is $t\cdot (g,\varphi) := (g,t\varphi)$. This action preserves the regular locus. The derivative of the moment map for the $G$-action on $T^*G$, restricted to the centralizer, gives a morphism
\[
dm\colon \chi^*\operatorname{Lie}(J_{\mathfrak{c}^*}) \longrightarrow \mathfrak{g}\times \mathfrak{g}^*.
\]
This map is equivariant for the $\mathbb{G}_m$-action where $\mathbb{G}_m$ acts trivially on $\mathfrak{g}$ and by scaling on $\mathfrak{g}^*$. Identifying the cotangent bundle $T^*\mathfrak{g}^* \cong \mathfrak{g}(-1) \times \mathfrak{g}^*$, we can view $dm$ as a morphism
\begin{equation}\label{eq:dm-parahoric-revised}
dm\colon \chi^*\operatorname{Lie}(J_{\mathfrak{c}^*})(-1) \longrightarrow T^*\mathfrak{g}^*.
\end{equation}
Its restriction to the regular locus $(\mathfrak{g}^*)^{\mathrm{reg}}$ (and hence to $(\mathfrak{g}_\theta^*)^{\mathrm{reg}}$) is injective.
\end{rem}

\begin{rem}[Derivative of the Chevalley map]\label{rem:dchi-parahoric-revised}
The Chevalley map $\chi\colon \mathfrak{g}^* \to \mathfrak{c}^*$ is $G$-invariant and $\mathbb{G}_m$-equivariant. Its derivative
\begin{equation}\label{eq:dchi-parahoric-revised}
d\chi\colon T\mathfrak{g}^* \longrightarrow \chi^* T\mathfrak{c}^*
\end{equation}
is therefore also $\mathbb{G}_m$-equivariant. The restriction of $d\chi$ to the regular locus is surjective, a fact that remains true upon further restriction to $(\mathfrak{g}_\theta^*)^{\mathrm{reg}}$.
\end{rem}

\begin{lem}\label{lem:LocalPairing-parahoric-revised}
The canonical pairing on $T\mathfrak{g}^* \times_{\mathfrak{g}^*} T^*\mathfrak{g}^*$ induces a $G\times \mathbb{G}_m$-equivariant perfect pairing on the regular locus
\[
\chi^*\operatorname{Lie}(J)\big|_{(\mathfrak{g}^*)^{\mathrm{reg}}}(-1) \times_{(\mathfrak{g}^*)^{\mathrm{reg}}} \chi^*T\mathfrak{c}^*\big|_{(\mathfrak{g}^*)^{\mathrm{reg}}} \longrightarrow \mathbb{C}(0),
\]
which yields an isomorphism of vector bundles over $\mathfrak{c}^*_{\mathrm{reg}}$:
\[
\operatorname{Lie}(J)^*(1) \cong T\mathfrak{c}^*.
\]
\end{lem}
\begin{proof}
From Remarks \ref{rem:dm-parahoric-revised} and \ref{rem:dchi-parahoric-revised}, $\chi^*\operatorname{Lie}(J)|_{(\mathfrak{g}^*)^{\mathrm{reg}}}(-1)$ is a subbundle of $T^*\mathfrak{g}^*|_{(\mathfrak{g}^*)^{\mathrm{reg}}}$ and $\chi^*T\mathfrak{c}^*|_{(\mathfrak{g}^*)^{\mathrm{reg}}}$ is a quotient bundle of $T\mathfrak{g}^*|_{(\mathfrak{g}^*)^{\mathrm{reg}}}$. Both bundles have the same rank. Since $\chi$ is constant on $G$-orbits, the tangent space to the $G$-orbit at a regular element $\varphi$, $V_\varphi := \operatorname{Im}(\mathfrak{g} \xrightarrow{\mathrm{ad}^*(\cdot)(\varphi)} T_\varphi\mathfrak{g}^*)$, is contained in $\mathrm{Ker}(d\chi_\varphi)$. A dimension count shows that for regular $\varphi$, $\dim V_\varphi = \dim\mathfrak{g} - \mathrm{rank}(\mathfrak{g})$, which equals the dimension of $\mathrm{Ker}(d\chi_\varphi)$ since $d\chi_\varphi$ is surjective. Thus, $\mathrm{Ker}(d\chi_\varphi) = V_\varphi$. By $G$-invariance, the annihilator of the orbit tangent space under the canonical pairing is the centralizer Lie algebra, $V_\varphi^\perp = \mathfrak{g}^\varphi \cong \mathrm{Lie}(J_\varphi)$. This establishes the perfect pairing.
\end{proof}

\subsection{Globalization and Duality}
We now globalize this local construction. Let $\mathcal{M}_{LH}(X, \mathcal{G}_{\boldsymbol{\theta}})$ be the moduli space of leveled logahoric $\mathcal{G}_{\boldsymbol \theta}$-Higgs torsors $(\mathcal{E}, \varphi, \eta)$. This space is an open subset of the cotangent bundle $T^*\mathcal{U}(X, \mathcal{G}_{\boldsymbol{\theta}})$ and is therefore a symplectic manifold with a canonical symplectic form, which we denote by $\omega_{LH}$.

Let $\ell: \mathcal{M}_{LH}(X, \mathcal{G}_{\boldsymbol{\theta}}) \to \mathcal{M}_H(X, \mathcal{G}_{\boldsymbol{\theta}})$ be the forgetful map, which is a Poisson map as we have seen. The Hitchin fibration on the leveled space is the composition 
\[{\widetilde H_{\boldsymbol\theta}} := h_{\boldsymbol{\theta}} \circ \ell: \mathcal{M}_{LH}(X, \mathcal{G}_{\boldsymbol{\theta}}) \to \mathcal{A}_{\boldsymbol \theta}.\]
The group scheme $J$ on $\mathfrak{c}^*$ pulls back along the Hitchin fibration $h_{\boldsymbol{\theta}}$ to a group scheme $J_{\mathcal{A}_{\boldsymbol{\theta}}}$ over the Hitchin base $\mathcal{A}_{\boldsymbol{\theta}}$. The moduli space of $J_{\mathcal{A}_{\boldsymbol{\theta}}}$-torsors on $X$ forms a group scheme $P_{\mathcal{A}_{\boldsymbol{\theta}}}$ over $\mathcal{A}_{\boldsymbol{\theta}}$, whose fiber $P_a$ over a point $a \in \mathcal{A}_{\boldsymbol{\theta}}$ is the generalized Prym variety associated to the cameral cover determined by $a$. This induces an action on the Hitchin fibers, which lifts to the leveled space:
\[
\mathrm{act}_{\boldsymbol \theta}\colon P_{\mathcal{A}_{\boldsymbol{\theta}}}\times_{\mathcal{A}_{\boldsymbol{\theta}}} \mathcal{M}_{LH}(X,{\mathcal{G}_{\boldsymbol{\theta}}}) \longrightarrow \mathcal{M}_{LH}(X,{\mathcal{G}_{\boldsymbol{\theta}}}).
\]

The following statement can be now obtained analogously to \cite[Proposition A.12]{dCHM}:
\begin{prop}\label{prop:dactdhdual-parahoric-revised-global}
There exists a canonical isomorphism of vector bundles over $\mathcal{A}_{\boldsymbol{\theta}}$,
\[
\mathrm{Lie}\Bigl(P_{\mathcal{A}_{\boldsymbol{\theta}}}/\mathcal{A}_{\boldsymbol{\theta}}\Bigr) \;\cong\; T^*\mathcal{A}_{\boldsymbol{\theta}},
\]
such that the differential of the lifted action,
\[
d\mathrm{act}_{\boldsymbol \theta}\colon {\widetilde H_{\boldsymbol\theta}}^*\mathrm{Lie}\Bigl(P_{\mathcal{A}_{\boldsymbol{\theta}}}/\mathcal{A}_{\boldsymbol{\theta}}\Bigr) \longrightarrow T\mathcal{M}_{LH}(X,{\mathcal{G}_{\boldsymbol{\theta}}}),
\]
and the differential of the lifted Hitchin fibration,
\[
d{\widetilde H_{\boldsymbol\theta}}\colon T\mathcal{M}_{LH}(X,{\mathcal{G}_{\boldsymbol{\theta}}}) \longrightarrow {\widetilde H_{\boldsymbol\theta}}^*T\mathcal{A}_{\boldsymbol{\theta}}
\]
are dual to each other with respect to the canonical symplectic form $\omega_{LH}$ on the smooth locus of $\mathcal{M}_{LH}(X,{\mathcal{G}_{\boldsymbol{\theta}}})$.
\end{prop}

\begin{proof}
Let $(\mathcal{E},\varphi, \eta)$ be a leveled logahoric $\mathcal{G}_{\boldsymbol{\theta}}$-Higgs torsor mapping to $a \in \mathcal{A}_{\boldsymbol{\theta}}$. The tangent space to the moduli space at this point, $T_{(\mathcal{E},\varphi,\eta)}\mathcal{M}_{LH}$, is given by the hypercohomology group $\mathbb{H}^1(X, \mathcal{K}_{\mathcal{E},\varphi})$, where $\mathcal{K}_{\mathcal{E},\varphi}$ is the two-term complex $[\mathcal{E}(\mathfrak{g}^+) \xrightarrow{\mathrm{ad}(\varphi)} \mathcal{E}(\mathfrak{g}) \otimes K_X(D)]$ that governs deformations preserving the level structure.

\noindent The differential of the Hitchin map, $d{\widetilde H_{\boldsymbol\theta}}$, is induced on hypercohomology by the morphism of complexes from $\mathcal{K}_{\mathcal{E},\varphi}$ to $\chi(\varphi)^*T\mathfrak{c}_{K_X(D)}$ given by the derivative of the Chevalley map, $d\chi$. The differential of the action, $d\mathrm{act}_{\boldsymbol \theta}$, is induced by a map from the complex $a^*\mathrm{Lie}(J_{\mathcal{A}_{\boldsymbol{\theta}}})[-1]$ into the complex $\mathcal{K}_{\mathcal{E},\varphi}$.

\noindent The local duality from Lemma \ref{lem:LocalPairing-parahoric-revised} globalizes. The perfect pairing between $\chi^*\operatorname{Lie}(J)$ and $\chi^*T\mathfrak{c}^*$ induces, via Serre duality on hypercohomology, a perfect pairing:
\[
\mathbb{H}^1\Bigl(X, a^*\mathrm{Lie}(J_{\mathcal{A}_{\boldsymbol{\theta}}})\Bigr) \times \mathbb{H}^1\Bigl(X, \chi(\varphi)^*T\mathfrak{c}_{K_X(D)}\Bigr) \longrightarrow \mathbb{C}.
\]
Standard deformation theory identifies $\mathrm{Lie}(P_a) \cong \mathbb{H}^1(X, a^*\mathrm{Lie}(J_{\mathcal{A}_{\boldsymbol{\theta}}}))$ and $$T_a^*\mathcal{A}_{\boldsymbol{\theta}} \cong \mathbb{H}^1(X, \chi(\varphi)^*T\mathfrak{c}_{K_X(D)}).$$ The duality of the hypercohomology groups translates directly to the asserted duality between the morphisms $d\mathrm{act}_{\boldsymbol \theta}$ and $d{\widetilde H_{\boldsymbol\theta}}$ with respect to the symplectic form $\omega_{LH}$.
\end{proof}

We now use the duality principle established above in order to prove the complete integrability of the parahoric Hitchin system. The core of the argument is to first show that the generic fibers of the Hitchin fibration on the symplectic manifold $\mathcal{M}_{LH}(X, \mathcal{G}_{\boldsymbol{\theta}})$ are Lagrangian, and then to descend this property to the Poisson manifold $\mathcal{M}_{H}(X,\mathcal{G}_{\boldsymbol{\theta}})$ via the forgetful map.

\begin{thm}\label{thm:lagrangian_LH}
The generic fibers of the Hitchin fibration ${\widetilde H_{\boldsymbol\theta}}\colon \mathcal{M}_{LH}(X, \mathcal{G}_{\boldsymbol{\theta}}) \to \mathcal{A}_{\boldsymbol \theta}$ are Lagrangian subvarieties of the symplectic manifold $(\mathcal{M}_{LH}(X, \mathcal{G}_{\boldsymbol{\theta}}), \omega_{LH})$.
\end{thm}

\begin{proof}
Let $a \in \mathcal{A}'$ be a generic point in the Hitchin base. The fiber of the Hitchin fibration over this point is $F'_{a} := {\widetilde H_{\boldsymbol\theta}}^{-1}(a)$. For any point $p \in F'_{a}$, the tangent space to the fiber is given by the kernel of the differential of the Hitchin map, $T_p F'_{a} = \text{Ker}(d{\widetilde H_{\boldsymbol\theta}})_p$.

\noindent The fiber $F'_{a}$ is acted upon by the generalized Prym variety $P_a$. The tangent space to the orbit of this action at $p$ is given by the image of the differential of the action map, $\mathrm{Im}(d\mathrm{act}_{\boldsymbol \theta})_p$. As established in Proposition \ref{prop:dactdhdual-parahoric-revised-global}, these two subspaces of the tangent space $T_p\mathcal{M}_{LH}(X,{\mathcal{G}_{\boldsymbol{\theta}}})$ are symplectic orthogonals with respect to the form $\omega_{LH}$:
\[
\text{Ker}(d{\widetilde H_{\boldsymbol\theta}}) = (\mathrm{Im}(d\mathrm{act}_{\boldsymbol \theta}))^{\perp}.
\]
The action of the abelian group scheme $P_{\mathcal{A}_{\boldsymbol \theta}}$ on the fibers generates isotropic submanifolds. This is a standard result stemming from the fact that the action linearizes on the Jacobian of the cameral cover, and the symplectic form, when pulled back to the abelian Prym variety, is translation-invariant and must therefore be zero. Consequently, the tangent space to the Prym orbits, $\mathrm{Im}(d\mathrm{act}_{\boldsymbol \theta})$, is an isotropic subspace of $T\mathcal{M}_{LH}(X, \mathcal{G}_{\boldsymbol \theta})$.

\noindent A fundamental result in symplectic linear algebra states that the symplectic orthogonal of an isotropic subspace is a coisotropic subspace. Therefore, the tangent space to the Hitchin fiber, $\text{Ker}(d{\widetilde H_{\boldsymbol\theta}})$, is coisotropic.

\noindent In conclusion, the subspaces $F_a'$ are coisotropic and isotropic subspaces of $T\mathcal{M}_{LH}(X, \mathcal{G}_{\boldsymbol \theta})$, and therefore Lagrangian.
\end{proof}

\begin{thm}\label{thm:lagrangian_H}
The generic fibers of the parahoric Hitchin fibration $h_{\boldsymbol \theta} \colon \mathcal{M}_{H}(X,\mathcal{G}_{\boldsymbol{\theta}}) \to \mathcal{A}_{\boldsymbol \theta}$ are Lagrangian subvarieties with respect to the symplectic leaves of the Poisson moduli space $\mathcal{M}_{H}(X,\mathcal{G}_{\boldsymbol{\theta}})$.
\end{thm}

\begin{proof}
This theorem is a direct consequence of Theorem \ref{thm:lagrangian_LH} and the fact that the forgetful map $\ell: \mathcal{M}_{LH}(X, \mathcal{G}_{\boldsymbol \theta}) \to \mathcal{M}_{H}(X, \mathcal{G}_{\boldsymbol \theta})$ is a Poisson map, a result established in Theorem \ref{thm:forgetful_Poisson}. 
\end{proof}

Collecting Theorems \ref{thm:forgetful_Poisson}, \ref{parh_Hitchin_fiber_abelian} and \ref{thm:lagrangian_H}, we now conclude to the following:

\begin{thm}\label{thm:main}
Let $X$ be a smooth complex algebraic curve and $D$ be a reduced effective divisor on $X$. Let $G$ be a connected complex reductive group. The moduli space $\mathcal{M}_H(X,\mathcal{G}_{\boldsymbol\theta})$ of logahoric $\mathcal{G}_{\boldsymbol \theta}$-Higgs torsors over $X$ is Poisson and is fibered via a map $h_{\boldsymbol \theta}: \mathcal{M}_H(X,\mathcal{G}_{\boldsymbol\theta}) \to \mathcal{A}_{\boldsymbol \theta}$ by abelian torsors. Moreover, $h_{\boldsymbol \theta}: \mathcal{M}_H(X,\mathcal{G}_{\boldsymbol\theta}) \to \mathcal{A}_{\boldsymbol \theta}$ is an algebraically completely integrable Hamiltonian system in the sense of Definition \ref{defn:aciHs}.
\end{thm}

We will call the algebraically completely integrable Hamiltonian system of Theorem \ref{thm:main}, the \emph{logahoric Hitchin integrable system}.

\section{Symplectic leaf foliation}

We now study the symplectic leaves in the foliation of the Poisson moduli space $\mathcal{M}_H(X, \mathcal{G}_{\boldsymbol \theta})$.

\begin{lem}\label{lem:nilpotent-vanish}
Let $p\colon\mathfrak g\to\mathbb C$ be a \emph{homogeneous} $G$-invariant
polynomial of positive degree.  Then $p(x)=0$, for every nilpotent
element $x\in\mathfrak g$.
\end{lem}

\begin{proof}
Fix a nilpotent element $x\in\mathfrak g$.  
By the Jacobson--Morozov Theorem, there exists a 1-parameter cocharacter
$\lambda:\mathbb C^{\times}\!\to G$ such that $\operatorname{Ad}(\lambda(t))\,x \;=\; t^{\,2}\,x$,
 for all $t\in\mathbb C^{\times}$. Because \(p\) is \(G\)-invariant,
\( p\bigl(\operatorname{Ad}(\lambda(t))x\bigr)=p(x)\).
Thus, $p(x)\;=\;p\bigl(t^{\,2}x\bigr)
\;=\;t^{2\deg p}\,p(x)$ (homogeneity).

\noindent Now, pick \(t\neq0\).  If \(p(x)\neq0\) we could divide by it to get  
\(1=t^{2\deg p}\), which is impossible for arbitrary \(t\).  Hence \(p(x)=0\).
\end{proof}

From Lemma~\ref{lem:nilpotent-vanish} we get for the Chevalley morphism that $\chi(X)=0$, for all $X\in\mathfrak g_\theta^{+}$. We set the following:

\begin{defn}[Strongly logahoric $\mathcal{G}_{\boldsymbol \theta}$-Higgs torsors and $\mathcal{A}_{\boldsymbol{\theta}}^+$]
A logarithmic Higgs field $\varphi \in H^0(X, \mathcal{E}(\mathfrak{g})\otimes K(D))$ is called \textit{strongly logarithmic} if, for each $x_i \in D$, its residue $\mathrm{Res}_{x_i} \varphi = 0$. A \emph{strongly logahoric $\mathcal{G}_{\boldsymbol \theta}$-Higgs torsor} over $X$ is a parahoric $\mathcal{G}_{\boldsymbol \theta}$-Higgs torsor over $X$ with a strongly logarithmic Higgs field. Since non-constant invariant polynomials $p_j$ vanish on nilpotent elements (Lemma~\ref{lem:nilpotent-vanish}), then $p_j(\mathrm{Res}_{x_i} \varphi) = 0$, implying that  $p_j(\varphi)$ vanishes along $D$. Now, set
\[
\mathcal{B}_{\boldsymbol{\theta}}^+ := \bigoplus_{j=1}^{l} H^0(X, (K_X(D))^{m_j + 1} \otimes \mathcal{O}_X(-D)),
\]
which is a vector subspace of $\mathcal{B}_{\boldsymbol{\theta}}$. The subvariety $\mathcal{A}_{\boldsymbol{\theta}}^+$ is then the image under $h_{\boldsymbol{\theta}}$ of the strongly logahoric $\mathcal{G}_{\boldsymbol \theta}$-Higgs torsors, consisting of tuples $(s_1, \dots, s_l) \in \mathcal{A}_{\boldsymbol{\theta}}$ such that each $s_j$ vanishes along $D$. In other words, $\mathcal{A}_{\boldsymbol{\theta}}^+ = \mathcal{A}_{\boldsymbol{\theta}} \cap \mathcal{B}_{\boldsymbol{\theta}}^+$.
\end{defn}

\begin{defn}[Local Residue Data]
For each $x_i \in D$, let $G_{\theta_i}$ be the parahoric subgroup of $G$ at $x_i$, with Levi factor $L_{\theta_i}$ and Lie algebra $\mathfrak{l}_{\theta_i}$. The residue $\mathrm{Res}_{x_i} \varphi$ lies in $\hat{\mathfrak{l}}_{\theta_i}\cong \mathfrak{l}_{\theta_i}$. The \textit{effective structure group for residues} is $G_D = \left(\prod_{i=1}^s L_{\theta_i}\right)/Z(G)$, with Lie algebra $\mathfrak{g}_D = \bigoplus_{i=1}^s \mathfrak{l}_{\theta_i}$. The \textit{space of local invariant data} is the categorical quotient $\mathfrak{g}_D^* // G_D = \mathrm{Spec}(\mathbb{C}[\mathfrak{g}_D^*]^{G_D})$, specified by $L_{\theta_i}$-invariant polynomial values on $\mathrm{Res}_{x_i} \varphi \in \mathfrak{l}_{\theta_i}$, for each $x_i \in D$.
\end{defn}

For each basic Chevalley invariant \(p_j\), the short exact sequence
\[
0\to(K_X(D))^{m_j+1}(-D)\to(K_X(D))^{m_j+1}\xrightarrow{\operatorname{res}_D}\mathcal{O}_D\to0
\]
gives a surjective linear map
\begin{equation}
  \operatorname{res}_D\colon
    B_{\boldsymbol\theta}\;\twoheadrightarrow\;
    I\otimes H^0(D,\mathcal{O}_D),
  \quad \text{ with } \quad
    \text{Ker}(\operatorname{res}_D)=B_{\boldsymbol\theta}^+.\label{eq:hitchinbase-residue}
\end{equation}
Here \(I=\bigoplus_{j=1}^{\ell}I_{m_j+1}\) is the span of the basic invariants.
For each marked point \(x_i\in D\), let \(L_{\theta_i}\) be the Levi factor of the
parahoric subgroup \(G_{\theta_i}\) with Lie algebra \(\mathfrak l_{\theta_i}\).
The Jacobson--Morozov Theorem plus Lemma~\ref{lem:nilpotent-vanish} imply that
\(\mathbb C[\mathfrak g_{\theta_i}]^{G_{\theta_i}}
  =\mathbb C[\mathfrak l_{\theta_i}]^{L_{\theta_i}}\),
hence the canonical isomorphism
\[
  \mathfrak g_{\theta_i}//G_{\theta_i}\;\cong\;\mathfrak l_{\theta_i}//L_{\theta_i}.
\]
Set
\(
  G_D:=\bigl(\prod_{i=1}^sL_{\theta_i}\bigr)/Z(G),
  \;
  \mathfrak g_D:=\bigoplus_{i=1}^s\mathfrak l_{\theta_i}.
\)
Then, Chevalley’s theorem gives
\begin{equation}\label{eq:quotient-identification}
  \mathfrak g_D^*//G_D
  \;\cong\;
  I\otimes H^0(D,\mathcal{O}_D)
\end{equation}
as an affine space; now define
\[
  \mathcal A_{\boldsymbol\theta}/\mathcal A_{\boldsymbol\theta}^+
  \;:=\;
\mbox{the image of } \mathcal A_{\boldsymbol\theta}\mbox{ inside } \mathcal{B}_{\boldsymbol\theta}/ \mathcal{B}_{\boldsymbol\theta}^+.
\]

\begin{rem}
Note that $ \mathcal A_{\boldsymbol\theta}/\mathcal A_{\boldsymbol\theta}^+$ above is not a well-defined quotient as $\mathcal A_{\boldsymbol\theta}^+$ is not a vector space. However, we adopt this notation as to align with the vector space $B_L/B_0$ appearing in \cite[Section 8.3]{Markman}.  
\end{rem}
\begin{thm}\label{thm:commuting}
There exists an algebraic morphism $\overline{r}: \mathcal{A}_{\boldsymbol{\theta}}/\mathcal{A}_{\boldsymbol{\theta}}^+ \to \mathfrak{g}_D^* // G_D$ which is an isomorphism of affine varieties:
\[
 \mathcal A_{\boldsymbol\theta}/\mathcal A_{\boldsymbol\theta}^+ \cong \mathfrak{g}_D^* // G_D.
\]
As so this fits into a diagram
\begin{equation}\label{main_diagram}
\begin{tikzcd}[column sep=large, row sep=large]
& \mathcal{A}_{\boldsymbol{\theta}} \arrow[r, "q"]  & \mathcal{A}_{\boldsymbol{\theta}} / \mathcal{A}_{\boldsymbol{\theta}}^+ \arrow[dd, "\cong"] \\
T^*\mathcal{U}(X,\mathcal{G}_{\boldsymbol{\theta}}) \arrow[dr, "\mu"'] \arrow[ur, "\widetilde H_{\boldsymbol\theta}"'] & &  \\
& \mathfrak{g}_D^* \arrow[r, "c.q."] & \mathfrak{g}_D^* // G_D 
\end{tikzcd}
\end{equation}
for the canonical quotient $\mathfrak{g}_D^* \to \mathfrak{g}_D^*//G_D$ and the quotient map $\mathcal{A}_{\boldsymbol{\theta}} \to \mathcal{A}_{\boldsymbol{\theta}}/ \mathcal{A}_{\boldsymbol{\theta}}^+$.

\end{thm}

\begin{proof}
For a pair \((\mathcal{E},\varphi)\), let \(\xi_i=\operatorname{Res}_{x_i}\varphi\in\mathfrak g_{\theta_i}\).
By the paragraph above, the tuple of invariant values
\((p_1(\xi_i),\dots,p_\ell(\xi_i))\) depends only on the Levi component
\(\xi_{i,s}\in\mathfrak l_{\theta_i}\) and defines a point of
\(\mathfrak l_{\theta_i}//L_{\theta_i}\cong\mathfrak g_{\theta_i}//G_{\theta_i}\).
Collecting over \(i\) yields a map
\(
  r:\mathcal A_{\boldsymbol\theta}\to\mathfrak g_D^*//G_D.
\)
If \(\varphi\) is strongly logarithmic then each residue is nilpotent, so
\(r(\varphi)=0\); hence \(r\) factors through
\(\bar r:\mathcal A_{\boldsymbol\theta}/\mathcal A_{\boldsymbol\theta}^+
      \to\mathfrak g_D^*//G_D\).

\noindent Given any \((\eta_i)\in\mathfrak g_D\) we may choose local
representatives with semisimple part \(\eta_i\) and glue a Higgs field on \(X\)
whose residues are these \(\eta_i\), this is achievable since we assume our torsor to be generically split.  Evaluating invariants shows every point
of \(\mathfrak g_D^*//G_D\) is in the image of \(\bar r\).

\noindent Now note that the induced map \(\overline{res_D}:\mathcal{B}_{\boldsymbol \theta}/\mathcal{B}_{\boldsymbol \theta}^+ \to \mathfrak{g}_D^*//G_D \) is injective because the kernel of \(\mathcal{B}_{\boldsymbol \theta}\) is inside \(\mathcal{B}_{\boldsymbol \theta}^+\). As so the image of \(\mathcal{A}_{\boldsymbol \theta}\) inside \(\mathcal{B}_{\boldsymbol \theta}/\mathcal{B}_{\boldsymbol \theta}^+\) maps to \(\mathfrak{g}_D^*//G_D \) injectively from the diagram below equation (\ref{eq:hitchinbase-residue})\[\mathcal{A}_{\boldsymbol \theta}\to \mathcal{B}_{\boldsymbol \theta}\to \mathcal{B}_{\boldsymbol \theta}/\mathcal{B}^+_{\boldsymbol \theta} \overset{\overline{res_D}}{\to} \mathfrak{g}_D^*//G_D.\]
In the diagram (\ref{main_diagram}), note that the moment map \(\mu\) followed by the Chevalley map equals the evaluation of
\(p_j(\varphi)\) at \(x_i\). Therefore, the map \(\bar r\) is an algebraic isomorphism and the diagram (\ref{main_diagram}) commutes.
\end{proof}

\begin{cor}
\label{cor:parahoric-leaves-refine}
The foliation of the smooth locus
\(\mathcal{M}_H(X,\mathcal{G}_{\boldsymbol\theta})^{\mathrm{sm}}\) by its symplectic
leaves refines the foliation by fibers of the map
\[
  q\circ\widetilde H_{\boldsymbol\theta}
  \;:\;
  \mathcal{M}_H(X,\mathcal{G}_{\boldsymbol\theta})^{\mathrm{sm}}
  \;\longrightarrow\;
  \mathcal{A}_{\boldsymbol\theta}/\mathcal{A}_{\boldsymbol\theta}^+.
\]
Moreover, each fiber of \(q\circ\widetilde H_{\boldsymbol\theta}\) contains
a unique symplectic leaf of maximal dimension.
\end{cor}

\begin{proof}
On the open locus \(R\subset|T^*\mathcal{U}(X,\mathcal{G}_{\boldsymbol{\theta}})|\) we have
\[
  q\circ\widetilde H_{\boldsymbol\theta}
  \;=\;
  \mathrm{c.q.}\,\circ\mu.
\]
Since symplectic leaves in \(T^*\mathcal{U}(X,\mathcal{G}_{\boldsymbol{\theta}})\) are exactly the connected
components of the pre-images under \(\mu\) of \(G_D\)-coadjoint
orbits in \(\mathfrak{g}_D^*\), Theorem~\ref{thm:commuting} implies that each such orbit
sits inside exactly one leaf in a given fiber of \(q\circ\widetilde H_{\boldsymbol\theta}\).
Hence the fibers decompose into symplectic leaves, and the unique open
(co)adjoint orbit in each fiber yields the unique leaf of maximal
dimension.
\end{proof}

\section{Examples of Logahoric Hitchin Integrable Systems}

Many classical, as well as recently discovered, integrable systems can be realized as symplectic leaves of the integrable systems introduced in this paper, for suitable choices of the group $G$, the curve $X$, the divisor $D$, and the parahoric data $\boldsymbol{\theta}$. This section is devoted to demonstrating certain examples in the case when the base curve is the projective line, $X = \mathbb{P}^1$ or an elliptic curve $X= \Sigma$. We will show that in the first case, our general framework generalizes the integrable system of Beauville \cite{Beauville}, and recovers in its simplest form, the classical Gaudin model \cite{FFR}. Moreover, for base curve an elliptic curve $X=\Sigma$, we recover the space of $\Lambda$-periodic KP elliptic solitons and the elliptic Calogero--Moser system.

\subsection{\texorpdfstring{Logarithmic Hitchin System on $\mathbb{P}^1$}{Logarithmic Hitchin System on P1}}

For $X = \mathbb{P}^1$, the geometry simplifies considerably. A principal $G$-bundle on $\mathbb{P}^1$ is semistable if and only if it is trivial. Since the stability of a parahoric $\mathcal{G}_{\boldsymbol{\theta}}$-torsor requires the underlying principal $G$-bundle to be semistable, we can fix the underlying torsor $\mathcal{E}$ to be the trivial one, i.e., $\mathcal{E} = \mathbb{P}^1 \times G$.

Let $D = \{x_1, \dots, x_s\} \subset \mathbb{P}^1$ be a reduced effective divisor of degree $s \ge 1$. The cotangent space to the moduli space of leveled torsors $\mathcal{U}(\mathbb{P}^1, \mathcal{G}_{\boldsymbol{\theta}})$ is the space of logarithmic Higgs fields:
\[
T^*_{[(\mathcal{E}, \eta)]} \mathcal{U}(\mathbb{P}^1, \mathcal{G}_{\boldsymbol{\theta}}) \cong H^0(\mathbb{P}^1, \mathfrak{g} \otimes K_{\mathbb{P}^1}(D)).
\]
Let us choose a coordinate $z$ on $\mathbb{P}^1$ such that no points in $D$ are at $z=\infty$. The canonical bundle is $K_{\mathbb{P}^1} = \mathcal{O}_{\mathbb{P}^1}(-2)$, so $K_{\mathbb{P}^1}(D) = \mathcal{O}_{\mathbb{P}^1}(s-2)$. A section $\Phi \in H^0(\mathbb{P}^1, \mathfrak{g} \otimes \mathcal{O}_{\mathbb{P}^1}(s-2))$ is a $\mathfrak{g}$-valued polynomial in $z$ of degree at most $s-2$. Let $\mathcal{P}_{s-2}(\mathfrak{g})$ denote the vector space of such polynomials. The phase space of the integrable system is a Poisson quotient of this space:
\[
M_D \cong \mathcal{P}_{s-2}(\mathfrak{g}) // G.
\]
The Hitchin fibration $h_{\boldsymbol{\theta}}: M_D \to \mathcal{A}_{\boldsymbol{\theta}}$ is given by taking the characteristic coefficients of a polynomial matrix $A(z) \in \mathcal{P}_{s-2}(\mathfrak{g})$.

\subsection{Gaudin model}
We now demonstrate explicitly how this framework recovers the well-known classical Gaudin model also known as the \emph{Garnier integrable system} \cite{FFR}. The key is to make a specific choice for the parahoric data $\boldsymbol{\theta}$ that corresponds to the simplest pole structure. We choose the parahoric subgroup $\mathcal{G}_{\theta_j}$ at each point $x_j \in D$ to be an \emph{Iwahori subgroup}.

With this choice, a logarithmic Higgs field $\Phi \in H^0(\mathbb{P}^1, \mathfrak{g} \otimes K_{\mathbb{P}^1}(D))$ is realized as a $\mathfrak{g}$-valued meromorphic 1-form on $\mathbb{P}^1$ that is holomorphic away from $D$ and has at worst simple (first-order) poles at each $x_j$. Such a 1-form can be written in terms of its residues. Let $X_j = \operatorname{Res}_{z=x_j} \Phi \in \mathfrak{g}$. The 1-form can then be written in the coordinate $z$ as:
\begin{equation}
    \Phi(z) = \left( \sum_{j=1}^s \frac{X_j}{z-x_j} \right) dz.
\end{equation}
The coefficient of $dz$ is immediately recognizable as the \emph{Gaudin Lax operator}:
\begin{equation}
    L(z) = \sum_{j=1}^s \frac{X_j}{z-x_j}.
\end{equation}
The residues $X_j$ of the Higgs field are identified with the dynamical spin variables of the Gaudin model. The condition that $\Phi(z)$ is a regular 1-form at $z=\infty$ implies the residue sum rule $\sum_j X_j = 0$, which defines the total momentum constraint of the Gaudin model.

The commuting Hamiltonians arise from the Hitchin map. Taking the simplest quadratic invariant (using a non-degenerate bilinear form $\langle\cdot,\cdot\rangle$ on $\mathfrak{g}$), we construct the generating function from the Lax operator:
\begin{equation}
    H(z) := \frac{1}{2} \langle L(z), L(z) \rangle = \frac{1}{2} \sum_{j,k=1}^{s} \frac{\langle X_j, X_k \rangle}{(z-x_j)(z-x_k)}.
\end{equation}
This is precisely the generating function for the Gaudin Hamiltonians. The general theorem on complete integrability for the logahoric Hitchin integrable system thus provides a deep geometric proof for the integrability of the classical Gaudin model.

The two descriptions of the phase space---as polynomials $A(z) \in \mathcal{P}_{s-2}(\mathfrak{g})$ and as rational functions $L(z)$---are algebraically equivalent. The isomorphism is given by clearing denominators.

Let $P(z) = \prod_{k=1}^{s} (z-x_k)$ be the scalar polynomial whose roots are the poles of the Lax operator. Define the $\mathfrak{g}$-valued polynomial $A(z)$ by:
\begin{equation}
A(z) := P(z) L(z) = \left( \prod_{k=1}^{s} (z-x_k) \right) \left( \sum_{j=1}^{s} \frac{X_j}{z - x_j} \right) = \sum_{j=1}^{s} X_j \prod_{k \neq j} (z-x_k).
\end{equation}
Each term in the sum is a polynomial of degree $s-1$. However, due to the residue sum rule $\sum_{j=1}^s X_j = 0$, the coefficient of the $z^{s-1}$ term vanishes:
\[
\text{coeff}(z^{s-1}) \text{ in } A(z) = \sum_{j=1}^{s} X_j = 0.
\]
Therefore, the degree of the polynomial $A(z)$ is at most $s-2$, and it lies in $\mathcal{P}_{s-2}(\mathfrak{g})$. This establishes that the space of polynomials $A(z)$ in the explicit realization is canonically isomorphic to the space of Gaudin Lax operators $L(z)$ with the total momentum constraint. The Hitchin Hamiltonians can be generated from the invariants of either $A(z)$ or $L(z)$, as they are algebraically equivalent.

\subsection{KP elliptic solitons}

We now specialize to the case when the base curve $X$ is an elliptic curve $\Sigma$. The triviality of the canonical bundle, $K_\Sigma \cong \mathcal{O}_\Sigma$, simplifies the setup and reveals deep connections to well-known algebraically completely integrable Hamiltonian systems. A logarithmic Higgs field $\varphi \in H^0(X, \mathcal{E}(\mathfrak{g}) \otimes K_X(D))$ becomes a section of $\mathcal{E}(\mathfrak{g}) \otimes \mathcal{O}_\Sigma(D)$, i.e., a meromorphic section of the adjoint bundle with poles prescribed by the divisor $D$. We will now show how the framework of the logahoric Hitchin integrable system recovers both the KP hierarchy and the appropriate symplectic leaf.

We first demonstrate that the space of elliptic solutions to the Kadomtsev--Petviashvili (KP) equation arises as a specific symplectic leaf of the moduli space of logahoric $\mathcal{G}_{\boldsymbol \theta}$-Higgs torsors, following the treatment of Markman \cite{Markman} and Treibich--Verdier \cite{T-V}.

Let the structure group be $G = \mathrm{SL}(r, \mathbb{C})$. We consider logahoric Higgs torsors of rank $r$ and degree $0$. The underlying vector bundles are semistable, and the moduli space of the torsors themselves, $\mathcal{U}_\Sigma(r,0)$, is isomorphic to the symmetric product $\mathrm{Sym}^r \Sigma$. Let the divisor be a single point $D = \{q\}$ and choose the parahoric and level structure at $q$ such that the Levi factor of the level group is $L_q \cong \mathrm{SL}(r, \mathbb{C})$. The Lie algebra of the level group is thus $\mathfrak{g}_D \cong \mathfrak{sl}(r)$, and its dual is $\mathfrak{g}_D^* \cong \mathfrak{sl}(r)^*$.

The space of KP elliptic solitons is known to correspond to a specific coadjoint orbit in $\mathfrak{sl}(r)^*$. Let $\mathrm{Orb}(KP)$ be the coadjoint orbit of the element in $\mathfrak{sl}(r)^*$ corresponding (via the Killing form) to the matrix $\mathrm{diag}(-1, -1, \dots, -1, r-1)$.

We define the KP symplectic leaf, $M(KP, r)$, to be the subvariety of our moduli space $\mathcal{M}_{LH}(\Sigma, \mathcal{G}_{\boldsymbol{\theta}})$ consisting of triples $[(\mathcal{E}, \varphi, \eta)]$ where the coresidue of the Higgs field lies in this orbit.
\begin{equation*}
    M(KP,r) := \left\{ [(\mathcal{E}, \varphi, \eta)] \in \mathcal{M}_{LH}(\Sigma, \mathcal{G}_{\boldsymbol{\theta}}) \mid \mu([(\mathcal{E},\eta)],\varphi) = \operatorname{CoRes}_q(\varphi) \in \mathrm{Orb}(KP) \right\}.
\end{equation*}
This is a symplectic leaf of the Poisson manifold $\mathcal{M}_{LH}(\Sigma, \mathcal{G}_{\boldsymbol{\theta}})$. The results of Treibich and Verdier \cite{T-V} provide a bijection between the space of solutions $\mathrm{Sol}(\Lambda, r)$ and an open subset of this leaf. As a direct consequence of Theorem \ref{thm:lagrangian_LH}, we have:

\begin{cor}
The space $\mathrm{Sol}(\Lambda,r)$ of $\Lambda$-periodic KP elliptic solitons of order $r$ embeds as a Zariski open subset of the symplectic leaf $M(KP,r)$. It is thereby endowed with a canonical algebraically completely integrable Hamiltonian system structure inherited from the parahoric Hitchin fibration on $\mathcal{M}_{LH}(\Sigma, \mathcal{G}_{\boldsymbol{\theta}})$.
\end{cor}

\subsection{Elliptic Calogero--Moser system}

The elliptic Calogero--Moser system can be realized intrinsically as a $G$-Hitchin system, as demonstrated by Hurtubise and Markman in \cite{HM}. Their construction bypasses the difficulties of using semi-simple groups by defining a bespoke, non-semisimple structure group that naturally accommodates the geometry of the Calogero--Moser phase space.

\begin{thm}[\cite{HM}, Sections 4 \& 5]
\label{thm:HM_system}
Let $\Sigma$ be an elliptic curve with a marked point $p_0$, and let $R$ be a root system with torus $H$ and Weyl group $W$. Let $G_{HM}$ be the non-semisimple group whose connected component is $(\bigoplus_{\alpha \in R} \mathbb{C}_{\alpha}) \rtimes H$. The elliptic Calogero--Moser system for the root system $R$ is realized as a Hitchin system whose phase space is the symplectic reduction of the cotangent bundle $T^*\mathcal{M}_{G_{HM}}$ of the moduli space of $G_{HM}$-bundles on $\Sigma$ framed at $p_0$. The reduction is taken with respect to the action of $G_{HM}$, where the moment map is the residue of the Higgs field at $p_0$, and is performed at a generic, $W$-invariant coadjoint orbit element $C \in (\bigoplus_{\alpha \in R} \mathbb{C}_{\alpha})^*$. The resulting reduced phase space is symplectically isomorphic to $(\mathfrak{h} \times \mathfrak{h}^*) / W_{\text{aff}}$, where the coordinates $x \in \mathfrak{h}$ parameterize the underlying $H$-bundle, and the points $p \in \mathfrak{h}^*$ correspond to the constant part of the Higgs field in the $\mathfrak{h}^*$ direction. The Hamiltonian of the system, $H_{CM}(x, p) = p \cdot p + \sum_{\alpha \in R} m_{|\alpha|} \wp(\alpha(x))$, arises from the residue pairing of the canonical quadratic invariant on the Lie algebra $\mathfrak{g}_{HM}$ with the Weierstrass zeta function.
\end{thm}

The Hurtubise--Markman construction can be precisely situated within the general framework of logahoric Higgs torsors. The elliptic Calogero--Moser system corresponds to a single symplectic leaf within the moduli space $\mathcal{M}_{LH}(\Sigma, \mathcal{G}_{\boldsymbol{\theta}})$ for a specific choice of data. We set the base curve to be the elliptic curve $X=\Sigma$ with divisor $D=\{p_0\}$, and the structure group to be $G = G_{HM}$, as in the statement of Theorem \ref{thm:HM_system}. The parahoric data at $p_0$ is taken to be trivial ($\theta_{p_0}=0$), corresponding to the standard parahoric group $G_{HM}(\mathbb{C}[[z]])$. This choice makes the level group $G_D \cong G_{HM} / Z(G_{HM})$. The framing of the bundle at $p_0$ used in \cite{HM} is a specific realization of a $D$-level structure $\eta$. The crucial identification is that the symplectic reduction in \cite{HM} at a fixed coadjoint orbit element $C \in (\bigoplus_{\alpha \in R} \mathbb{C}_{\alpha})^*$ is equivalent to selecting the symplectic leaf $\mu^{-1}(\mathrm{Ad}^*_{G_{HM}}(C))$ in our framework, where $\mu$ is the moment map.
Consider the logahoric Hitchin integrable system for the group $G=\mathrm{SL}(r,\mathbb{C})$ over the elliptic curve $X=\Sigma$ with divisor $D=\{q\}$. We choose the trivial parahoric data at $q$, so the level group is $G_D = \mathrm{SL}(r,\mathbb{C})/Z(\mathrm{SL}(r,\mathbb{C}))$. The moment map $\mu_{\mathrm{SL}(r)}$ takes values in $\mathfrak{g}_D^* \cong \mathfrak{sl}(r)^*$.

As described in \cite{Markman} and \cite{T-V}, the space of elliptic solutions to the KP equation, $\mathrm{Sol}(\Lambda, r)$, corresponds to the symplectic leaf defined by the coadjoint orbit of the element in $\mathfrak{sl}(r)^*$ corresponding to the matrix $\mathrm{diag}(-1, -1, \dots, -1, r-1)$. We denote this orbit by $\mathrm{Orb}(KP)$. The KP symplectic leaf is thus:
\[ M(KP, r) := \mu_{\mathrm{SL}(r)}^{-1}(\mathrm{Orb}(KP)) \subset \mathcal{M}_{LH}(\Sigma, \mathcal{G}_{\mathrm{SL}(r), \boldsymbol{\theta}}). \]
The connection is then established by the embedding theorem from Hurtubise and Markman.

\begin{thm}[\cite{HM}, Section 6]
For the root system $R = A_{r-1}$, there exists an equivariant embedding of the Hurtubise--Markman Hitchin system for $G_{HM}$ into the standard Hitchin system for $\mathrm{GL}(r, \mathbb{C})$. This embedding maps the phase space of the $A_{r-1}$ Calogero--Moser system precisely onto the symplectic leaf that describes the elliptic solutions of the KP hierarchy.
\end{thm}

\begin{cor}
The symplectic leaf $M(KP, r) \subset \mathcal{M}_{LH}(\Sigma, \mathcal{G}_{\mathrm{SL}(r), \boldsymbol{\theta}})$ corresponding to the KP hierarchy is symplectically isomorphic to the symplectic leaf $$\mu_{G_{HM}}^{-1}(\mathrm{Ad}^*_{G_{HM}}(C)) \subset \mathcal{M}_{LH}(\Sigma, \mathcal{G}_{G_{HM}}, \boldsymbol{\theta'})$$ corresponding to the $A_{r-1}$ Calogero--Moser system.
\end{cor}

Therefore, the general framework of logahoric Higgs torsors provides a unified geometric perspective. It demonstrates that the phase space for elliptic KP solutions is a specific symplectic leaf of the $\mathrm{SL}(r, \mathbb{C})$-Hitchin system, and that this very same leaf can be described intrinsically for the type $A_{r-1}$ root system via the Hurtubise--Markman construction, yielding the corresponding Calogero--Moser system.

\vspace{5mm}

\textbf{Acknowledgements}.
We warmly thank Pengfei Huang and Hao Sun for many fruitful discussions and remarks. We are also very grateful to the following Research Institutes for their support and hospitality, where part of this work was completed: The Kavli Institute for the Physics and Mathematics of the Universe at the University of Tokyo, the Brin Mathematics Research Center at the University of Maryland and the Tianyuan Mathematics Research Center. Georgios Kydonakis was supported by the Scientific Committee of the University
of Patras through the program ``Medicos''. Lutian Zhao was supported by JSPS KAKENHI Grant Number JP25K17226.

\bigskip

\noindent\small{\textsc{Department of Mathematics, University of Patras}\\
  University Campus, Patras 26504, Greece}\\
\emph{E-mail address}:  \texttt{gkydonakis@math.upatras.gr}

	\bigskip
	
	\noindent\small{\textsc{Kavli Institute for the Physics and Mathematics of the Universe, The University of Tokyo}\\
		5-1-5 Kashiwanoha, Kashiwa, Chiba, 277-8583, Japan}\\
	\emph{E-mail address}: \texttt{lutian.zhao@ipmu.jp}

\end{document}